\documentclass[11pt]{article}
\usepackage{amsmath}
\usepackage{amsfonts}
\usepackage{color}
\usepackage{multirow}
\usepackage{latexsym}
\usepackage{epsfig}
\usepackage{float}
\usepackage{array}
\usepackage{amssymb,amsthm}

\newcommand*{\E}{\mathbb{E}}

\usepackage{hyperref}
\usepackage{url}

\usepackage{amsmath}
\usepackage{amsfonts}
\usepackage{latexsym}
\usepackage{amssymb}

\newtheorem{thm}{Theorem}[section]
\newtheorem{theorem}[thm]{Theorem}

\newtheorem{lemma}[thm]{Lemma}

\newtheorem{proposition}[thm]{Proposition}

\newtheorem{corollary}[thm]{Corollary}

\newtheorem{assumption}[thm]{Assumption}

\newtheorem{remark}[thm]{Remark}

\newcommand{\beq}{\begin{equation}}
	\newcommand{\eeq}{\end{equation}}
\newcommand{\beqa}{\begin{eqnarray}}
	\newcommand{\eeqa}{\end{eqnarray}}
\newcommand{\beqas}{\begin{eqnarray*}}
	\newcommand{\eeqas}{\end{eqnarray*}}
\newcommand{\bi}{\begin{itemize}}
	\newcommand{\ei}{\end{itemize}}
\newcommand{\bas}{\begin{align*}}
	\newcommand{\eas}{\end{align*}}

\newcommand{\vgap}{\vspace{.1in}}
\newcommand{\lessgap}{\vspace{-.1in}}

\newcommand{\nn}{\nonumber}

\newcommand{\R}{\mathbb{R}}

\newcommand{\lam}{{\lambda}}

\newcommand{\inner}[2]{\langle #1,#2\rangle}

\newcommand{\argmin}{\mathrm{argmin}\,}

\newcommand{\dom}{\mathrm{dom}\,}

\newcommand{\vg}[2]{{\color{red} #1}{\color{blue}\  #2}}

\begin{document}
	\title{A single cut proximal bundle method for \\stochastic convex composite optimization}
	% \date{July 18, 2022 (first revision: November 3, 2022; second revision: April 30, 2023)}
	\date{July 18, 2022 \\ 
		1st revision: November 3, 2022 \\
		2nd revision: April 30, 2023 \\ 3rd revision: October 20, 2023}
	% 	\author{
		% 		Jiaming Liang \thanks{School of Industrial and Systems
			% 			Engineering, Georgia Institute of
			% 			Technology, Atlanta, GA, 30332-0205.
			% 			(email: {\tt jiaming.liang@gatech.edu} and {\tt renato.monteiro@isye.gatech.edu}). This work
			% 			was partially supported by ONR Grant N00014-18-1-2077 and AFORS Grant FA9550-22-1-0088.}\qquad 
		% 		Renato D.C. Monteiro \footnotemark[1]\qquad }
	\author{
		Jiaming Liang \thanks{Goergen Institute for Data Science and Department of Computer Science, University of Rochester, Rochester, NY 14620 (email: {\tt jiaming.liang@rochester.edu}).}\qquad
		% 		School of Data Science, The Chinese University of Hong Kong, Shenzhen, 518172, China
		% 		(email: {\tt liangjiaming@cuhk.edu.cn})
		Vincent Guigues \thanks{School of Applied Mathematics FGV/EMAp, 22250-900 Rio de Janeiro, Brazil. (email: {\tt vincent.guigues@fgv.br}).} \qquad
		Renato D.C. Monteiro \thanks{School of Industrial and Systems
			Engineering, Georgia Institute of
			Technology, Atlanta, GA 30332.
			(email: {\tt renato.monteiro@isye.gatech.edu}). This work
			was partially supported by AFOSR Grant FA9550-22-1-0088.}
	}
	\maketitle
	
	\begin{abstract}
		This paper considers
		optimization problems
		where the objective is
		the sum of a function given by an expectation and a closed convex function,
		and proposes
		stochastic composite proximal
		bundle (SCPB) methods 
		%with optimal complexity 
		for solving it.
		Complexity guarantees are established for them without requiring knowledge of parameters associated with the problem
		instance. Moreover, it is shown that they have optimal complexity when
		these problem parameters are known.
		% on the objective functions
		% The method does not require estimation
		% of parameters involved in the assumptions
		% on the objective functions.
		To the best of our knowledge, this is the
		first proximal bundle method for stochastic
		programming able to deal with
		continuous distributions.
		Finally, we present 
		computational results showing that SCPB substantially
		outperforms the robust stochastic approximation (RSA) method in all instances considered.\\

		% 	\if{
			% 		This paper presents a proximal bundle variant, namely, the RPB method, for solving convex nonsmooth composite optimization problems.
			% 		Like other proximal bundle variants, RPB solves a sequence of prox bundle subproblems whose objective functions are regularized composite cutting-plane models.
			% 		Moreover, RPB uses a novel condition to decide
			% 		whether to perform a serious or null iteration
			% 		which does not necessarily yield a function value decrease.
			% 		Optimal (possibly up to a logarithmic term) iteration complexity bounds for RPB are established for a large range of prox stepsizes, both in the convex and strongly convex settings.
			% 		To the best of our knowledge, this is the first time that a proximal
			% 		bundle variant is shown to be optimal for a large range of prox stepsizes.
			% 		Finally, iteration complexity results for RPB to obtain iterates satisfying practical termination criteria, rather than near optimal solutions,
			% 		are also derived.}\fi
		% 		\\
		
		% 	{\bf Keywords:} Nonsmooth optimization $\cdot$ iteration complexity $\cdot$ Proximal bundle method $\cdot$ Optimal complexity bound
		\\
		
		% 	{\bf Mathematics Subject Classification (2010)} 
		% 	49M37 $\cdot$ 65K05 $\cdot$ 68Q25 $\cdot$ 90C25 $\cdot$ 90C30 $\cdot$ 90C60
	\end{abstract}
	
	\section{Introduction}\label{sec:intro}
	
	The main goal of this paper is to propose and study the complexity of some stochastic composite proximal bundle (SCPB) variants to solve the
	stochastic	
	convex composite optimization (SCCO) problem
	\begin{equation}\label{eq:ProbIntro}
		\phi_{*}:=\min \left\{\phi(x):=f(x)+h(x): x \in \R^n\right\}
	\end{equation}
	where 
	\begin{equation}\label{pbint2}
		f(x)=\mathbb{E}_{\xi}[F(x,\xi)].
	\end{equation}
	
	% 	-----
	
	% 	\usepackage{}Exact reformulations
	% 	\[
	% 	f(x) = \sum_{i \in I}  p_i F(x,\xi_i)
	% 	\]
	
	% 	SAA type methods
	% 	\[
	% 	f(x) \approx  \frac{1}N \sum_i  F(x,\xi_i)
	% 	\]
	
	% 	\[
	% 	\left | f(x) -  \frac{1}N \sum_i  F(x,\xi_i) \right| \le \frac{C}{\sqrt{N}}
	% 	\]
	
	% 	-----
	
	We assume the following conditions hold:
	i) $f, h: \R^{n} \rightarrow \R\cup \{ +\infty \} $ are proper closed convex functions
	such that
	$ \dom h \subseteq \dom f $;
	% 	ii) $h$ is $M_h$-Lipschitz continuous on $\dom h$ for some $ M_h\in \R_+ $;
	ii) a stochastic first-order oracle, which
	for every $x \in \dom h$ and almost every random vector $\xi$ returns
	$s(x,\xi)$ such that $\mathbb{E}[s(x,\xi)] \in \partial f(x)$, is available;
	and iii) for every $x \in \dom h$, $\E[\|s(x,\xi)\|^2] \le \bar M^2$ for some $\bar M\in \R_+$.
	
	{\bf Literature Review.}
	Proximal bundle methods for solving the deterministic version of \eqref{eq:ProbIntro},
	i.e., where an oracle that outputs $f(x)$
	for any $x$ is available, have been
	proposed in \cite{lemarechal1975extension,lemarechal1978nonsmooth,mifflin1982modification,wolfe1975method}.
	Moreover, convergence (but not complexity) analyses of proximal bundle methods have been developed for example in \cite{frangioni2002generalized,de2014convex,ruszczynski2011nonlinear,van2017probabilistic},
	and their 
	iteration complexities have been derived for example in \cite{astorino2013nonmonotone,diaz2021optimal,du2017rate,kiwiel2000efficiency,liang2020proximal,liang2021unified}.
	
	We now discuss methods for solving (the stochastic version of) problem \eqref{eq:ProbIntro}.
	Methods for solving
	\eqref{eq:ProbIntro} when $f$ can
	be computed exactly (e.g., $\xi$ is a discrete random vector with small support)
	have been discussed for example in \cite{birge-louv-book,birgelouv} 
	and are  usually based on solving a deterministic (but large-scale) reformulation of \eqref{eq:ProbIntro}, 
	using decomposition (such as the L-shaped method
	\cite{vwets}) possibly combined with regularization as in
	\cite{guilejtekregsddp,hise96}.
	
	Solution methods  for problem \eqref{eq:ProbIntro}
	in which $\xi$ has a continuous distribution
	are basically based on one of the following 
	three ideas:
	i) a single (usually expensive) approximation of \eqref{eq:ProbIntro}
	where $f$ is approximated by a Monte Carlo average
	$H_N(x) :=\sum_{i=1}^N F(\cdot,\xi_i)/N$ for a large i.i.d.\ sample  $(\xi_1,\ldots,\xi_N)$ of $\xi$ is constructed
	at the beginning of the method and is then
	solved to yield an approximate solution of
	\eqref{eq:ProbIntro} (SAA-type methods);
	see for instance \cite{ioudnemgui15,kingrock1993,kleywegt2002sample,shap1991,verweij2003sample} and also Chapter 5 of \cite{shadenrbook} for their complexity analysis;
	% 	and
	ii) simple approximations
	of \eqref{eq:ProbIntro} are constructed
	at every iteration based on
	a small (usually a single)
	sample and their solutions are
	used to obtain 
	an approximation solution of
	\eqref{eq:ProbIntro} (SA-type methods);
	SA-type methods have been originally proposed
	in \cite{monroe51} and further extended  in
	\cite{guiguesmstep17,guiguesinexactsmd,nemjudlannem09,nemyud78,nesterov2009primal,polyak90,polyakjud92};
	and 
	iii) hybrid type methods which sit in between SAA and SA-type ones in that they use
	partial Monte Carlo averages $H_k(\cdot)$ (and their expensive subgradients) for increasing 
	iteration indices $k$  \cite{hise96}.

	\par {\textbf{Contributions.}} Although the cutting plane methodology can be used
	in the context of SAA methods to solve  single  approximations of \eqref{eq:ProbIntro} generated at their outset, such methodology has not been used in the
	context of SA-type methods.
	This paper partially addresses this issue
	by developing regularized aggregated cutting plane methods for solving \eqref{eq:ProbIntro} where some of the most recent (linear approximation) cuts (in expectation) are combined,
	i.e., a suitable convex combination of them
	is chosen,
	so that
	a single aggregated cut (in expectation) is obtained.
	Two SCPB variants
	based on the aforementioned aggregated one-cut
	scheme are
	proposed which can be viewed as
	natural extensions of the one-cut variant
	developed in \cite{liang2021unified} (based on the analysis of \cite{liang2020proximal}) for solving the
	deterministic version of \eqref{eq:ProbIntro}.
	More specifically, 
	at every iteration,
	these SCPB variants solve the prox bundle subproblem
	\begin{equation}\label{eq:x-pre}
		x = \underset{u\in \R^n}\argmin \left \{\Gamma (u) + \frac{1}{2\lam} \|u-x^c\|^2 \right\}
	\end{equation}
	where $\lam>0$ is the prox stepsize, $x^c$ is the current prox-center,
	and
	$\Gamma$ is the current
	bundle function in expectation, i.e., it satisfies
	$\E[\Gamma(\cdot)] \le \phi(\cdot)$.
	The prox-center remains the same for several consecutive iterations which are referred to as a cycle.
	% As in the proximal bundle method of \cite{liang2021unified},
	% it also performs two types of iterations, i.e., serious or null ones.
	In the beginning of a cycle, the prox-center is
	updated to $x^c \leftarrow x $ and
	the bundle function $\Gamma$
	is chosen to be the composite linear approximation
	$F(x,\xi)+\inner{s(x,\xi)}{\cdot-x} + h(\cdot)$
	of the function $F(\cdot,\xi) + h(\cdot)$ at $x$ for some new
	independent sample
	$\xi$.
	% (referred to as the composite linear approximation .
	For other iterations of the cycle,
	% In a null iteration, 
	the prox-center remains the same
	but $\Gamma$ is set to be a convex combination of the previous bundle function and the most recent composite linear approximation as constructed above.
	It is then shown that
	both SCPB variants obtain a stochastic iterate $y \in \R^n$ (determined by some of the above generated $x$'s) such that
	$
	\E[\phi(y)] - \phi_* \le  \varepsilon$,
	where $\phi_*$ is as in \eqref{eq:ProbIntro}, in
	${\cal O}(\varepsilon^{-2})$ iterations/resolvent evaluations.
	To our knowledge, these are the
	first SA-type SCPB  methods for solving SCCO problems where $\xi$ can have either a discrete or continuous distribution.
	Finally, it is shown that
	the robust stochastic approximation (RSA) method of \cite{nemjudlannem09} is a special case of SCPB with
	a  relatively small prox stepsize.

	\if{
		
		{\bf Vincent: please on the part below including the paper organization} An important class of problems of the form \eqref{eq:ProbIntro}-\eqref{pbint2}
		is the class of two-stage stochastic convex programs
		(see for instance \cite{shadenrbook, guiguessiopt2016, guigues2016isddp, lecphilgirar12}) given by
		\begin{equation}\label{def2stageconv1}
			\left\{ 
			\begin{array}{l}
				\min \;f_1(x_1)+
				\mathbb{E}[\mathfrak{Q}(x_1,\xi)]\\
				x_1 \in X_1
			\end{array}
			\right.
		\end{equation}    
		where: i) $X_1$ is convex, $f_1: X_1 \rightarrow \mathbb{R}$ is convex,
		and 
		%  $\mathcal{Q}(x_1)=\mathbb{E}[\mathfrak{Q}(x_1,\xi)]$ for
		\begin{equation}\label{def2stageconv2}
			\mathfrak{Q}(x_1,\xi) =\left\{ 
			\begin{array}{l}
				\min_{x_2} \;f_2(x_1,x_2,\xi)\\
				A_2 x_2 + B_2 x_1 = b_2,\\
				g_2(x_1,x_2,\xi) \leq 0,\\
				x_2 \in X_2;
			\end{array}
			\right.
		\end{equation}
		and ii)
		$X_2$ is convex, and
		$f_2(x_1,\cdot,\xi)$
		and all the components of $g_2(x_1,\cdot,\xi)$
		are convex
		for almost every $\xi$ and all $x_1 \in X_1$.
		In this problem, $\xi$ contains in particular
		the random components in $A_2$, $B_2$ and  $b_2$.
		Clearly, problem \eqref{def2stageconv1}-\eqref{def2stageconv2}
		is a special case of \eqref{eq:ProbIntro}-\eqref{pbint2}
		in which $h$ is the indicator function
		$\delta_{X_1}(\cdot)$
		of set $X_1$, $x=x_1$, and
		$F(x,\xi)=f_1(x)+\mathfrak{Q}(x,\xi)$.
		
	}\fi
	
	{\bf Organization of the paper.} Subsection~\ref{subsec:DefNot}  presents basic definitions and notation used throughout the paper.
	Section~\ref{SCPBf} formally describes the assumptions on the SCCO problem \eqref{eq:ProbIntro},
	presents the SCPB scheme, and two cycles rules for determining the length of the cycles in SCPB.
	Section~\ref{sec:SCPB1} presents various convergence rate bounds for the SCPB variant based on the first cycle rule and discusses the relationship between SCPB and RSA.
	Section~\ref{sec:SCPB2} provides convergence rate bounds for the SCPB variant based on the second cycle rule.
	Section~\ref{sec:proofth2} collects proofs of the main results in Sections~\ref{sec:SCPB1} and \ref{sec:SCPB2}.
	Section~\ref{sec:num} reports the numerical experiments.
	Finally, Section~\ref{sec:remarks} presents some concluding remarks and possible extensions.

	% Two sets of sequences $\{j_k\}$
	% are considered (see (B1) and (B2)) for the complexity analysis.
	% The convergence rate bounds of the corresponding
	% SCPB variants are given in
	% Theorem \ref{thm:main1} for variant (B1) and in Theorem \ref{thm:main2} for variant (B2).
	% The proof of Theorem \ref{thm:main1}
	% is given in Section \ref{prooffirstth}
	% while the proof of Theorem \ref{thm:main2}
	% is developed in Section \ref{sec:proofth2}.
	% Finally, numerical experiments are reported
	% in Section \ref{sec:num}.
	% \vg{}{Proofs are collected in the Appendix.}

	\if{
		----------------------
		
		{\bf L-shaped:}
		We want to solve
		$$
		\begin{array}{l}
			\min_{x_1} c_1^T x_1 + \mathcal{Q}(x_1) \\
			A_1 x_1 + B_1 x_0 = b_1
		\end{array}
		$$
		where
		$$
		\mathcal{Q}(x_1)
		=\mathbb{E}[\mathfrak{Q}(x_1,\xi)]
		= \sum_{i=1}^N p_i \mathfrak{Q}(x_1,\xi_i)
		$$
		and
		$$
		\begin{array}{l}
			\mathfrak{Q}(x_1,\xi)=
			\left\{ 
			\begin{array}{l}
				\min_{x_2} c_2^T x_2 \\
				A_2 x_2 + B_2 x_1 = b_2
			\end{array}
			\right.
		\end{array}
		=
		\left\{ 
		\begin{array}{l}
			\max_{\pi}  \pi^T(b_2-B_2 x_1) \\
			\pi \in D
		\end{array}
		\right.
		$$
		$\xi=(c_2,A_2,B_2,b_2)$.
		
		Start from $x^0$.
		At iteration $k+1$, compute a linearization
		$\ell_k$
		of $\mathcal{Q}$ at $x^k$
		and compute $x^{k+1}$ solution
		of 
		$$
		\begin{array}{l}
			\min_{x_1} c_1^T x_1 + \max_{i \leq k} \ell_i(x_1) \\
			A_1 x_1 + B_1 x_0 = b_1.
		\end{array}
		$$
		The linearization $\ell_k$ is given
		by
		$$
		\sum_{i=1}^N p_i \pi_{k, i}^T(b_{2.i}-B_{2,i}x_1).
		$$
		
		-------------------------------
		
		Many solution methods have been proposed
		to solve these problems
		when
		the distribution of $\xi$ is discrete
		such as the L-shaped method  \cite{birge-louv-book}
		or multicut decomposition methods \cite{birgelouv}.
		Typically problems having a finite number of
		second stage scenarios are obtained by applying
		the Sample Average Approximation (SAA)
		scheme to the original
		two-stage continuous problem (see e.g., \
		\cite{shadenrbook, verweij2003sample, kleywegt2002sample, shap1991, kingrock1993, ioudnemgui15}).
		For two-stage stochastic problems with continuous distributions, efficient solution methods
		include the Stochastic
		Approximation (SA) \cite{monroe51}, the Robust Stochastic Approximation (RSA) \cite{polyak90}, \cite{polyakjud92}, Stochastic 
		Mirror Descent (SMD) \cite{nemjudlannem09},
		Multistep Stochastic Mirror Descent \cite{guiguesmstep17}, and
		Inexact Stochastic Mirror Descent (ISMD) \cite{guiguesinexactsmd}.
		An extension of the cutting plane method
		for two-stage stochastic linear programs with
		continuous distributions, called
		Stochastic Decomposition (SD), was proposed in \cite{hise96}.
		However, to the best of our knowledge, no bundle type regularization 
		of such cutting plane (referred to here as proximal bundle) methods has so far  been proposed
		for two-stage stochastic programs with
		continuous distributions. In this paper,
		we propose and analyze the complexity of
		a family of proximal bundle methods for a problem of form
		\eqref{eq:ProbIntro}-\eqref{pbint2} that apply in particular to two-stage stochastic convex
		programs.

		%     Instead of focusing on a specific proximal bundle variant, the \vg{HCPB}{SCPB} framework is a family of proximal bundle type methods and states the minimal conditions for establishing complexity bounds for all algorithms contained \vg{on}{in} it in a unified manner. 
		%     \vg{In every iteration}{At iteration $j$}, \vg{HCPB}{SCPB} solves the prox bundle subproblem
		% 	\begin{equation}\label{eq:x-pre}
			% 	    \vg{x}{x_{j+1}} = \underset{u\in \R^n}\argmin \left \{\vg{\Gamma (u)}{\Gamma_{j+1}(u)} + \frac{1}{2\lam} \|u-\vg{x_0}{x_{j+1}^c}\|^2 \right\}
			% 	\end{equation}
		% 	where $\lam$ is the prox stepsize, $x^c_{j+1}$ is 
		% 	\vg{thef}{the} prox-center \vg{}{and $\Gamma_{j+1}$ is a 
			% $\mu$-convex stochastic 
			% function such that
			% $\mathbb{E}[\Gamma_{j+1}] \leq \phi$}.
		
		%     {\bf Contributions.}

		%     {\bf Related works.} 
		
		{\bf Organization of the paper.}

		?????????????????????????????????????????
		
	}\fi
	
	\subsection{Basic definitions and notation} \label{subsec:DefNot}
	
	Let $\mathbb{N}_{++}$ denote the set of positive integers.
	The sets of real numbers, non-negative and positive real numbers are denoted by $\R$, $\R_+ $ and $\R_{++}$, respectively. 
	Let $\R^n$ denote the standard $n$-dimensional Euclidean 
	space equipped with  inner product and norm denoted by $\left\langle \cdot,\cdot\right\rangle $
	and $\|\cdot\|$, respectively.

	Let $\psi: \R^n\rightarrow (-\infty,+\infty]$ be given. The effective domain of $\psi$ is denoted by
	$\dom \psi:=\{x \in \R^n: \psi (x) <\infty\}$ and $\psi$ is proper if $\dom \psi \ne \emptyset$.
	For $\varepsilon \ge 0$, the \emph{$\varepsilon$-subdifferential} of $ \psi $ at $z \in \dom \psi$ is denoted by
	$\partial_\varepsilon \psi (z):=\left\{ s \in\R^n: \psi(u)\geq \psi(z)+\left\langle s,u-z\right\rangle -\varepsilon, \forall u\in\R^n\right\}$.
	The subdifferential of $\psi$ at $z \in \dom \psi$, denoted by $\partial \psi (z)$, is by definition the set  $\partial_0 \psi(z)$.
	Moreover, a proper function $\psi: \R^n\rightarrow (-\infty,+\infty]$ is $\mu$-strongly convex for some $\mu \ge 0$ if
	$$
	\psi(\alpha z+(1-\alpha) u)\leq \alpha \psi(z)+(1-\alpha)\psi(u) - \frac{\alpha(1-\alpha) \mu}{2}\|z-u\|^2
	$$
	for every $z, u \in \dom \psi$ and $\alpha \in [0,1]$. Note that we say $\psi$ is convex when $\mu=0$.
	% The set of all proper lower semicontinuous convex functions $\psi:\R^n\rightarrow (-\infty,+\infty]$  is denoted by $\bConv{n}$.
	We use the notation
	$\xi_{[t]}=(\xi_0,\xi_1,\ldots,\xi_t)$ for the history of
	the sampled observations
	of $\xi$ up to iteration $t$.
	Define $\ln_0^+(\cdot) := \max\{0,\ln(\cdot)\}$.
	Define the diameter of a set $X$ to be $D_X  :=  \sup \{  \| x- x' \| : x, x' \in \dom X \}$.

	\section{Assumptions and two SCPB variants}\label{SCPBf}
	
	This section presents the assumptions made on problem \eqref{eq:ProbIntro} and
	states two SCPB variants for solving it.

	\subsection{Assumptions}
	
	Let $\Xi$ denote the
	support of random 
	vector $\xi$ and assume that the following conditions on \eqref{eq:ProbIntro} are assumed to hold:
	\begin{assumption}\label{assump:basic}
		\begin{itemize}
			\item[(A1)]
			$f$ and $ h$ are proper closed convex functions satisfying
			$\dom f \supset \dom h$;		
			\item[(A2)] for almost every $\xi \in \Xi$,
			a functional oracle $F(\cdot,\xi) :\dom h \to \R$ and
			a stochastic subgradient
			oracle $s(\cdot,\xi):\dom h \to \R^n$ satisfying
			\[
			f(x) = \E[F(x,\xi)], \quad f'(x) := \E[s(x,\xi)] \in \partial f(x)
			\]
			for every $x \in \dom h$ are available;
			\item[(A3)]
			$\bar M := \sup \{ \E[\|s(x,\xi)\|^2]^{1/2} : x \in \dom h \} < \infty$;
			\item[(A4)]
			the set of optimal solutions $X^*$ of
			\eqref{eq:ProbIntro}-\eqref{pbint2}	is nonempty.
		\end{itemize}
	\end{assumption}

	% 	In the next subsection, we will describe two variants of SCPB.
	% 	While the first one requires conditions (A1)-(A4),
	% 	the second one requires all of them.
	% 	Note that we will discuss two variants of the proposed framework and  condition (A5) is only assumed in the second variant.
	
	We now make some observations about the above conditions.
	First, as in \cite{nemjudlannem09},
	condition (A2) does not require $F(\cdot,\xi)$ to be
	convex.
	Second,
	condition (A3) implies that
	\begin{equation}\label{ineq:fp}
		\|f'(x)\| = \|\E[s(x,\xi)]\| \le \E[\|s(x,\xi)\|] \le \left(\E[\|s(x,\xi)\|^2]\right)^{1/2} \le \bar M \quad \forall x \in \dom h.
	\end{equation}
	Third, defining for every $\xi \in \Xi$ and $x \in \dom h$,
	\begin{equation}\label{def:Phi}
		\Phi(\cdot,\xi)=F(\cdot,\xi)+h(\cdot), \quad 
		\ell(\cdot;x,\xi)= f(x)+\inner{s(x,\xi)}{\cdot-x} + h(\cdot), 
	\end{equation}
	it follows from (A2), the second identity in \eqref{def:Phi},
	and the convexity of $f$ by (A1), that
	\begin{equation}\label{eq:exp0}
		\E[\Phi(\cdot,\xi)] = \phi(\cdot)\ge f(x) + \inner{f'(x)}{\cdot-x} + h(\cdot) =
		\E[\ell(\cdot;x,\xi)]
	\end{equation}
	where $\phi(\cdot)$ is as in \eqref{eq:ProbIntro}.
	Hence, $\ell(\cdot;x,\xi)$ is
	a stochastic composite linear approximation of
	$\phi(\cdot)$
	in the sense that its expectation is a true composite linear approximation of $\phi(\cdot)$. (The terminology ``composite" refers to the function $h$ which is included in the approximation $\ell(\cdot;x,\xi)$ as is.)
	% The convergence rate result, as well as other technical
	% lemmas in the paper, use the following notation: 
	% We define the diameter of $ \dom h $ is $D_h  :=  \sup \{  \| z-\bar z \| : z, \bar z \in \dom h \}$.
	
	\if{Third,
		for the special case
		of \eqref{def2stageconv1}-\eqref{def2stageconv2},
		we refer to Lemma 2.1 of \cite{guiguessiopt2016} for a
		characterization of
		$\partial F(x,\xi)$ and how to compute
		a stochastic subgradient $s(x,\xi) \in 	\partial F(x,\xi)$.
	}\fi

	\subsection{Description of the two SCPB variants}

	Before describing the two SCPB variants, we motivate them by interpreting them as
	inexact implementations of the (theoretical) proximal point method for solving \eqref{eq:ProbIntro}.
	
	Their $k$-th cycle of iterations performs a finite number of iterations
	to solve the prox subproblem
	\[
	\underset{u\in \R^n}\min \left \{\phi (u) + \frac{1}{2\lam} \|u-\hat x_{k-1} \|^2 \right\}
	\]
	where $\hat x_{k-1}$ denotes the prox-center
	during the cycle. Each iteration within the cycle solves a subproblem of the form
	\begin{equation}\label{eq:update}
		x=\underset{u\in \R^n}\argmin \left \{{\cal A} (u) + h(u) + \frac{1}{2\lam} \|u-\hat x_{k-1} \|^2 \right\}
	\end{equation}
	where ${\cal A}(\cdot)$ is an 
	affine bundle for
	$f$ in expectation,
	i.e., an affine function
	such that
	% $\Gamma(\cdot) - h(\cdot)$ is affine and
	$\E[{\cal A}(\cdot)]\le f(\cdot)$. (This type of bundle has been considered in
	the inexact proximal point approach considered in \cite{liang2021unified}
	where it is referred to
	as a one-cut bundle
	for $f$.) The bundle
	${\cal A}^+$ for the next subproblem in the cycle is then taken to be a linear combination of the current bundle ${\cal A}$ and a newly generated 
	stochastic linear
	approximation of $f$ of
	the form
	$F(x,\xi)+\inner{s(x,\xi)}{\cdot-x}$. Moreover,
	the first iteration of
	every cycle starts by
	setting the prox-center to the most recently generated $x$ as in \eqref{eq:update} and the bundle to the most recently generated 
	stochastic linear
	approximation of $f$ at $x$.

	Both  SCPB variants are based on the SCPB scheme described below.
	As stated below, the scheme is not a completely specified algorithm since its step 2
	does not
	describe how to select the index $j_k$. Two rules for doing so, and hence the complete description of the two aforementioned SCPB variants, are then given 
	following the statement of the scheme.

	At every iteration $j \geq 1$, the
	SCPB scheme samples
	an independent realization $\xi_{j-1}$ of
	$\xi$.
	
	\noindent\rule[0.5ex]{1\columnwidth}{1pt}
	
	SCPB
	
	\noindent\rule[0.5ex]{1\columnwidth}{1pt}
	{\bf Input:} Scalars $\lam,\theta>0$, integer $K \ge 1$,
	and initial point $ x_0\in \dom h $.
	\begin{itemize}
		\item [0.] 
		% Let scalars $\lam,\theta>0$, integer $K \ge 1$,
		% and initial point $ x_0\in \dom h $
		% be given, and 
		Set 
		$j=k=1$, $j_0=0$, and 
		\begin{equation}\label{def:tau}
			\tau = \frac{\theta K}{\theta K + 1};
		\end{equation}
		% 	moreover, take a sample $\xi_0$ of
		% 	r.v.\ $\xi$;
		\item [1.] take 
		a sample $\xi_{j-1}$ of r.v.\ $\xi$ independent from the previous samples $\xi_0,\ldots,\xi_{j-2}$
		% 	which is independent of $\{\xi_0,\xi_1,\ldots,\xi_{j-1}\}$
		and compute
		\lessgap
		\begin{align} \label{def:xc}
			x_j^c = \left\{\begin{array}{ll}
				x_{j_{k-1}}, & \text { if } j=j_{k-1}+1 , \\ 
				x_{j-1}^c,  & \text { otherwise},
			\end{array}\right.
		\end{align}
		\lessgap
		\lessgap
		\begin{align} \label{eq:Sj}
			S_j =  \left\{\begin{array}{ll}
				s(x_{j_{k-1}},\xi_{j_{k-1}}), & \text { if } j=j_{k-1}+1 , \\ 
				(1-\tau)
				s(x_{j-1},\xi_{j-1}) +  \tau S_{j-1} ,  & \text { otherwise},
			\end{array}\right.
		\end{align}	
		\lessgap
		\lessgap
		\begin{align}
			x_{j} =\underset{u\in \R^n}\argmin
			\left\lbrace  
			h(u) + \inner{S_j}{u} +\frac{1}{2\lam}\|u- x_{j}^c \|^2 \right\rbrace,  \label{def:xj} 
		\end{align}
		and
		\lessgap
		\lessgap
		\begin{align} \label{def:yj}
			y_j =  \left\{\begin{array}{ll}
				x_{j}, & \text { if } j=j_{k-1}+1 , \\ 
				(1-\tau) x_j + \tau y_{j-1},  & \text { otherwise};
			\end{array}\right.
		\end{align}		 
		\lessgap
		\item [2.]
		if $j=j_{k-1}+1$, then 
		choose an integer $j_k$ such that
		\[
		j_k \ge j_{k-1} + 1;
		\]
		if $j < j_k$, then
		set $j \leftarrow j+1$
		and
		go to step 1; else,
		set
		$ \hat y_k= y_{j_k} $,
		and
		go to step~3;
		\item[3.]
		if $k<K$, then set
		$k \leftarrow k+1$ and
		$j \leftarrow j+1$, and go to
		step 1; otherwise, compute 
		\begin{equation}\label{eq:output}
			\hat y_K^a = \frac1{\lceil K/2 \rceil}\sum_{k=\lfloor K/2 \rfloor + 1}^K \hat y_k
		\end{equation} and \textbf{stop}.
	\end{itemize}
	{\bf Output:} $\hat y_K^a$.
	
	\noindent
	\rule[0.5ex]{1\columnwidth}{1pt}

	We first discuss the roles played by the
	two index counts $j$ and $k$ used
	by SCPB.
	First, $j$ counts the total number of iterations/resolvent evaluations performed by SCPB since it is
	increased by one every
	time SCPB returns to step 1.
	Second, defining the $k$-th cycle
	as the iteration indices $j$ lying in
	\begin{equation}\label{def:Ck}
		{\cal C}_k := \{ i_{k}, \ldots, j_k\}, \quad \mbox{\rm where} \ \ \  i_k := j_{k-1}+1,
	\end{equation}
	it immediately follows that 
	$k$ counts the number of cycles
	generated by SCPB.
	Third, step 1 determines two types of
	iterations depending on whether $j=j_k$ (serious iteration)
	or $j\in {\cal C}_k\setminus\{j_k\}$ (null iteration). Hence, the last iteration of a cycle is a serious one while the others are null ones.

	We now make several basic remarks about SCPB.
	First, every execution of step 1 involves one resolvent evaluation of $\partial h$, i.e.,
	an evaluation of the point-to-point operator $(I+ \alpha \partial h)^{-1}(\cdot)$ for some $\alpha>0$.
	Second, SCPB generates three sequences of iterates,
	namely, the sequence of prox-centers
	$\{x_j^c\}$ computed in \eqref{def:xc}, the
	sequence $\{x_j\}$ determined by \eqref{def:xj},
	and the sequence $\{y_j\}$ given by \eqref{def:yj}.
	Third, it follows from \eqref{def:xc} that
	$x^c_j=x_{j_{k-1}}$ 
	for every $j \in {\cal C}_k$. Hence,
	the prox-center $x^c_j$ remains constant between consecutive iterations within 
	a cycle and
	(possibly) changes
	only
	at the
	beginning of the first iteration of the following cycle.
	Fourth, $\{\hat y_k\}$ is the subsequence
	of $\{y_j\}$ consisting of all the last cycle iterates $y_{j_k}$ generated
	by SCPB.
	Fifth, the convergence rates for the two specific variants of the SCPB scheme described below
	are with respect to the average of the iterates
	$\hat y_{\lfloor K/2 \rfloor +1}, \ldots,\hat y_K$, namely, the point
	$\hat y_K^a$ as in \eqref{eq:output}
	(see
	Theorems \ref{thm:main1} and \ref{thm:main2} below).
	% Sixth, it is shown in Lemma \ref{lemfirstserious} (resp., Theorem \ref{cor:cmplx1}(b)) that the expected value of $u_j$ (resp., $\hat u_K^a$)
	% computed in \eqref{def:u} (resp., \eqref{eq:output})
	% majorizes the expected value of
	% $\phi(y_j)$ (resp., $\phi(\hat y_K^a)$).
	
	As already mentioned in the second paragraph preceding the description of SCPB,
	the scheme is not completely specified since its step 2
	does not
	describe how to select $j_k$. We now describe two cycle rules for doing so which depend
	on a pre-specified parameter $R>0$, namely:

	% {\bf Note: Change the definition of a serious iteration}
	
	% \subsection{Two cycle rules for SCPB}

	% We now describe two ways of
	% determining $j_k$ in step 2 of SCPB, namely:
	\begin{itemize}
		\item[(B1)] for every $k \ge 1$, let $j_k$ be the smallest integer $\ge i_k$ such that
		$\lam k \tau^{j_k-i_k}\le R$;
		% \red{where $R$ is a user-chosen parameter that will be specified later.}
		\item[(B2)] for every $k \ge 1$, let $j_k$ be the smallest integer $\ge i_k+1$ such that
		\begin{equation}\label{ineq:example}
			\lam k \tau^{j_k-i_k}\left(
			F(x_{i_k},\xi_{i_k})    
			-  \tilde \ell_k (x_{i_k}) - \frac{1}{2\lam}\|x_{i_k}- x_{i_k}^c \|^2 \right) \le R
		\end{equation}
		where  $i_k$ is as in \eqref{def:Ck} and
		\begin{equation}\label{def:tl}
			\tilde \ell_k (\cdot) := F(x_{i_{k}-1},\xi_{i_{k}-1}) + 
			\langle s(x_{i_{k}-1},\xi_{i_{k}-1}),\cdot-x_{i_{k}-1}\rangle.
		\end{equation}

	\end{itemize}
	
	We make the following remarks about cycle rules (B1) and (B2).
	First, the sequence $\{j_k\}$ determined by the cycle rule (B1) is deterministic, while the one determined by (B2) is stochastic since the sequence $\{x_{i_k}\}$ used in \eqref{ineq:example} is stochastic.
	Second, another difference between the two cycle rules is that (B1) allows $j_k=i_k$, while $j_k$ in (B2) is at least $i_k+1$. 
	In other words, the cycle length for (B1) may be equal to one, but the one for
	(B2) is
	at least two.
	Third, the length of cycle ${\cal C}_k$ for both rules above depends on the
	cycle index $k$. Hence,
	even though (B1) is deterministic, the length of the cycles generated by it changes with $k$.
	
	Throughout our presentation,
	SCPB based on cycle rule (B1) (resp., (B2)) is referred to as SCPB1 (resp., SCPB2).

	\section{Complexity results for SCPB1}\label{sec:SCPB1}
	
	This section presents the main complexity results for SCPB1 under various assumptions and discusses the relationship between SCPB1 and RSA.
	% For a given pair $(\lam,K)$,
	% the following result and the paragraph following it describe
	% a way of choosing $C$ and $\theta$
	% so that SCPB based on
	% condition (B1) has optimal iteration
	% complexity.
	
	% \subsection{Complexity analysis when \vg{}{$\dom h$ is bounded}}
	
	\subsection{Convergence rate bounds of SCPB1 with bounded $\dom h$}
	
	The following result states a general convergence rate result for SCPB1 that holds
	for
	bounded $\dom h$ and
	for any choice of
	input $(\lam,\theta,K)$ in SCPB1 and constant $R$ as in (B1).
	The proof is postponed to Subsection \ref{prooffirstth}.

	\begin{theorem}\label{thm:main1}
		Assume that Assumption \ref{assump:basic} holds and $\dom h$ has a finite diameter $D_h \ge 0$. 
		Then, for any given $(\lam,\theta,K) \in \R_{++}^2\times \mathbb{N}_{++}$ and $R>0$,
		SCPB1 with any input $(\lam,\theta,K)$ and constant $R$ in (B1)
		satisfies the following statements:
		\begin{itemize}
			\item[a)] the number of iterations within the $k$-th cycle ${\cal C}_k$ (see \eqref{def:Ck}) is bounded by 
			\begin{equation}\label{inner}
				\left\lceil (\theta K+1)
				\ln_0^+ \left(\frac{\lambda k}{R} \right)\right\rceil + 1;
			\end{equation}
			\item[b)] we have
			\begin{equation}\label{ineq:consequence}
				\E[\phi(\hat y_K^a)] - \phi_* \le  \frac1K \left(\frac{ D_h^2}{\lam} + \frac{6 R \min\{\lam \bar M^2, \bar M D_h\}}{\lam} +  \frac{2\lam \bar M^2}{\theta}\right)
			\end{equation}
			where 
			$D_h$ is the diameter of $\dom h$.
		\end{itemize}
	\end{theorem}
	
	We now make some remarks about Theorem \ref{thm:main1}.
	First, its overall iteration complexity is
	given by $K$, which is its outer iteration complexity, times its inner iteration complexity given in \eqref{inner}. Second, \eqref{ineq:consequence} gives a
	bound on the expected primal gap $\E[\phi(\hat y_K^a)] - \phi_*$ in terms of $K$,
	and hence provides a sufficient condition on
	how large $K$ should be chosen for SCPB1
	to generate a desired approximate solution.

	Even though Theorem \ref{thm:main1} holds for the general case in which $\dom h$ is unbounded, all its corollaries stated in this subsection and Subsection \ref{subsec:practical} assume that
	$\dom h$ is bounded.
	For any given $(\lam,K)$,
	the following result describes a convergence rate bound for SCPB1
	with a specific choice of $(\theta,R)$. 
	% The first result chooses a practical pair , which is the choice in our computational results.
	% The second result chooses a more theoretical pair and 
	% SCPB with this choice of parameters is shown to possess optimal complexity for a large range of prox stepsize $\lam$.
	
	\begin{corollary} \label{cor:cmplx1}
		Assume that Assumption \ref{assump:basic} holds and $\dom h$ has a finite diameter $D_h>0$. Let
		a pair $(\lambda, K)$ be given and
		% an estimate $(D,M)$ of
		% $(d_0,\bar M)$ be given.
		% Moreover, consider SCPB based
		% on condition (B1) where
		% $\theta$ in \eqref{def:tau} and $C$ in (B1)
		% are chosen as
		consider SCPB1 with
		input $(\lam,\theta,K)$
		and $R$
		in (B1) given
		by
		\begin{equation}\label{eq:theta1}
			\theta = \frac{2\lam^2 M^2}{D^2}, \quad 
			% \vg{C = \frac{D^2}{6 \min\{\lam M^2, M D\}}}
			R=\frac{D}{6M}
		\end{equation}
		where $(D,M)$ is an estimate for
		the (usually unknown) pair
		$(D_h,\bar M)$.
		Then, the following statements hold:
		\begin{itemize}
			\item[a)] we have
			\[
			\E[\phi(\hat y_K^a)] - \phi_* \le
			% \vg{\frac{D^2}{\lambda K} \left( \kappa_D + 2 \kappa_M \right)}
			\frac{3 D^2}{2 \lambda K} \left( \kappa_D + \kappa_M \right)
			\]
			where
			\begin{equation}\label{def:kappa}
				\kappa_D:= \frac{D_h^2}{D^2}, \quad
				\kappa_M := \frac{\bar M^2}{M^2};
			\end{equation}
			\item[b)] its expected overall iteration complexity (up to a logarithmic term) is
			\begin{equation}\label{eq:bound1}
				{\cal O}\left( \frac{\lam^2 M^2 K^2}{D^2} + K \right).
			\end{equation}
		\end{itemize}
	\end{corollary} 
	
	\begin{proof}
		a) Using \eqref{ineq:consequence},
		the definitions
		of $\kappa_D$
		and $\kappa_M$
		in \eqref{def:kappa},
		and the definitions
		of $\theta$ and
		$R$ in \eqref{eq:theta1},
		we get
		$$
		\begin{array}{lcl}    
			\displaystyle    \E[\phi(\hat y_K^a)] - \phi_* & \le &
			\displaystyle
			\frac1K \left(\frac{ \kappa_D D^2}{\lam} + \frac{D \min\{\lam \bar M^2, \bar M D_h\}}{\lam M} +  \frac{\kappa_M D^2}{\lambda}\right)\\
			& \leq & \displaystyle \frac{D^2}{\lambda K} \left(\kappa_D + \kappa_M + \sqrt{\kappa_D \kappa_M} \right)\\
			& \leq & \displaystyle \frac{3 D^2}{2 \lambda K} \left(\kappa_D + \kappa_M \right),
		\end{array}
		$$
		where in the second inequality
		we have used
		$\min\{\lam \bar M^2, \bar M D_h\}
		\leq \bar M D_h$
		and the definitions of 
		$\kappa_D$ and $\kappa_M$
		while in the last inequality we have
		used the relation $\sqrt{a b} \leq (a+b)/2$ for every $a,b \geq 0$.
		
		b) It follows from Theorem \ref{thm:main1}(a) that the overall complexity (up to a logarithmic term) is ${\cal O}(\theta K^2 + K)$, which in turn is \eqref{eq:bound1} in view of $\theta$ as in \eqref{eq:theta1}.
	\end{proof}	
	
	We now argue that the overall iteration (and sample) complexity of the
	SCPB1 variant of
	Corollary \ref{cor:cmplx1} for finding an
	$\varepsilon$-solution $x$ of
	\eqref{eq:ProbIntro}, i.e.,
	one that
	satisfies
	% the corresponding
	% ergodic iterate
	% $\hat y_K^a$ satisfies
	$
	\E[\phi(x)] - \phi_* \le \varepsilon
	$,
	is optimal for a large range of
	prox stepsizes.
	% For a given tolerance
	% $\varepsilon>0$, we now consider
	% a specialization of the SCPB1 variant of
	% Corollary \ref{cor:cmplx1} with $K$ chosen so that
	% $x=\hat y_K^a$ is a
	% $\varepsilon$-solution of
	% \eqref{eq:ProbIntro}, i.e., it satisfies 
	% % the corresponding
	% % ergodic iterate
	% % $\hat y_K^a$ satisfies
	% $
	%     \E[\phi(x)] - \phi_* \le \varepsilon
	% $,
	% and show that it has optimal
	% overall iteration complexity for a large range of prox stepsizes.
	% we argue that
	% Corollary \ref{cor:cmplx1} implies
	% that, for any given $\varepsilon$,
	% the $\varepsilon$-iteration
	% complexity of
	% its SCPB variant is optimal.
	% For a given $\varepsilon>0$,
	% we now discuss the iteration complexity of SCPB with a specific choice of $K$ to obtain $\hat y_K^a$ such that
	% $
	%     \E[\phi(\hat y_K^a)] - \phi_* \le \varepsilon
	% $.
	Indeed,
	setting $K= \lceil T_\varepsilon \rceil$ where
	\[
	T_\varepsilon := 
	% \vg{\frac{D^2}{\lambda \varepsilon} \left( \kappa_D + 2\kappa_M \right)}
	\frac{3 D^2}{2 \lambda \varepsilon} \left( \kappa_D + \kappa_M \right),
	\]
	it follows from the above result 
	that
	$
	\E[\phi(\hat y_K^a)] - \phi_* \le \varepsilon
	$.
	% is
	% given by \eqref{eq:bound1} with
	Since $K \le T_\varepsilon+1$, we conclude
	from \eqref{eq:bound1} that the expected overall iteration complexity   of
	SCPB1 is
	bounded by
	\[
	{\cal O}\left( \frac{\lam^2 M^2 (T_\varepsilon +1)^2}{D^2} +  T_\varepsilon +1 \right)
	=
	{\cal O}\left( \frac{M^2 D^2}{\varepsilon^2}\left[ \kappa_D^2 + \kappa_M^2 \right] + \frac{\lam^2 M^2}{D^2} + \frac{D^2}{\lambda \varepsilon} \left[ \kappa_D + \kappa_M \right] + 1 \right).
	\]
	In particular, if
	$D \ge D_h$ and $M \ge \bar M$,
	or equivalently,
	$\kappa_D\le 1$ and $\kappa_M \le 1$, then the above complexity reduces to
	\[
	{\cal O}\left( \frac{M^2 D^2}{\varepsilon^2} + \frac{\lam^2 M^2}{D^2} + \frac{D^2}{\lambda \varepsilon} + 1 \right).
	\]
	Moreover, under the assumption
	that the prox stepsize $\lambda$ lies in the interval
	$[\varepsilon/M^2, D^2/\varepsilon]$, the above complexity bound further reduces to
	${\cal O}(M^2D^2/\varepsilon^2)$,
	which is known to
	be the optimal complexity 
	of finding an $\varepsilon$-solution for any
	instance of \eqref{eq:ProbIntro} such that its corresponding pair
	$(D_h,\bar M)$ satisfies
	the condition that
	$D \ge D_h$ and $M \ge \bar M$
	(e.g., see \cite{nemirovsky1983problem}).

	\subsection{Relationship between
		SCPB1 and the RSA method of \cite{nemjudlannem09}} \label{subsec:relation}
	
	We argue in this subsection that RSA can be viewed as a special instance of SCPB1 where
	every cycle ${\cal C}_k$
	contains only one index
	(or equivalently, an instance for which every
	iteration is serious).
	
	Recall that the RSA method of \cite{nemjudlannem09},
	which is developed under the assumption that $h$ is the indicator function of a nonempty compact convex set $X$,
	with a given initial point $x_0\in X$ and constant prox stepsize $\lam>0$
	% $\lam$ 
	% given by
	% \[
	% \lam = \frac{\alpha D}{M\sqrt{K}}
	% \]
	% where $D$ and $M$ are upper bounds on
	% $D_h$ and $\bar M$, respectively.
	% 	RSA
	recursively computes its iteration sequence $\{x_j\}_{j=1}^N$
	according to
	\begin{equation}\label{eq:sub}
		x_j =\underset{u\in X}\argmin
		\left\lbrace  \inner{s(x_{j-1},\xi_{j-1})}{u} +\frac{1}{2\lam} \|u- x_{j-1} \|^2 \right\rbrace
		\qquad \forall j=1,\ldots,N.
	\end{equation}
	
	For the purpose of reviewing the iteration complexity of RSA, assume that
	$D$ is an upper bound on
	the diameter of $X$ and
	$M$ is an upper bound on $\bar M$.
	For $1\le i \le N$, let $\tilde x_i^N$ denote the average of the iterates $\{x_j\}_{j=i}^N$, i.e., 
	\begin{equation}\label{def:tx}
		\tilde x^N_i=\frac{1}{N-i+1}\sum_{j=i}^N x_j.
	\end{equation}
	It is shown in equation (2.24)
	of \cite{nemjudlannem09} that if the stepsize $\lam>0$ is chosen as
	\begin{equation}\label{eq:lam-RSA}
		\lam = \frac{\alpha D}{M\sqrt{N}}
	\end{equation}
	for some fixed scalar $\alpha>0$,
	% where $N$ is the specified number of iterations in \eqref{eq:sub} and $D$ (resp., $M$) is an upper bound on the diameter of $X$ (resp., $\bar M$),
	then the ergodic iterate $\tilde x^N_i$ 
	% as in \eqref{def:tx} satisfies
	% \[
	%     \E[\phi(\tilde x_i^N)] - \phi_* \le  \max\{\alpha, \alpha^{-1}\}\frac{DM}{\sqrt{N}}\left(\frac{2N}{N-i+1}+\frac12 \right).
	% \]
	with $i=\lfloor N/2 \rfloor +1$ satisfies
	\begin{equation}\label{ineq:converge}
		\E[\phi(\tilde x^N_{\lfloor N/2 \rfloor +1})] - \phi_* \le  \max\{\alpha, \alpha^{-1}\}\frac{9DM}{2\sqrt{N}}.
	\end{equation}
	Hence, for a given tolerance $\varepsilon>0$, the smallest $N$ satisfying $\E[\phi(\tilde x^N_{\lfloor N/2 \rfloor +1})] - \phi_* \le \varepsilon$
	has the property that
	% the overall iteration complexity of RSA to obtain $\tilde x^N_{\lfloor N/2 \rfloor +1}$ such that $\E[\phi(\tilde x^N_{\lfloor N/2 \rfloor +1})] - \phi_* \le \varepsilon$
	% is 
	\begin{equation}\label{eq:complexity-RSA}
		N=   {\cal O}\left(\frac{\max\{\alpha^2,\alpha^{-2}\} M^2 D^2}{\varepsilon^2} \right).
	\end{equation}
	
	It turns out that
	RSA is a special case
	of 
	SCPB1 with  $R$ in (B1)  given by
	\[
	R=\frac{\alpha D \sqrt{K}}{M}.
	\]
	%  and  cycle rule (B1)
	%  is used to determine
	%  the size of a cycle.
	Indeed, it follows from the above choice of $R$ and $\lam$ as in \eqref{eq:lam-RSA} with $N$ replaced by $K$ that
	\[
	\frac{R}{\lam k} \ge \frac{R}{\lam K} = 1
	\]
	and hence that $j_k=i_k$ satisfies (B1).
	Thus, every cycle only performs 
	one iteration, i.e., its only
	serious iteration. Moreover, every iteration of this SCPB1 variant is a serious one and $K$ is its total number of iterations.

	\subsection{A practical SCPB1 variant}\label{subsec:practical}
	
	From a computational point of view, 
	the choice of $\theta$
	in 
	Corollary \ref{cor:cmplx1}
	% does not
	% depend on $K$ and
	usually results in
	the quantity $\theta K$,
	and hence the inner complexity bound \eqref{inner},
	being large.
	The following result
	provides a practical variant of SCPB1 with an alternative choice for $\theta$ and
	$R$ which partially remedies the
	above drawback
	by forcing $\theta K$ to
	be constant.
	% We describe a practical SCPB variant that will be used in our numerical experiments and provide the following theoretical guarantee of it.
	A nice feature of this variant is that it is able to choose large
	prox stepsizes without loosing
	the optimality of its overall
	iteration complexity.

	\begin{corollary} 
		\label{cor:cmplx1-practical2}
		Assume that Assumption \ref{assump:basic} holds and $\dom h$ has a finite diameter $D_h>0$. Let positive
		integer $K$ and constant $C\ge 1$ be given, and
		define
		\begin{equation}\label{eq:theta1-practical2}
			\theta = \frac{C}{K}, \quad 
			R = \frac{D}{M},\quad
			\lambda= \frac{\sqrt{C}  D}{M \sqrt{K}}
		\end{equation}
		where $(D,M)$ is an estimate for the pair $(D_h,\bar M)$.
		Then, the following statements about
		SCPB1 with input $(\lam,\theta,K)$ and $R$ as  above
		hold:
		\begin{itemize}
			\item[a)] we have
			\begin{equation}\label{ineq:Dh2}
				\E[\phi(\hat y_K^a)] - \phi_* \le  \frac{(4\kappa_D+5\kappa_M) D M}{\sqrt{C K}} 
			\end{equation}
			where $\kappa_D$ and $\kappa_M$ are as in \eqref{def:kappa};
			\item[b)]
			the number of iterations within the $k$-th cycle ${\cal C}_k$ is bounded by 
			\[
			\left\lceil (C+1)
			\ln_0^+ \left(\frac{\sqrt{C} k}{\sqrt{K}} \right)\right\rceil + 1,
			\]
			and hence, up to a logarithmic term,  is
			${\cal O}(C)$;
			\item[c)] its expected overall iteration complexity,
			up to a logarithmic term, is 
			${\cal O}(C K)$.
		\end{itemize}
	\end{corollary}

	\begin{proof}
		a) Using \eqref{ineq:consequence},
		the definitions
		of $\kappa_D$
		and $\kappa_M$
		in \eqref{def:kappa},
		and the definitions
		of $\theta$ and
		$R$ in \eqref{eq:theta1-practical2},
		we get
		\begin{align}
			\E[\phi(\hat y_K^a)] - \phi_* & \le
			\frac1K \left(\frac{ \kappa_D D^2}{\lam} + \frac{6 D \min\{\lam \bar M^2, \bar M D_h\}}{\lam M} +  \frac{2 \lam K {\bar M}^2}{C}\right) \nn \\
			& \leq \frac{1}{\lambda K} \left(
			\kappa_D D^2
			+\frac{6 D {\bar M}D_h}{M}\right)+ \frac{2\lam M^2 \kappa_M}{C} \nn \\
			& = \frac{D^2}{\lambda K}(\kappa_D + 6 \sqrt{\kappa_M\kappa_D}) + \frac{2\lam  M^2 \kappa_M}{C}, \label{ineq:rate-bound}
		\end{align}
		where in the second inequality
		we have used
		$\min\{\lam \bar M^2, \bar M D_h\}
		\leq \bar M D_h$
		and the definition of $\kappa_M$.
		It follows from \eqref{ineq:rate-bound} and the fact that $\sqrt{\kappa_M \kappa_D}\le (\kappa_M+\kappa_D)/2$ that
		\[
		\E[\phi(\hat y_K^a)] - \phi_* \le \frac{(4\kappa_D+3\kappa_M)D^2}{\lambda K}+ \frac{2\kappa_M\lam  M^2}{C}.
		\]
		Finally, the above bound with $\lam$ as in \eqref{eq:theta1-practical2}
		implies \eqref{ineq:Dh2}.
		
		b) This statement immediately follows from Theorem \ref{thm:main1}(a) with $\theta$, $R$, and $\lam$ as in \eqref{eq:theta1-practical2}.
		
		c) This statement follows from (b) and the fact that SCPB1 has $K$ cycles.
	\end{proof}

	We now make
	two remarks about
	the practical SCPB1 variant of Corollary~\ref{cor:cmplx1-practical2}.
	First, although $\theta$  in
	\eqref{eq:theta1-practical2} depends neither
	on  $M$ 
	nor $D$,
	the choice of $R$ depends
	on both of these estimates.
	On the other hand, SCPB2 will be analyzed in Section \ref{sec:SCPB2} where
	$\theta$ 
	depends neither on $M$ nor $D$, and
	$R$ depends on $D$ but not $M$.
	% ??????Thus, this second variant has the interesting
	% feature in that it does not
	% require the estimate $M$ which,
	% compared to $D$,
	% is usually much harder to obtain.
	Second, Corollary~\ref{cor:cmplx1-practical2} (see its statement (b)) implies that the number of
	iterations within a cycle of
	SCPB1
	is bounded (up to a logarithmic term) by the a priori
	(user specified) constant $C$.
	Thus, 
	the SCPB1 variant of Corollary \ref{cor:cmplx1-practical2}
	can be viewed as an extended
	version of RSA where
	the number of
	iterations within a cycle can be larger than one, instead of
	being equal to one as in RSA
	(see the discussion in the second paragraph 
	of Subsection~\ref{subsec:relation}).
	
	In the remaining part of this subsection, we compare the
	performance
	of RSA and SCPB1 when both use the prox stepsize
	$\lam$ as in \eqref{eq:theta1-practical2}
	for some relatively large scalar $C \ge 1$.
	For this discussion, we assume that their performance measure
	is the overall iteration complexity (or sample complexity) for finding
	an $\varepsilon$-solution of \eqref{eq:ProbIntro}.
	% it is interesting
	% to compare the behavior of
	% RSA and the SCPB variant of Corollary \ref{cor:cmplx1-practical2}
	% for finding an $\varepsilon$-solution of \eqref{eq:ProbIntro}
	% when both choose
	% the same prox stepsize $\lam$
	% as in \eqref{eq:theta1-practical2} where $C \ge 1$ is
	% a possibly large scalar.
	For simplicity,
	we assume as in
	Subsection~\ref{subsec:relation} that $h$ is the indicator function of a nonempty closed convex set and
	that the
	estimation pair
	$(D,M)$ satisfies $D\ge D_h$ and $M\ge \bar M$, or equivalently, $\kappa_D\le 1$ and $\kappa_M\le 1$.
	
	We first consider the performance (see the previous paragraph) of SCPB1.
	It follows from Corollary \ref{cor:cmplx1-practical2}(a) that there exists 
	$K = {\cal O}(D^2  M^2/(C\varepsilon^2))$
	such that
	$\hat y_K^a$ is an $\varepsilon$-solution of \eqref{eq:ProbIntro}.
	Hence, it follows from
	Corollary \ref{cor:cmplx1-practical2}(c) that
	the performance of SCPB1
	is ${\cal O}(M^2 D^2/\varepsilon^2)$.
	In conclusion,
	SCPB1 with the above choice of $K$
	is able to choose
	a prox stepsize $\lam$ as in \eqref{eq:theta1-practical2} with
	a large constant $C$ while preserving its optimal performance.
	We now consider the performance
	of RSA.
	% Fourth, it is interesting
	% to compare the behavior of
	% RSA and the SCPB variant of Corollary \ref{cor:cmplx1-practical2}
	% for finding an $\varepsilon$-solution of \eqref{eq:ProbIntro}
	% when both choose
	% the same prox stepsize $\lam$
	% as in \eqref{eq:theta1-practical2} where $C \ge 1$ is
	% a possibly large scalar.
	% (For simplicity,
	% we assume as in
	% Subsection~\ref{subsec:relation} that $h$ is the indicator function of a nonempty closed convex set.)
	% Indeed,
	It follows from \eqref{eq:lam-RSA} and \eqref{eq:complexity-RSA} with $\alpha=\sqrt{C}$ that the performance of  RSA is
	% with 
	% $\lam$ as in \eqref{eq:theta1-practical2} has
	% overall iteration complexity given by
	$
	{\cal O}\left(
	C D^2 M^2/\varepsilon^2
	\right)$.
	In conclusion, while both RSA and SCPB1 with prox stepsize $\lam$ as in \eqref{eq:theta1-practical2} have their own performance guarantee,
	the one for RSA
	becomes worse than that of SCPB1 as $C$ becomes large.
	% while the one for SCPB does not depend on $C$ and remains equal to the
	% optimal complexity bound for the class
	% of all
	% instances of \eqref{eq:ProbIntro}.
	% such that their corresponding pair
	% $(D_h,\bar M)$ satisfies
	% the condition that
	% $D \ge D_h$ and $M \ge \bar M$.
	
	Finally, although the SCPB1 variant of Corollary \ref{cor:cmplx1-practical2} chooses $\lam$ as in \eqref{eq:theta1-practical2}, 
	our numerical experiments uses
	a more aggressive prox stepsize, i.e.,
	\[
	\lambda=\beta_1 \frac{\sqrt{C}  D}{M \sqrt{K}}
	\]
	where $\beta_1=10$.
	It is possible to show that similar conclusions (with slightly modified bounds) as those of Corollary \ref{cor:cmplx1-practical2}  hold for this choice of $\lambda$. Finally, SCPB1 with this prox stepsize $\lam$ substantially outperforms RSA
	on the (relatively small number of) instances considered in the computational  experiments
	of Section \ref{sec:num}.

	\section{Complexity results for SCPB2} \label{sec:SCPB2}
	
	This section provides the main complexity results for SCPB2.
	The following result is an analogue of Theorem \ref{thm:main1} and describes the convergence rate bound for the SCPB2 without imposing any condition on its input $(\lam, \theta, K)$
	and the constant $R$ in (B2).
	The proof is postponed to Subsection \ref{subsec:proof2}.

	\begin{theorem}\label{thm:main2}
		Assume that Assumption \ref{assump:basic} holds and $\dom h$ has a finite diameter $D_h>0$.
		% and that $\theta K \ge 1$.
		Then, SCPB2
		satisfies the following statements:
		\begin{itemize}
			\item[a)] the expected number of iterations within the $k$-th cycle ${\cal C}_k$ (see \eqref{def:Ck}) is bounded by 
			\begin{equation}\label{ineq:Enk}
				\left \lceil(\theta K +1) \ln_0^+\left(\frac{2\bar M^2 \lam^2 k}{R} \right) \right \rceil + 1;
			\end{equation}
			\item[b)] we have
			\[
			\E[\phi(\hat y_K^a)] - \phi_* \le  \frac1K \left( \frac{3 R + D_h^2}{ \lam } + \frac{2\lam \bar M^2}{\theta } + \frac{2\lam \bar M^2}{\theta^2 K}
			\right).
			\]
		\end{itemize}
	\end{theorem}
	
	Following a similar argument as in the paragraph following Corollary \ref{cor:cmplx1},  it can be shown that SCPB2 has optimal iteration complexity (up to a logarithmic term) for finding an $\varepsilon$-solution of \eqref{eq:ProbIntro} for a large range of prox stepsizes.
	
	The following result
	is the analogue of
	Corollary \ref{cor:cmplx1-practical2}
	when SCPB2 is implemented using
	cycle rule (B2)
	instead of (B1).
	As in Corollary \ref{cor:cmplx1-practical2}, it forces
	the quantity $\theta K$ to be constant but,
	in contrast to the
	choice of $R$ of Corollary \ref{cor:cmplx1-practical2},
	its choice for $R$ 
	does not depend on an estimate $M$ for $\bar M$.

		\begin{corollary} \label{cor:cmplx1-1-practicalscpb2bis}
			Assume that Assumption \ref{assump:basic} holds and $\dom h$ has a finite diameter $D_h>0$. Let positive integers $K$ and constant $C\ge 1$ be given, and
			% a pair $(\lambda, K)$ be given and
			define
			\begin{equation}\label{eq:theta1-practical2b}
				\theta = \frac{C}{K}, \quad 
				R = D^2,\quad
				\lambda= \frac{\sqrt{C} D}{M \sqrt{K}}
			\end{equation}
			where $D$ is an estimate for $D_h$
			and $M$ is an estimate for $\bar M$.
			% such that $D \geq D_h$ and
			% $M \geq \bar M$.
			Then, the following statements for 
			SCPB2 with input $(\lam,\theta,K)$ based
			on cycle rule (B2) with $R$, $\theta$, and 
			$\lambda$ as  above
			hold:
			\begin{itemize}
				\item [a)]
				we have 
				\begin{equation}\label{finalboundscpb2}
					\E[\phi(\hat y_K^a)] - \phi_* \le \frac{(3+\kappa_D+4\kappa_M) D M}{\sqrt{C K}}
				\end{equation}
				where $\kappa_D$ and $\kappa_M$ are as in \eqref{def:kappa};
				\item[b)] the expected number of iterations within the $k$-th cycle ${\cal C}_k$ is bounded by 
				\[
				\left \lceil(C +1) \ln_0^+\left(\frac{2\kappa_M C k}{K} \right) \right \rceil + 1,
				\]
				and hence, up to a logarithmic term,  is
				${\cal O}(C)$;
				\item[c)] its expected overall iteration complexity,
				up to a logarithmic term, is 
				${\cal O}(C K)$.
				% if $K$ is chosen as
				% \begin{equation}\label{pracchoiceKscpb2}
					% K = \left \lceil \frac{64 D^2  M^2}{C\varepsilon^2} \right \rceil,
					% \end{equation}
				% then
				% $
				%     \E[\phi(\hat y_K^a)] - \phi_* \le \varepsilon
				% $
				% and the expected overall
				% iteration
				% complexity of SCPB
				% is ${\cal O}(M^2 D^2/\varepsilon^2)$.
			\end{itemize}
		\end{corollary}
		
		\begin{proof}
			a) Using Theorem \ref{thm:main2}(b) with $\theta$ and $R$ as in \eqref{eq:theta1-practical2b} and the definitions
			of $\kappa_D$
			and $\kappa_M$
			in \eqref{def:kappa}, we have
			\[
			\E[\phi(\hat y_K^a)] - \phi_* \le \frac {(3+\kappa_D) D^2}{ \lam K} + \frac{4\kappa_M\lam M^2}{C},
			\]
			which together with  $\lambda$  in 
			\eqref{eq:theta1-practical2b} implies \eqref{finalboundscpb2}.
			
			b) This statement follows from \eqref{ineq:Enk} with $\theta$, $R$, and $\lam$ as in \eqref{eq:theta1-practical2b} and the definition of $\kappa_M$ in \eqref{def:kappa}.
			% and the assumption that $M\ge \bar M$.
			
			c) This statement follows from (b) and the fact that SCPB2 has $K$ cycles.
			% It follows from \eqref{finalboundscpb2} with $K$ as in  \eqref{pracchoiceKscpb2} that $\E[\phi(\hat y_K^a)] - \phi_* \le \varepsilon$. Moreover, it follows from (b) that the overall expected overall
			% iteration complexity of SCPB is ${\cal O}(CK)$, i.e., ${\cal O}(M^2 D^2/\varepsilon^2)$ in view of $K$ in \eqref{pracchoiceKscpb2}.
		\end{proof}

		\if{
			
			\section{Inexact SHCPB}
			
			In this section, we 
			are still interested to solve
			\eqref{eq:ProbIntro}
			with $f$ given by
			\eqref{pbint2}.
			We consider the situation where
			we have access to an inexact
			oracle and derive 
			a corresponding
			inexact variant of
			SCPB denoted  by ISCPB. More
			specifically, the
			inexact oracle
			returns at iteration
			$j$ an inexact stochastic
			subgradient, which, for 
			$x \in \dom h $
			and $\xi \in \Xi$,
			returns $s_j(x;\xi)$ such that
			\begin{equation}\label{inexoracle1}
				\mbox{(I1)}\;\;\;\;
				f_j'(x):=\mathbb{E}[s_j(x;\xi)] \in 
				\partial_{\beta_j} f(x)
			\end{equation}
			for some nonnegative error term $\beta_j \geq 0$. At iteration $j$, the inexact
			oracle also provides approximate
			values $\overline F_j(x,\xi)$
			for $F(x,\xi)$ which satisfy
			\begin{equation}\label{inexoracle2}
				\mbox{(I2)}\;\;\;F(x,\xi)-\beta'_j \leq 
				\overline F_j(x,\xi)
				\leq F(x,\xi)+\beta'_j
			\end{equation}
			for some $\beta'_j \geq 0$
			and all
			$x \in \dom h$,
			$\xi \in \Xi$.
			At iteration $j$,
			the linearization
			\begin{equation}\label{exactlin}
				\ell_\Phi(\cdot;x_j,\xi_j)= F(x_{j},\xi_j)+\inner{s(x_{j};\xi_j)}{\cdot-x_{j}} + h(\cdot)
			\end{equation}
			used in SCPB
			is replaced in
			ISCPB
			by
			\begin{equation}\label{inexlineardef}
				\ell_{\Phi,j}(\cdot;x_j,\xi_j)=
				{\overline F}_j(x_j,\xi_j)
				+ \langle s_j(x_j;\xi_j),\cdot - x_j \rangle +h(\cdot)-\beta_j-\beta'_j
			\end{equation}
			where exact function
			values $F(x_j,\xi_j)$
			and 
			subgradient
			values $s(x_j;\xi_j)$
			in \eqref{exactlin}
			are respectively replaced
			by their approximate
			values 
			${\overline F}(x_j,\xi_j)$
			and $s_j(x_j;\xi_j)$
			provided by the inexact
			stochastic oracle.
			Correspondingly, we will use
			the notation
			$$
			{\overline \Phi}_j(x,\xi_j)
			={\overline F}_j(x,\xi_j)+h(x)
			$$
			so that linearization
			$\ell_{\Phi,j}(\cdot;x_j,\xi_j)$
			in \eqref{inexlineardef}
			can be written
			\begin{equation}\label{inexlineardefbis}
				\ell_{\Phi,j}(\cdot;x_j,\xi_j)=
				{\overline \Phi}_j(x_j,\xi_j)
				+ \langle s_j(x_j;\xi_j),\cdot - x_j \rangle -\beta_j-\beta'_j
			\end{equation}

			ISCPB is simply obtained
			replacing at iteration $j$ in SCPB 
			function values
			$F(x_j,\xi_j)$
			by approximate
			function values
			${\overline F}(x_j,\xi_j)$
			and linearization
			$\ell_{\Phi}(\cdot;x_j,\xi_j)$ by 
			approximate
			linearizations
			$\ell_{\Phi,j}(\cdot;x_j,\xi_j)$.
			
			\noindent\rule[0.5ex]{1\columnwidth}{1pt}
			
			ISCPB
			
			\noindent\rule[0.5ex]{1\columnwidth}{1pt}
			\begin{itemize}
				\item [0.] Let  $\lam,\theta>0$, integer $K>0$,
				and $ x_0\in \dom h $
				be given, and set $x_0^c=x_0$,
				$j=k=1$, $j_0=0$, and 
				\begin{equation}\label{def:taui}
					\tau = \frac{\theta K}{\theta K + 1};
				\end{equation}
				moreover, take a sample $\xi_0$ of
				r.v.\ $\xi$ and nonnegative real sequences $(\beta_j)$ and $(\beta'_j)$;
				% 	\item [1.] 
				% 	let $n_k$ be the smallest $n \ge 2$ satisfying
				% \begin{equation}\label{ineq:nk}
					%     \tau^{n-1} \le \frac{a}{\lam k};
					% \end{equation}
				% i.e., take 
				% $$
				% n_k=1+
				% \left \lceil 
				% \frac{\ln(\frac{a}{\lambda k})}{\ln(\tau)}
				% \right \rceil; 
				% $$
				\item [1.] compute
				\lessgap
				\begin{align} \label{def:xci}
					x_j^c = \left\{\begin{array}{ll}
						x_{j_{k-1}}, & \text { if } j=j_{k-1}+1 , \\ 
						x_{j-1}^c,  & \text { otherwise},
					\end{array}\right.
				\end{align}
				\lessgap
				\lessgap
				% 	\begin{align} \label{eq:Gamma}
					% 	\Gamma_{j} =  \left\{\begin{array}{ll}
						% 	    \ell_\Phi(\cdot;x_{j-1},\xi_{j-1}), & \text { if } i =1 , \\ 
						% 		    (1-\tau)
						% 	    \ell_\Phi(\cdot;x_{j-1},\xi_{j-1}) +  \tau \Gamma_{j-1} ,  & \text { otherwise},
						% 		    \end{array}\right.
					%     \end{align}	
				%     \lessgap
				%     \lessgap
				\begin{align}
					x_{j} =\underset{u\in \R^n}\argmin
					\left\lbrace  
					\Gamma_{j}^\lam(u):=
					\Gamma_{j}(u) +\frac{1}{2\lam}\|u- x_{j}^c \|^2 \right\rbrace,  \label{def:xji} 
				\end{align}
				where
				\begin{align} \label{eq:Gammai}
					\Gamma_{j}(\cdot) :=  \left\{\begin{array}{ll}
						\ell_{\Phi,j}(\cdot;x_{j_{k-1}},\xi_{j_{k-1}}), & \text { if } j=j_{k-1}+1 , \\ 
						(1-\tau)
						\ell_{\Phi,j}(\cdot;x_{j-1},\xi_{j-1}) +  \tau \Gamma_{j-1}(\cdot) ,  & \text { otherwise};
					\end{array}\right.
				\end{align}	
				also, compute
				\lessgap
				\lessgap
				\begin{align} \label{def:yji}
					y_j = (1-\tau) x_j + \left\{\begin{array}{ll}
						\tau x_{j-1}, & \text { if } j=j_{k-1}+1 , \\ 
						\tau y_{j-1},  & \text { otherwise},
					\end{array}\right.
				\end{align}		 
				\lessgap
				and take a sample $\xi_j$ of r.v.\ $\xi$ which is independent of $\{\xi_0,\xi_1,\ldots,\xi_{j-1}\}$,
				and set  
				\begin{align} \label{def:ui}
					u_j = (1-\tau) 
					\Phi(x_{j},\xi_{j}) + \left\{\begin{array}{ll}
						\tau \Phi(x_{j-1},\xi_{j-1}), & \text { if } j=j_{k-1}+1 , \\ 
						\tau u_{j-1},  & \text { otherwise};
					\end{array}\right.
				\end{align}		    
				\item [2.]
				if $j=j_{k-1}+1$, then 
				choose an integer $j_k$ such that
				% $j \ge j_{k-1} + 2$ satisfying
				\begin{equation}\label{ineq:nki}
					j_k \ge j_{k-1} + 2;
				\end{equation}
				if $j < j_k$, then
				set $j \leftarrow j+1$
				and
				go to step 1; else,
				set
				$\hat u_k=u_{j_k}$
				and $ \hat y_k= y_{j_k} $,
				and
				go to step~3;
				\item[3.]
				if $k=K$ then \textbf{stop} and output 
				\begin{equation}\label{eq:outputi}
					\hat y_K^a = \frac1{\lceil K/2 \rceil}\sum_{k=\lfloor K/2 \rfloor + 1}^K \hat y_k, \quad \hat u_K^a = \frac1{\lceil K/2 \rceil}\sum_{k=\lfloor K/2 \rfloor + 1}^K \hat u_k;
				\end{equation}
				otherwise,
				set %$\ell_{k+1} = j_k+1$,  
				$k \leftarrow k+1$ and
				$j \leftarrow j+1$, and go to
				step 1.
			\end{itemize}
			\rule[0.5ex]{1\columnwidth}{1pt}
			
			By definition 
			of the approximate
			function and subgradient
			values, we have
			in the inexact case
			\begin{equation}\label{ver2stage}
				\begin{array}{lcl}
					F(x,\xi_j)+h(x)
					&\geq & 
					F(x_j,\xi_j)
					+ \langle s_j(x_j;\xi_j),x - x_j \rangle +h(x)- \Delta_j\\
					&\geq & 
					{\overline F}(x_j,\xi_j)
					+ \langle s_j(x_j;\xi_j),x- x_j \rangle +h(x)
					-\Delta_j
					-\Delta'_j \\
					& = & 
					\ell_{\phi,j}(x;x_j,\xi_j).
				\end{array}
			\end{equation}
			
			Due to the fact that  linearizations $\ell_{\phi,j}(\cdot;x_j,\xi_j)$ are below
			$F(\cdot,\xi_j)+h(\cdot)$,
			in the terminology of
			\cite{liang2020proximal},
			models $\Gamma_j$
			are still bundle functions
			with respect to $\phi$.
			
			The exact case of SHCPB is of course a special case
			of inexact SHCPB 
			with $\Delta_j+\Delta'_j=0$ where we get an exact linearization.
			
			We still assume (A2), (A3), (A5),
			and (A1) replacing the oracle
			by the inexact stochastic
			oracle we have just described
			that satisfies I1 and I2.
			Assumption (A4) will naturally
			be replaced by
			\begin{itemize}
				\item[(A4')] there is $\sigma \geq 0$ such that 
				$\mathbb{E}\left[\left\|s_j\left(x; \xi\right)-f_j'(x)\right\|^{2}\right] \leq \sigma^{2}$ for all
				$j \geq 0$.
			\end{itemize}
			
			In the case of two-stage stochastic programs of
			form \eqref{def2stageconv1}-\eqref{def2stageconv2} an inexact stochastic oracle can arise and
			be obtained as follows. 
			We assume that given $x_1 \in X_1$ and
			$\xi$ (here by a slight abuse of notation
			$\xi$ denotes a realization of the random vector)
			we solve at iteration $j$ subproblems of
			form \eqref{def2stageconv2} approximately, meaning
			that we compute approximate optimal primal-dual solutions of \eqref{def2stageconv2}.
			The computation of an inexact stochastic
			subgradient of $\mathcal{Q}$ for such problems
			is given 
			for linear problems in Proposition 2.1 of \cite{guigues2016isddp},
			for convex nonlinear differentiable problems in Proposition 2.2
			of \cite{guigues2016isddp} and Proposition 3.2
			of \cite{guiguesinexactsmd}.
			For second stage problems with strongly concave dual functions,
			inexact stochastic subgradients are derived in
			Proposition 3.8 of \cite{guiguesinexactsmd} (see
			\cite{guiguesinexactsmd} and \cite{guistrongconcave}
			for conditions ensuring that this assumption of strong
			concavity of the dual function is satisfied).
			For two-stage stochastic nondifferentiable
			problems, inexact stochastic subgradients are
			derived in Proposition 2.3
			of \cite{gms2020isddp}.
			More precisely, assume that 
			$\hat x_2$ is an 
			$\varepsilon_{P}$-optimal feasible solution of problem \eqref{def2stageconv2}
			written for
			$x_1 = \bar x_1$
			and consider the following reformulation of problem 
			\eqref{def2stageconv2}:
			\begin{equation}\label{refdef2stageconv2b}
				\mathfrak{Q}(\bar x_1,\xi) =\left\{ 
				\begin{array}{l}
					\min_{x_2,z} \;f_2(z,x_2,\xi)\\
					A_2 x_2 + B_2 z = b_2,\\
					g_2(z,x_2,\xi) \leq 0,\\
					x_2 \in X_2,\\
					z=\bar x_1. \;\;[\hat \lambda]
				\end{array}
				\right.
			\end{equation}    
			Let 
			$\theta_{\bar x_1}(\lambda)$
			be the dual function
			$$
			\theta_{\bar x_1}(\lambda)=
			\left\{ 
			\begin{array}{l}
				\displaystyle \min_{x_2,z} \;f_2(z,x_2,\xi)+ \langle
				\lambda, z-\bar x_1 \rangle \\
				A_2 x_2 + B_2 z = b_2,\\
				g_2(z,x_2,\xi) \leq 0,\\
				x_2 \in X_2,
			\end{array}
			\right.
			$$
			for problem \eqref{refdef2stageconv2b}
			and consider the corresponding
			dual problem
			\begin{equation}
				\label{pbdualine}
				\max_{\lambda} \; \theta_{\bar x_1}(\lambda).
			\end{equation}
			Let  $\hat \lambda$ be an
			$\varepsilon_{D}$-optimal 
			feasible solution of the dual problem 
			\eqref{pbdualine}.
			Then under assumptions given in Proposition 2.3
			of \cite{gms2020isddp} we have that 
			for all $x_1 \in X_1$:
			\begin{equation}\label{eqcutin2st}
				\begin{array}{lcl}
					\mathfrak{Q}(x_1,\xi)
					& \geq & f_2(\bar x_1,\hat x_2,\xi)
					-(\varepsilon_{P}+\varepsilon_{D})
					+ \langle \hat \lambda , x_1 - \bar x_1 \rangle \\
					& \geq & \mathfrak{Q}(\bar x_1, \xi)
					-(\varepsilon_{P}+\varepsilon_{D})
					+ \langle \hat \lambda, x_1 - \bar x_1 \rangle. 
				\end{array}
			\end{equation}
			If $f_1'(\bar x_1) \in \partial f_1(\bar x_1)$
			it follows that
			\begin{equation}\label{eqsjdef}
				s(\bar x_1,\xi)=f_1'(\bar x_1) + \hat \lambda 
				\in \partial_{\Delta} (f_1+\mathfrak{Q}(\cdot,\xi))(\bar x_1) 
			\end{equation}
			where $\Delta=\varepsilon_{P}+\varepsilon_{D}$.
			However, recall that we assumed that only
			approximate values of
			$\mathfrak{Q}(x_1,\xi)$ are computed by the oracle
			and therefore we can build
			the following inexact linearization
			for $f_1(\cdot)+\mathfrak{Q}(\cdot,\xi)$
			at $\bar x_1$:
			$$
			f_1(\bar x_1)+f_2(\bar x_1,\hat x_2,\xi)
			-(\varepsilon_{P}+\varepsilon_{D})
			+ \langle f_1'(\bar x_1) + \hat \lambda, x_1 - \bar x_1 \rangle.
			$$
			It follows that
			for such 
			two-stage stochastic
			problems,
			we can derive an
			inexact variant of SHCPB
			where
			at iteration $j$
			the stochastic
			linearization 
			$\ell_{\phi,j}(\cdot;x_{1 j},\xi_j)$
			is given
			by 
			\begin{equation}\label{deflphilin2st}
				\ell_{\phi,j}(\cdot;x_{1 j},\xi_j)=
				f_1(x_{1 j})+
				f_2(x_{1 j},x_{2 j},\xi_j)+
				\langle 
				\hat \lambda_j + f_1'(x_{1 j}),\cdot-x_{1 j}
				\rangle +h(\cdot)-(2\varepsilon_{P,j}+\varepsilon_{D,j})
			\end{equation}
			where for iteration $j$
			\begin{itemize}
				\item $x_{1 j}$ is the first stage solution;
				\item $x_{2 j}$
				is an $\varepsilon_{P,j}$-optimal
				primal solution of \eqref{def2stageconv2}
				written for $\bar x_1=x_{1 j}$ and $\hat \lambda_j$
				is an approximate 
				$\varepsilon_{D,j}$-optimal dual solution.
			\end{itemize}
			We see that
			inexact
			linearization
			$\ell_{\phi,j}$ for  two-stage
			stochastic
			programs
			given
			by \eqref{deflphilin2st}
			is of form
			\eqref{inexlineardef}
			with
			$x_j=x_{1 j}$,
			$$
			\overline{F}(x_j,\xi_j)=
			\overline{F}(x_{1 j},\xi_j)=
			f_1(x_{1 j})
			+f_2(x_{1 j},x_{2 j},\xi_j)
			\mbox{ and }s_j(x_j;\xi_j)=\hat \lambda_j 
			+f_1'(x_{1 j}).
			$$
			By definition of $x_{2 j}$ we have
			$$
			\mathfrak{Q}(x_{1 j},\xi_j)
			\leq f_2(x_{1 j},x_{2 j},\xi_j)
			\leq \mathfrak{Q}(x_{1 j},\xi_j) + \varepsilon_{P,j}
			$$
			which gives
			$$
			F(x_{1 j},\xi_j) \leq {\overline F}(x_{1 j},\xi_j)
			\leq F(x_{1 j},\xi_j) + \varepsilon_{P,j}
			$$
			and \eqref{inexoracle2}
			is satisfied
			with $\Delta'_j=\varepsilon_{P,j}$.
			Also
			\eqref{eqcutin2st}
			shows that
			$$
			s_j(x_{j};\xi_j)=
			s_j(x_{1 j};\xi_j)=
			\hat \lambda_j + f_1'(x_{1 j})
			\in \partial_{x,\Delta_j} F(x_{1 j},\xi_j)
			$$
			for
			$\Delta_j=\varepsilon_{P,j}+\varepsilon_{D,j}$.
			Since (I1)
			and (I2)
			are satisfied we can apply
			\eqref{ver2stage} which
			can be written
			for two-stage stochastic problems
			$$
			\begin{array}{lcl}
				F(x,\xi_j)+h(x)
				&\geq &   {\overline F}(x_j,\xi_j)
				+ \langle s_j(x_j;\xi_j),x- x_j \rangle +h(x)
				-\Delta_j
				-\Delta'_j\\
				&=&f_1(x_{j})
				+f_2(x_{j},x_{2 j},\xi_j) +
				\langle 
				\hat \lambda_j + f_1'(x_{j}),\cdot-x_{j}
				\rangle +h(\cdot)-(2\varepsilon_{P,j}+\varepsilon_{D,j})\\
				&= & \ell_{\phi,j}(\cdot;x_{j},\xi_j)
			\end{array}
			$$
			for $x_j=x_{1 j}$.
			Using \eqref{eqcutin2st}, we see that this bound can actually be refined for two-stage stochastic programs as
			$$
			F(x,\xi_j)+h(x)
			\geq 
			\ell_{\phi,j}(x;x_j,\xi_j)+\varepsilon_{P,j}.
			$$
			We now turn our attention to the complexity analysis 
			of SHCPB.
			We can easily adapt the analysis
			of SHCPB to inexact SHCPB.
			Most details of the proofs are omitted
			since they are immediate
			adaptations of
			the proofs in the exact
			case. Lemma \ref{lem:101} now becomes
			
			\begin{lemma}\label{lem:101i}
				Let $ s_j $ denote $ s_j(x_j;\xi_j) $, and let
				\begin{align}
					\ell_{\phi,j}^\lam(\cdot;x_j,\xi_j)&=\ell_{\phi,j}(\cdot;x_j,\xi_j) + \frac{1}{2\lam}\|\cdot - x_{\ell_0}\|^2, \\
					% \Psi_j=F(x_{j},\xi_{j})+h(x_{j}), \quad
					\Psi_j^\lam&={\overline F}(x_{j},\xi_{j})+h(x_{j})+\frac{1}{2\lam}\|x_{j}-x_{\ell_0}\|^2,\\
					\delta_j&={\overline F}(x_{j+1},\xi_{j+1})-f(x_{j+1})+f(x_j)-{\overline F}(x_j,\xi_j). 
				\end{align}
				
				Then, the following statements hold:
				\begin{itemize}
					\item[a)] $\Phi_{j+1}^\lam=\tau \Phi_j^\lam + (1-\tau) \Psi_{j+1}^\lam$;
					\item [b)] 
					$\Psi_{j+1}^\lam - \ell_{\phi,j}^\lam(x_{j+1};x_j,\xi_j)
					\le \delta_j +\left(2M + \|s_j-f'(x_j)\|\right) \|x_{j+1}-x_j\| +  L \|x_{j+1}-x_j\|^2/2+\Delta_j+\Delta'_j$;
					\item [c)] $-2\Delta'_j \leq \E[\delta_j] \leq 2 \Delta'_j$, $\E[\|s_j-f'(x_j)\|^2]\le \sigma^2$;
					\item[d)] $\E[\Psi_{j+1}^\lam - \ell_{\phi,j}^\lam(x_{j+1};x_j,\xi_j)]
					\le  M^2+\sigma^2/2 + (L+3) \E[\|x_{j+1}-x_j\|^2]/2+3\Delta'_j+\Delta_j$.
				\end{itemize}
			\end{lemma}
			\begin{proof}
				a) This statement immediately follows from the definitions of $\Phi_{j+1}^\lam$ and $\Psi_j^\lam$ in inexact SHCPB.

				b) Using the definitions of $\ell_{\phi,j}^\lam(\cdot;x_j,\xi_j)$, $\Psi_{j+1}^\lam$, and $\delta_j$ 
				for inexact SHCPB, we have
				$$
				\begin{array}{l}
					\Psi_{j+1}^\lam - \ell_{\phi,j}^\lam(x_{j+1};x_j,\xi_j)
					= {\overline F}(x_{j+1},\xi_{j+1}) - {\overline F}(x_j,\xi_j) - \inner{s_j}{x_{j+1}-x_j}+\Delta_j+\Delta'_j \\
					= f(x_{j+1}) - f(x_j) -\inner{f'(x_j)}{x_{j+1}-x_j} + \delta_j -\inner{s_j - f'(x_j)}{x_{j+1}-x_j}+\Delta_j+\Delta'_j,
				\end{array}
				$$
				and we conclude as in the end
				of the proof of Lemma \ref{lem:101}-(b).
				
				c) We have
				$$
				\begin{array}{lcl}
					\delta_j & = & {\overline F}(x_{j+1},\xi_{j+1})-F(x_{j+1},\xi_{j+1})\\
					&  & +F(x_{j+1},\xi_{j+1})-f(x_{j+1})\\
					& & +f(x_j)-F(x_j,\xi_j)\\
					&& F(x_j,\xi_j)-{\overline F}(x_j,\xi_j)
				\end{array}
				$$
				and therefore 
				$$
				\mathbb{E}[\delta_j]=
				\mathbb{E}[({\overline F}-F)(x_{j+1},\xi_{j+1})
				+(F-{\overline F})(x_j,\xi_j)]
				$$
				and using (I2) we obtain
				$$
				-2\Delta'_j \leq
				\mathbb{E}[\delta_j]=
				\mathbb{E}[({\overline F}-F)(x_{j+1},\xi_{j+1})
				+(F-{\overline F})(x_j,\xi_j)] \leq 2\Delta'_j.
				$$
				The second part in c) is proved as in Lemma 
				\ref {lem:101}.
				
				d) The proof is similar
				to the proof of
				Lemma \ref {lem:101}-d)
				now using Lemma \ref{lem:101i}-b) and
				$\mathbb{E}[\delta_j] \leq 2\Delta'_j$
				instead of 
				$\mathbb{E}[\delta_j]=0$.
				
			\end{proof}

			Now Lemma \ref{lem:tj}
			becomes (the proof, which is similar to the proof of Lemma
			\ref{lem:tj}, is omitted):
			\begin{lemma}\label{lem:tj2}
				For every $j\ge 1$, define
				\begin{equation}
					t_j:=\Phi_j^\lam-\Gamma_j^\lam(x_j), \quad a_j:=\frac{[2M+\|s_j-f'(x_j)\|]^2}{8M^2 + 2\sigma^2}, \emph{ and } b_j:= a_j \frac{\varepsilon}{4}+\delta_j+\Delta_j+\Delta'_j,
				\end{equation}
				where 
				$\delta_j={\overline F}(x_{j+1},\xi_{j+1})-f(x_{j+1})+f(x_j)-{\overline F}(x_j,\xi_j)$.
				Then, the following statements hold for every $j\ge 1$:
				\begin{itemize}
					\item[a)] $t_{j+1} \le \tau t_j + (1-\tau) b_j$;
					\item[b)] $\E[a_j]\le 1$ \emph{and} $\E[b_j]\le \varepsilon/4+3\Delta'_j+\Delta_j$.
				\end{itemize}
			\end{lemma}
			Finally, in the null steps
			analysis, Proposition \ref{prop:rsk} becomes (the proof follows the proof of Proposition \ref{prop:rsk}
			now using $\mathbb{E}[b_j] \leq \varepsilon/4+3\Delta'_j+\Delta_j$):
			\begin{lemma}\label{lem:rsk2}
				Let $N_0=0$ and $N_k=N_{k-1}+n_k+1$ for every $k\ge 1$. Then we have $r_{N_k}=\mathbb{E}[t_{N_k}]\le \varepsilon/2
				+\displaystyle 
				\max_{N_{k-1}+1 \leq j \leq N_{k}-1} 3 \Delta'_j + \Delta_j$.
			\end{lemma}
			
			In the analysis of serious steps, Lemma
			\ref{lemfirstserious} is modified as follows:
			\begin{lemma}\label{lemseriousfirsti} We have for every
				$k \geq 1$:
				\begin{equation}\label{formyhatki}
					\hat y_k = \tau^{n_k+1} x_{N_{k-1}} + \tau^{n_k}(1-\tau) x_{N_{k-1}+1} + \cdots + \tau(1-\tau) x_{N_{k}-1} + (1-\tau) x_{N_k},
				\end{equation}
				\begin{equation}\label{ineq:weightedi}
					\E[\hat \Phi_k ]  \ge \E[\phi(\hat y_k)]-\tau^{n_k+1}\Delta'_{N_{k-1}}-
					(1-\tau)\sum_{i=0}^{n_k} \tau^{n_k-i}\Delta'_{1+i+N_{k-1}},
				\end{equation}
				\begin{equation}\label{ineq:Phi_ki}
					\Phi_j^\lam \ge \Phi_j, \quad \hat \Phi_k^\lam \ge \hat \Phi_k.
				\end{equation}
			\end{lemma}

			Lemma \ref{lem:stoch-outer} remaing unchanged because
			the linearizations $\ell_{\phi,j}(\cdot;x_j,\xi_j)$
			are still valid lower bounding affine functions
			for $h(\cdot)+F(\cdot,\xi_j)$:
			\begin{lemma}\label{lem:stoch-outer2}
				For every $z\in \dom h$ and $k\ge 1$, we have
				\[
				\E[\hat \Phi_k] - \phi(z) \le \frac{\varepsilon}{2} +\frac1{2\lam}\E[\|z-\hat x_{k-1}\|^2] -\frac{1+\lam \mu}{2\lam}\E[\|z-\hat x_k\|^2]. 
				\]
			\end{lemma}
			We obtain the following
			extension of Theorem \ref{cor:cmplx1}:
			\begin{theorem}\label{theocompl2}
				Define for every cycle $k$ the error term
				$$
				\theta_k=\displaystyle 
				\max_{N_{k-1}\leq j \leq N_k} \Delta'_j.
				$$  
				The following statements hold for inexact SHCPB when applied to \eqref{eq:ProbIntro}:
				\begin{itemize}
					\item [a)] for any given outer iteration limit $K\ge 1$, we have
					\begin{equation}\label{ineq:main2}
						\begin{array}{lcl}
							0 \leq  \E[\phi(\hat y_R)] - \phi^*  \le  \E[\hat \Phi_R] - \phi^* +\theta_R & \le & \displaystyle \frac{\varepsilon}{2} + \min\left \lbrace \frac{d_0^2}{2\lam K}, \frac{\mu d_0^2}{2\left[ (1+\lam\mu)^K-1 \right]} \right \rbrace +\theta_R.
						\end{array}
					\end{equation}
					\item[b)] 
					If the outer iteration limit $K$ is such that
					$K\geq K_0$ with
					\[
					K_0=\left \lfloor \min \left\lbrace \frac{2 d_0^2}{\lam \varepsilon}, \frac{1+\lam \mu}{\lam \mu} \log \left( \frac{2 \mu d_0^2}{\varepsilon} + 1 \right) \right\rbrace \right \rfloor +1,
					\]
					the inexact SHCPB method finds a solution
					satisfying
					$\E[{\hat \Phi}_R] - \phi^* \le \varepsilon/2$. Moreover, if 
					$\lim_{j \rightarrow +\infty} \Delta'_j =0$
					and for fixed ${\varepsilon}>0$ if 
					$j_0$ is such that for $j \geq j_0$ we have
					$\Delta'_j \leq {\varepsilon}/2$
					then if $K$ is such that $K\geq K_0$
					and $N_{K-1} \geq j_0$ we also have
					$\E[\phi(\hat y_R)] - \phi^*  \le  \varepsilon$.
				\end{itemize}
			\end{theorem} 
			\begin{proof} a) Using Lemma \ref{lemseriousfirsti}, we have
				$$
				\begin{array}{lcl}
					\E[\phi(\hat y_R)] & \leq & 
					\E[\hat \Phi_R] +
					\tau^{n_{R}+1}\Delta'_{N_{R-1}}+
					(1-\tau)\displaystyle  \sum_{i=0}^{n_R} \tau^{n_R-i}\Delta'_{1+i+N_{R-1}}\\
					& \leq &    E[\hat \Phi_R] +
					\theta_R(\tau^{n_{R}+1}+
					(1-\tau)\displaystyle \sum_{i=0}^{n_R} \tau^{n_R-i}) \\
					& = & E[\hat \Phi_R] +  \theta_R.
				\end{array}    
				$$  
				The end of the proof of a) and the proof
				of b) follows the proofs of 
				Theorem \ref{cor:cmplx1}-a),b).
			\end{proof}

			We deduce from Theorem 
			\ref{theocompl2} that
			that if $\Delta'_j$
			is bounded:
			$\Delta'_j \leq \Delta'$ for all $j$
			with $\Delta'$ finite then we have
			$$
			\phi^* \leq 
			\varliminf_{K \rightarrow +\infty}
			\mathbb{E}[\phi(\hat y_R)] 
			\leq \varlimsup_{K \rightarrow +\infty}
			\mathbb{E}[\phi(\hat y_R)]
			\leq \phi^* + \Delta'.  
			$$
			
			Finally we obtain the following complexity
			of inexact SHCPB:
			
			\begin{theorem}\label{complinexact}
				\vg{}{
					Let $0 \leq \Delta<+\infty$
					(resp., $0 \leq \Delta'<+\infty$)
					be an upper bound
					on the sequence of error terms $\Delta_j$ (resp., $\Delta_j'$).
					Let $\bar t =(1-\tau)\bar T_1
					+\tau \bar T_2$ now with
					$$
					\begin{array}{l}
						\bar T_1 = M^2+\frac{\sigma^2}2 + 2(L+3) \left[\max\{1,4\lam^2 L^2\}3d_0^2 + 8 \lam^2  M^2 + 2 \lam^2 \sigma^2\right]+3 \Delta' + \Delta  \mbox{ and}\\
						\bar T_2:=\lam [M_h^2 + \sigma^2 + 2\|f'(x_0)\|^2 + 8M^2+8(\sqrt{2}+1)MLd_0 + 2(2\sqrt{2}+3)d_0^2L^2]+\Delta +
						\Delta'.
					\end{array}
					$$
					Then the mean number of
					iterations of SHCPB before it
					computes an approximate optimal value
					$\E[{\hat \Phi}_R]$
					satisfying $\E[{\hat \Phi}_R] - \phi^* \le \varepsilon/2$ is} 
				\[
				{\cal O} \left(\left( \left[ 1+  \lam_\mu \left(L+\frac{16 M^2 + 4 \sigma^2}{\varepsilon} \right)\right]  \log\left( \frac{ \bar t}{\varepsilon}\right) + 1 \right) \left[\left \lfloor \min \left\lbrace \frac{2 d_0^2}{\lam \varepsilon}, \frac{1+\lam \mu}{\lam \mu} \log \left( \frac{2 \mu d_0^2}{\varepsilon} + 1 \right) \right\rbrace \right \rfloor  + 1\right] \right).
				\]
			\end{theorem}
			The proof of Theorem
			\ref{complinexact}  
			follows from Theorem
			\ref{theocompl2} and
			Lemma \ref{lem:rsk2}.
			In this proof, the bound
			$\bar t$
			on
			$\mathbb{E}[\bar t_k]$
			is easily obtained
			following the
			proof of the bound
			on $\mathbb{E}[\bar t_k]$
			given in Lemma \ref{lem:t1} in
			the Appendix for the exact case.
			In particular, we
			observe that in this proof
			we can still use
			Lemma A.1-b) of \cite{liang2020proximal} since Lemma \ref{lem:stoch-outer2}
			on which Lemma A.1-b) of \cite{liang2020proximal}
			is based holds both
			in the exact and inexact cases.
			
		}\fi

		\section{Proofs of main results in Sections~\ref{sec:SCPB1} and \ref{sec:SCPB2}}\label{sec:proofth2}
		
		This section contains three subsections.
		The first one presents some technical results that apply to the SCPB scheme regardless of how
		the index
		$j_k$ is chosen in step 2.
		The second and third ones are then devoted to
		the proofs of Theorems~\ref{thm:main1} and \ref{thm:main2}, respectively. 
		
		\subsection{Proofs of 
			some technical results}\label{subsec:proof}

		% It also presents an iteration complexity result for SCPB to obtain an $\varepsilon$-pair of \eqref{eq:ProbIntro} in expectation.
		
		% It consists of two parts.
		% The first one  establishes Proposition~\ref{prop:rsk} where the behavior of SCPB within a cycle is described.
		% The second one derives a
		% recursive formula about the outer iterations which plays an important role in the proof of Theorem~\ref{cor:cmplx1} given at the end of this subsection.
		% \vg{We now introduce some notation that will be used throughout the analysis of this paper.}{}
		
		We assume Assumption \ref{assump:basic} holds throughout this subsection.
		Recall that
		for every $j \geq 0$
		\[
		\xi_{[j]}=(\xi_0,\xi_1,\ldots,\xi_{j})
		\]
		and for $p \leq q$ positive
		integers we denote by
		$\xi_{[p:q]}$ the portion
		$\xi_{[p:q]}=(\xi_p,\xi_{p+1},\ldots,\xi_q)$
		of realizations of the r.v.\ $\xi$
		over the iterations 
		$p,p+1,\ldots,q$.
		For convenience, in what follows we set 
		\begin{equation}\label{def:sj}
			s_j := s(x_j,\xi_j).
		\end{equation}
		For every $k \ge 1$ and $j \in {\cal C}_k$, define
		\begin{align} \label{def:u}
			u_j :=  \left\{\begin{array}{ll}
				\Phi(x_{i_k},\xi_{i_k}), & \text { if } j=i_{k} , \\ 
				(1-\tau) 
				\phi(x_{j}) + \tau u_{j-1},  & \text { otherwise},
			\end{array}\right.
		\end{align}	
		and
		\begin{align} \label{eq:Gamma}
			\Gamma_{j}(\cdot) :=  \left\{\begin{array}{ll}
				\tilde \ell_k (\cdot) + h(\cdot), & \text { if } j=i_{k} , \\ 
				(1-\tau)
				\ell(\cdot;x_{j-1},\xi_{j-1}) +  \tau \Gamma_{j-1}(\cdot) ,  & \text { otherwise},
			\end{array}\right.
		\end{align}	
		where $\Phi(\cdot,\xi)$ and $\ell(\cdot;x,\xi)$ are as in \eqref{def:Phi} and and $\tilde \ell_k (\cdot)$ is as in \eqref{def:tl}.
		It is easy to see from \eqref{eq:Sj}, \eqref{def:xj}, and the above definition of $\Gamma_j$ that
		\begin{equation}\label{eq:xj}
			x_{j} =\underset{u\in \R^n}\argmin
			\left\lbrace  
			\Gamma_{j}^\lam(u):=
			\Gamma_{j}(u) +\frac{1}{2\lam}\|u- x_{j}^c \|^2 \right\rbrace.  
		\end{equation}

		The first result below
		provides some basic relations which are often
		used in our analysis.
		
		\begin{lemma} \label{lemfirstserious} For every
			$j \geq 1$, we have
			\begin{align}
				\E[\Phi(x_j,\xi_j)] &= \E[\phi(x_j)], \label{eq:exp}\\
				\E[\phi( y_j)] &\le \E[u_j], \label{ineq:basic1} \\
				\E[\Gamma_j(x)]  &\le \phi(x) \quad \forall x \in \dom h. \label{ineq:basic2}
			\end{align}
		\end{lemma}
		
		\begin{proof}
			Observe that $x_j$
			is a function of $\xi_{[j-1]}$
			and not of $\xi_j$. Hence, $x_j$ is independent of
			$\xi_j$ in view of the fact that
			$\xi_j$ is chosen in step 1 of SCPB to be independent of $\xi_{[j-1]}$.
			Using the relation 
			$f(x)=\E[F(x,\xi)]$ (see (A2)), it follows that
			$$
			\begin{array}{lcl}
				\mathbb{E}[\Phi(x_j,\xi_j)]&=&
				\mathbb{E}_{
					\xi_{[j]}}
				[F(x_j,\xi_j)+h(x_j)]=
				\mathbb{E}_{\xi_{[j-1]}}[
				\mathbb{E}_{\xi_{j}}[
				F(x_j,\xi_j)+h(x_j)|\xi_{[j-1]}]]\\
				& = &\mathbb{E}_{\xi_{[j-1]}}[
				f(x_j)+h(x_j)] = \mathbb{E}[\phi(x_j)],
			\end{array}
			$$
			which is identity \eqref{eq:exp}. It then suffices to show that,
			for any given $k \ge 1$,
			\eqref{ineq:basic1} and \eqref{ineq:basic2} hold for every $j$ in the $k$-th cycle, i.e., $j \in \mathcal{C}_k$.
			We show this by induction on $j$ where
			$j$ is the iteration count. If $j=i_k$, then
			it follows from
			\eqref{def:yj}, \eqref{def:u}, and \eqref{eq:exp} that
			\[
			\E[u_j] \overset{\eqref{def:u}}{=}
			\E[\Phi(x_{j},\xi_{j})]
			\overset{\eqref{eq:exp}}{=}
			\E[\phi(x_{j})] 
			\overset{\eqref{def:yj}}{=} \E[\phi(y_j)],
			\]
			and from \eqref{eq:Gamma} with $j=i_k$, \eqref{def:tl}, and Assumption \ref{assump:basic} (A1)-(A2) that for every $x\in \dom h$,
			\[
			\E[\Gamma_j(x)] \overset{\eqref{eq:Gamma}}{=} \E[\tilde \ell_k(x)+h(x)] \overset{\eqref{def:tl}, (A2)}{=} f(x_{i_k-1}) + \inner{f'(x_{i_k-1})}{x-x_{i_k-1}} +h(x) \overset{(A1)}{\le}
			\phi(x).
			\]
			Let $j$ be such that $j>i_k$
			and \eqref{ineq:basic1} and \eqref{ineq:basic2} hold
			for $j$. Then, it follows
			from \eqref{def:yj}, \eqref{def:u}, the fact that \eqref{ineq:basic1} holds for $j$, and the convexity of $\phi$, that
			\[
			\E[u_{j+1}]  \overset{\eqref{def:u},\eqref{ineq:basic1}}{\ge}  (1-\tau) \E[\phi(x_{j+1})] + \tau \E[\phi(y_j)]
			\ge \E[\phi((1-\tau)x_{j+1} + \tau y_j)]  \overset{\eqref{def:yj}}{=} \E[\phi(y_{j+1})],
			\]
			and from \eqref{eq:exp0}, \eqref{eq:Gamma} and the fact that \eqref{ineq:basic2} holds for $j$, that
			\[
			\E[\Gamma_{j+1}(x)] \overset{\eqref{eq:Gamma}}{=} \tau \E[\Gamma_j(x)] + (1-\tau) \E[\ell(x;x_{j},\xi_{j})] 
			\overset{\eqref{eq:exp0}, \eqref{ineq:basic2}}{\le} \tau \phi(x) + (1-\tau) \phi(x) = \phi(x).
			\]
			We have thus shown that
			\eqref{ineq:basic1} and \eqref{ineq:basic2} hold for every $j \in {\cal C}_k$.
		\end{proof}

		It is worth noting that
		the proof of \eqref{ineq:basic2} is strongly based on the fact  that
		$\Gamma_j$ is a convex combination of affine functions whose expected values are underneath $\phi$. Moreover, this inequality would not necessarily be true if $\Gamma_j$ were
		for example the maximum of
		functions as just described.
		
		The next result provides a useful estimate for the quantity $\phi(x_j,\xi_j) - \ell(x_{j};x_{j-1},\xi_{j-1})$.
		
		\begin{lemma}\label{lem:101}
			% and for every $j \ge  2$, define
			% \begin{equation}\label{def:aj}
				%     b_{j}:= \frac{ (M + \|s_{j-1}\|)^2}{2\eta} +\Delta_j - \Delta_{j-1}, \quad\eta:= \frac{\tau}{\lam(1-\tau)}.
				% \end{equation}
			For every $j\in {\cal C}_k$ such that $j\ge i_k$, we have:
			% \begin{itemize}
				%     \item[a)] we have
				\begin{equation}\label{ineq:Phi}
					\phi(x_j) - \ell(x_{j};x_{j-1},\xi_{j-1})
					\le (\bar M + \|s_{j-1}\|) \|x_{j}-x_{j-1}\|.
				\end{equation}
				%     \item[b)] 
				%     $t_{j} \le \tau t_{j-1} + (1-\tau) b_j$.
				% \end{itemize}
		\end{lemma}
		
		\begin{proof}
			Using the definitions of $\phi$ and $\ell(\cdot;x,\xi)$ in \eqref{eq:ProbIntro} and \eqref{def:Phi}, respectively, we have
			\[
			\phi(x_j) - \ell(x_{j};x_{j-1},\xi_{j-1})
			= f(x_{j}) - f(x_{j-1}) - \inner{s_{j-1}}{x_{j}-x_{j-1}} \le \inner{f'(x_j) - s_{j-1}}{x_j-x_{j-1}} 
			\]
			where the inequality is due to the convexity of $f$.
			The above inequality, the Cauchy-Schwarz inequality, the triangle inequality and \eqref{ineq:fp} then imply \eqref{ineq:Phi}.
		\end{proof}

		The technical result below introduces a key quantity, namely, scalar $t_j$ below,  and
		provides a useful recursive 
		relation for it over the iterations of the
		$k$-th cycle. This recursive
		relation will then be used in Proposition \ref{prop:rsk}
		to show that the $t_{j}$ at the end of the $k$-th
		cycle, namely $t_{j_k}$, is relatively small in expectation.

		\begin{lemma}\label{lem:102}
			For every $j \ge 1$, define
			\begin{equation}
				% s_j := s(x_j;\xi_j), \quad \Delta_j:= %\Delta(x_{j},\xi_{j})=
				% \Phi(x_{j},\xi_{j})-\phi(x_{j}), \quad 
				t_j:=u_j-\Gamma_j^\lam(x_j), \quad
				b_{j+1}:= \frac{ \lam (\bar M^2 + \|s_{j}\|^2)}{\theta K}
				\label{def:delta}
			\end{equation}
			where $\lambda$, $\theta$, and $K$
			are as in step 0 of SCPB, and $s_j$ is as in \eqref{def:sj}.
			% and for every $j \ge  2$, define
			% \begin{equation}\label{def:aj}
				%     b_{j}:= \frac{ \lam (M^2 + \|s_{j-1}\|^2)}{\theta K+1} +\Delta_j - \Delta_{j-1}.
				% \quad\eta:= \frac{\tau}{\lam(1-\tau)} = \frac{\theta K+1}{\lam}.
				% \end{equation}
			Then,  for every $j\in {\cal C}_k$ such that $j\ge i_{k}+1$, we have
			\begin{equation}\label{ineq:tj}
				t_{j} \le \tau t_{j-1} + (1-\tau) b_j
			\end{equation}
			where $\tau$ is as in \eqref{def:tau},
			and hence
			% As a consequence, for every
			% $j \in {\cal C}_k$ such that
			% such that $j\ge i_k+1$, we have:
			\begin{equation} \label{ineq:recursive}
				t_j  \le 
				\tau^{j-i_k} t_{i_k} + (1-\tau)
				\sum_{i=i_{k}+1}^j \tau^{j-i} b_i.
			\end{equation}
		\end{lemma}

		% \begin{lemma}\label{lem:101}
			% For every $j \ge 1$, let
			% \begin{equation}
				%     s_j := s(x_j;\xi_j), \quad \Delta_j:= %\Delta(x_{j},\xi_{j})=
				%     \Phi(x_{j},\xi_{j})-\phi(x_{j}), \quad t_j:=u_j-\Gamma_j^\lam(x_j), \label{def:delta}
				% \end{equation}
			% and for every $j \ge  2$, define
			% \begin{equation}\label{def:aj}
				%     b_{j}:= \frac{ (M + \|s_{j-1}\|)^2}{2\eta} +\Delta_j - \Delta_{j-1}, \quad\eta:= \frac{\tau}{\lam(1-\tau)}.
				% \end{equation}
			% Then, the following statements hold for every $j\in {\cal C}_k$ satisfying $j\ge N_{k-1}+2$:
			%     \begin{itemize}
				%         \item[a)] we have
				%       \begin{equation}\label{ineq:Phi}
					%         \Phi(x_j,\xi_j) - \ell_\Phi(x_{j};x_{j-1},\xi_{j-1})
					%     \le \Delta_j - \Delta_{j-1}  + (M + \|s_{j-1}\|) \|x_{j}-x_{j-1}\|;
					%         \end{equation}
				%         \item[b)] 
				%         $t_{j} \le \tau t_{j-1} + (1-\tau) b_j$.
				%     \end{itemize}
			% \end{lemma}

		\begin{proof} 
			Let $j \in \mathcal{C}_k$ with
			$j \geq i_k+1$ be given.
			It follows from the definitions of $\Gamma_{j}$ and $\Gamma_{j}^\lam$ in \eqref{eq:Gamma} and \eqref{eq:xj}, respectively, that
			\begin{align}
				\Gamma_{j}^\lam(x_{j})& =  (1-\tau) \ell(x_{j};x_{j-1},\xi_{j-1}) + \tau \Gamma_{j-1}(x_{j})+\frac{1}{2\lambda}\|x_j-x_j^c\|^2 \nn \\
				& \ge (1-\tau) \ell(x_{j};x_{j-1},\xi_{j-1}) + \tau \left[\Gamma_{j-1}(x_{j})+ \frac{1}{2\lambda}\|x_j-x_{j-1}^c\|^2 \right]\nn \\
				& = (1-\tau) \ell(x_{j};x_{j-1},\xi_{j-1}) + \tau \Gamma_{j-1}^\lam(x_{j}) \nn \\
				&\ge (1-\tau) \ell(x_{j};x_{j-1},\xi_{j-1}) + \tau \left[\Gamma_{j-1}^\lam(x_{j-1}) + \frac{1}{2\lam}\|x_{j}-x_{j-1}\|^2\right], \label{ineq:Gamma}
			\end{align}
			where for the first inequality we used the fact
			that $\tau < 1$ and $x_j^c=x_{j-1}^c$
			for $j \in \mathcal{C}_k$ with $j \geq i_k+1$
			while for
			the second inequality is due to the facts that $\Gamma_j^\lam$ is $(1/\lam)$-strongly convex and 
			$x_{j-1}$ is the minimizer of $\Gamma_{j-1}^\lam$ (see \eqref{eq:xj}).
			Using \eqref{def:tau}, \eqref{ineq:Phi} and \eqref{ineq:Gamma}, 
			we have
			\begin{align*}
				& \Gamma_{j}^\lam(x_{j})  - \tau \Gamma_{j-1}^\lam(x_{j-1}) \overset{\eqref{def:tau}, \eqref{ineq:Gamma}}{\ge}  (1-\tau) \left[ 
				\ell(x_{j};x_{j-1},\xi_{j-1}) + \frac{\theta K}{2\lam} \|x_j-x_{j-1}\|^2 \right] \\
				&\overset{\eqref{ineq:Phi}}{\ge} (1-\tau)\phi(x_{j}) + (1-\tau) \left[\frac{\theta K}{2\lam} \|x_{j}-x_{j-1}\|^2 - (\bar M + \|s_{j-1}\|) \|x_{j}-x_{j-1}\|\right] \\
				&\ge  (1-\tau) \phi(x_j) - (1-\tau) \frac{ \lam(\bar M + \|s_{j-1}\|)^2}{2\theta K}
			\end{align*}
			where the last inequality
			is obtained by minimizing
			its left hand side with respect to $\|x_j-x_{j-1}\|$.
			The above inequality, the fact that $(\alpha_1+\alpha_2)^2\le 2\alpha_1^2+2\alpha_2^2$ for every $\alpha_1,\alpha_2\in \R$, and the definition of $b_j$ in \eqref{def:delta} imply that
			\[
			\Gamma_{j}^\lam(x_{j})  - \tau \Gamma_{j-1}^\lam(x_{j-1}) \ge  (1-\tau) \phi(x_j)  - (1-\tau) \frac{ \lam(\bar M^2 + \|s_{j-1}\|^2)}{\theta K}
			\overset{\eqref{def:delta}}{=} (1-\tau) \phi(x_j) - (1-\tau) b_{j}.
			\]
			Rearranging the above
			inequality and using the definition of $t_j$ in
			\eqref{def:delta}, identity
			\eqref{def:u}, and the fact that $j\ge i_k+1$, we then conclude that
			$$
			\begin{array}{lcl}
				\Gamma_{j}^\lam(x_{j}) + (1-\tau)b_j &\ge &
				\tau \Gamma_{j-1}^\lam(x_{j-1}) + (1-\tau) \phi(x_j) \\
				& \overset{\eqref{def:delta}}{=} &
				\tau (u_{j-1} - t_{j-1}) + (1-\tau) \phi(x_j)
				\overset{\eqref{def:u}}{=} u_j - \tau t_{j-1},
			\end{array}
			$$
			which, in view of the definition of
			$t_j$ in \eqref{def:delta}, implies \eqref{ineq:tj}. 
			Inequality \eqref{ineq:recursive} follows immediately  from \eqref{ineq:tj}
			and an induction argument.
		\end{proof}
		
		% ------
		
		% \begin{align*}
			% \tau \Gamma_1(x_1) + (1-\tau) \phi(x_2)
			% &= \tau u_1 + (1-\tau) \phi(x_2)- \tau [u_1-\Gamma_1(x_1)] \\
			% &=u_2 - \tau [u_1-\Gamma(x_1)]
			% \end{align*}
		% where
		% \[
		% \Gamma_1(x_1) := \Phi(x_0,\xi_0) + \inner{s_0}{x_1-x_0} + h(x_1) + \frac1{2\lam} \|x_1-x_0\|^2
		% \]
		% and
		% \[
		% \ell(x_i;x_{i-1},\xi_0) = \phi(x_{i-1}) + \inner{s_{i-1}}{x_i-x_{i-1}} + h(x_i) + \frac1{2\lam} \|x_i-x_{i-1}\|^2 \quad i \ge 2
		% \]
		% and
		% \[
		% u_1=\Phi(x_1,\xi_1)
		% \]
		% Set
		% \[
		% u_i = \tau u_{i-1} + (1-\tau) \phi(x_i) \quad \forall i \ge 2
		% \]
		% \begin{align} \label{eq:Gamma}
			% 	\Gamma_{j}(\cdot) :=  \left\{\begin{array}{ll}
				% 	\Phi(\cdot,\xi_0) + \inner{s_0}{\cdot-x_0} + h(\cdot), & \text { if } j=1 , \\ 
				% 		    (1-\tau)
				% 	    \ell(\cdot;x_{j-1},\xi_{j-1}) +  \tau \Gamma_{j-1}(\cdot) ,  & \text { otherwise};
				% 		    \end{array}\right.
			%     \end{align}	
		
		% ------
		
		The following technical result provides some useful bounds on $b_j$.

		\begin{lemma}\label{lem:tj}
			% For every $j\ge 1$, define
			%     \begin{equation}\label{def:tbj}
				%         \tilde b_{j}:=\frac{\lam(\bar M^2 + \|s_{j-1}\|^2)}{\theta K}.
				%     \end{equation}
			% The following statements hold:
			% \begin{itemize}
				%     \item [a)] for every $j\ge 1$ and $\ell \le j-1$, we have
				%     $$
				%         \begin{array}{l}
					%     \E[\Delta_j \, | \, \xi_{[\ell]} ] =0, \quad
					%     \E[\bar \Delta_{\ell+1} \, | \, \xi_{[\ell]} ] \leq \Delta, \quad
					% \mathbb{E}[\bar \Delta_{\ell+1}] \leq \Delta,    
					%     \label{condexp1} \\
					%         \E[\|s_j\|^2  \, | \, \xi_{[\ell]} ] \le M^2; \label{condexp2}
					%         \end{array}
				%         $$
				For every $\ell \ge 0$ and $j\ge \ell +2$, we have
				\begin{equation}\label{ineq:bj}
					% \E[\tilde b_j ] \le \frac{2\lam \bar M^2}{\theta K}, \quad 
					\E[b_j \, | \xi_{[\ell]} ] \le \frac{2\lam \bar M^2}{\theta K}, \quad \E[ b_j ] \le \frac{2\lam \bar M^2}{\theta K}.
				\end{equation}
			\end{lemma}
			\begin{proof}
				We first show that for every $j\ge 1$ and $\ell \le j-1$,
				\begin{equation}\label{ineq:basic}
					\E[\|s_j\|^2  \, | \, \xi_{[\ell]} ] \le \bar M^2.
				\end{equation}
				Fix $j\ge 1$. Since $x_j$ becomes deterministic when 
				$\xi_{[j-1]}$ is given,
				it follows from Assumption \ref{assump:basic} (A3) with $x=x_j$ and the definition of $s_j$ in \eqref{def:sj}
				that
				\[
				\E_{\xi_j}[\|s_j\|^2 \, | \, \xi_{[j-1]} ] \le \bar M^2.
				\]
				Now, if $\ell \le j-2 $, then
				the above relations together with the law of total expectation imply that
				% \[
				% \E[\Delta_j \, | \, \xi_{[\ell]} ]
				% =
				% \E_{\xi_{[\ell+1:j]}}
				% [\Delta_j \, | \, \xi_{[\ell]} ]
				% =
				% \E_{\xi_{[\ell+1:j-1]}} [\E_{\xi_j}[\Delta_j \, | \, \xi_{[j-1]}]]=0,
				% \]
				and
				\[
				\E[\|s_j\|^2 \, | \, \xi_{[\ell]} ]
				=
				\E_{\xi_{[\ell+1:j]}}
				[\|s_j\|^2 \, | \, \xi_{[\ell]} ]
				=
				\E_{\xi_{[\ell+1:j-1]}} [\E_{\xi_j}[\|s_j\|^2 \, | \, \xi_{[j-1]} ] 
				] \le \bar M^2.
				\]
				We have thus shown that 
				\eqref{ineq:basic} holds for any $\ell \le j-1$.

				The first inequality in \eqref{ineq:bj} then follows from the definition of $ b_j$ in \eqref{def:delta}. The second inequality in \eqref{ineq:bj} follows from the first one and the law of total expectation.
			\end{proof}
			
			The next technical result provides a
			bound on the initial $t_j$ for the $k$-th cycle, namely $t_{i_k}$, in expectation.
			
			\begin{lemma}\label{lem:t1-2}
				For every $k \ge 1$, we have
				$\E[t_{i_k}]\le 2 \min\{\lam \bar M^2, \bar M D_h\}$
				where $i_k$ and $t_j$ are as in \eqref{def:Ck} and \eqref{def:delta}, respectively.
			\end{lemma}
			\begin{proof} Let
				\begin{equation}\label{def:error}
					\Delta_j
					= \Phi(x_{j},\xi_{j})-\phi(x_{j})
					= F(x_{j},\xi_{j})-f(x_{j}).
				\end{equation}
				Using the definitions of $t_j$ and $\Gamma_j^\lam$ in \eqref{def:delta}  and \eqref{eq:xj}, respectively, \eqref{def:u} with $j=i_k=j_{k-1}+1$ (see \eqref{def:Ck}), we have
				\begin{align}
					t_{i_k} & \stackrel{\eqref{def:delta}}{=} u_{i_k} - \Gamma_{i_k}^{\lam}(x_{i_k}) \nn \\ &\stackrel{\eqref{def:u},\eqref{eq:Gamma}}{=} \Phi(x_{i_k},\xi_{i_k}) - \left[F(x_{j_{k-1}},\xi_{j_{k-1}}) + \inner{s_{j_{k-1}}}{x_{i_k}-x_{j_{k-1}}} + h(x_{i_k})\right] - \frac{1}{2\lam} \|x_{i_k}-x_{j_{k-1}}\|^2 \nn \\
					& = \Delta_{i_k} - \Delta_{j_{k-1}} + \phi(x_{i_k}) - \ell(x_{i_k};x_{j_{k-1}},\xi_{j_{k-1}}) - \frac{1}{2\lam} \|x_{i_k}-x_{j_{k-1}}\|^2 \nn \\
					&\le \Delta_{i_k} - \Delta_{j_{k-1}}
					+ \left(\bar M+\|s_{j_{k-1}}\| \right) \|x_{i_k}-x_{j_{k-1}}\| - \frac{1}{2\lam} \|x_{i_k}-x_{j_{k-1}}\|^2 \label{eq:tik}
				\end{align}
				where the inequality is due to Lemma \ref{lem:101}.
				Maximizing the right hand side
				of the last inequality above with
				respect to $\|x_{i_k}-x_{j_{k-1}}\|$
				and using the relation
				$(a+b)^2\le 2a^2+2b^2$ for every $a, b \in \R$, we obtain
				\begin{equation}
					t_{i_k} 
					\le \Delta_{i_k} - \Delta_{j_{k-1}} + \frac{\lam}2  \left(\bar M+\|s_{j_{k-1}}\| \right)^2 
					\le \Delta_{i_k} - \Delta_{j_{k-1}} + \lam \left(\bar M^2+\|s_{j_{k-1}}\|^2  \right). \label{ineq:t1}
				\end{equation}
				Moreover, \eqref{eq:tik} and the fact that $\|x_{i_k}-x_{j_{k-1}}\|\le D_h$ also imply that
				\begin{equation}\label{ineq:t1-1}
					t_{i_k} \le  \Delta_{i_k} - \Delta_{j_{k-1}}
					+ \left(\bar M+\|s_{j_{k-1}}\| \right) D_h.
				\end{equation}
				It follows from \eqref{def:sj}, \eqref{def:error}, and Assumption \ref{assump:basic} (A2) and (A3) that 
				\[
				\E[\Delta_{i_k}]=0, \quad \E[\Delta_{j_{k-1}}]=0, \quad \E[\|s_{j_{k-1}}\|^2] \le \bar M^2.
				\]
				Hence, the lemma follows by taking expectations of \eqref{ineq:t1} and \eqref{ineq:t1-1}
				and using the above three relations.
			\end{proof}

			We emphasize that
			all results developed in this subsection hold
			regardless of the way $j_k$ is chosen in
			step~2. On the other hand, the results in the following two subsections strongly use the fact that $j_k$ is chosen according to either (B1) or (B2).
			
			\subsection{Proof of 
				Theorem \ref{thm:main1}}\label{prooffirstth}
			
			This subsection is devoted to the proof of Theorem \ref{thm:main1}.
			The following result derives a bound in expectation
			for $t_{j_k}$  when $j_k$ is chosen according to cycle rule (B1).
			
			\begin{proposition}\label{prop:rsk}
				In addition to
				Assumption \ref{assump:basic}, assume also that (B1) holds.
				Then, for every $k \ge 1$, we have
				%   \begin{equation}\label{ineq:expectation}
					%       \mathbb{E}[t_{j_k}]\le \frac{C \bar T}{\lam k} + \frac{2\lam M^2}{\theta K}
					%   \end{equation}
				
				\begin{equation}\label{ineq:expectation}
					\mathbb{E}[t_{j_k}]\le \frac{2 R \min\{\lam \bar M^2, \bar M D_h\}}{\lam k} + \frac{2\lam \bar M^2}{\theta K} 
				\end{equation}
				where $t_j$ is as in \eqref{def:delta}.   
			\end{proposition}
			\begin{proof}
				Fix $k\ge 1$. 
				% %To simplify notation, let
				% $\ell := N_{k-1}+1$.
				It follows from cycle rule (B1)  and inequality \eqref{ineq:recursive} with $j=j_k$ that
				%     for every $j \in {\cal C}_k$ such that $j\ge \ell+1$,
				% 	\begin{equation}\label{ineq:middle}
					% 	    t_j \le \tau^{j-\ell} t_{\ell} + (1-\tau) \left[ \tau^{j-\ell-1} b_{\ell+1} + \ldots + \tau b_{j-1} + b_{j} \right].
					% 	\end{equation}
				% 	The above inequality with $j=N_k$ implies that
				\begin{equation}\label{ineq:tjk}
					t_{j_k} 
					\le \tau^{j_k-i_k}t_{i_k} + (1-\tau)
					\sum_{i=i_{k}+1}^{j_k} \tau^{j_k-i} b_i.
				\end{equation}    
				% = \frac{C t_{i_k}}{\lam k} + (1-\tau)
				% \sum_{i=i_{k}+1}^{j_k} \tau^{j_k-i} b_i.
				% 	Observe that
				% 	\[
				% 	j_k-i = (j_k-i_k) + (i_k-i)
				% 	\]
				% 	and 
				% 	\[
				% 	\tau^{j_k-i} \le \frac{a_k}{\lam k}\tau^{i_k-i} = \frac{a}{\lam k}\tau^{i_k-i}.
				% 	\]
				In view of (B1) and \eqref{def:Ck}, it follows that $j_k$ and $i_k$ are both deterministic.
				Hence, taking expectation of the above inequality and using the last inequality in \eqref{ineq:bj}, cycle rule (B1), and Lemma \ref{lem:t1-2},  we conclude that 
				\begin{align*}
					\E[t_{j_k}] 
					&\le \tau^{j_k-i_k} \E[t_{i_k}] + (1-\tau)
					\sum_{i=i_{k}+1}^{j_k} \tau^{j_k-i} \E[b_i]\\
					&\le \frac{2 R \min\{\lam \bar M^2, \bar M D_h\}}{\lam k}  + (1-\tau) \frac{2\lam \bar M^2}{\theta K}  \sum_{i=i_{k}+1}^{j_k} \tau^{j_k-i},
					% \le \frac{C \bar T}{\lam k} + (1-\tau^{j_k-i_k}) \frac{2\lam M^2}{\theta K},
				\end{align*}
				and hence that \eqref{ineq:expectation} holds.
			\end{proof}

			% \red{Add a comment on the intuition/role of (B1)}
			
			It is worth noting that rule (B1) plays an important role in showing that the expectation of the first term on the right-hand side of \eqref{ineq:tjk}
			is ${\cal O}(1/k)$.
			On the other hand, the proof of the ${\cal O}(1/K)$ bound for the expectation
			of the second term on the right-hand side of \eqref{ineq:tjk}  does not depend on rule (B1) but on the fact that  the expectation of
			$b_j$ is small, namely, ${\cal O}(1/K)$ (see \eqref{ineq:bj} and its definition in \eqref{def:delta}). In conclusion, rule (B1) provides a way to estimate the magnitude of $\E[t_{j_k}]$,
			a quantity
			which by itself
			is intractable to compute exactly.
			
			In the remaining part of this subsection, we analyze the behavior of
			the ``outer" sequence of iterations
			$\{\hat y_k\} = \{y_{j_k}\} \subset \R^n$
			generated in step 2 of SCPB.
			For this purpose, define 
			\begin{equation}\label{def:Gammak}
				\hat \Gamma_k:=\Gamma_{j_k} \quad \forall k\ge 1
			\end{equation}
			and
			\begin{equation}
				\hat x_k:=x_{j_k}, \quad \hat u_k:=u_{j_k}. 	      \label{not2}
			\end{equation}
			
			In what follows, we make some
			remarks about the above ``outer" sequences
			which follow as immediate consequences
			of the results developed above.
			In view of the above definitions, relation \eqref{eq:xj} with $j=j_k$,
			and the way the prox-centers
			$x_j^c$ are updated in \eqref{def:xc}, we have that
			\begin{equation}\label{def:zk}
				\hat x_k =\underset{x\in  \R^n}\argmin
				\left\lbrace \hat\Gamma_k(x) +\frac{1}{2\lam}\|x- \hat x_{k-1} \|^2 \right\rbrace \quad \forall k \ge 1.
			\end{equation}
			Moreover, it follows from \eqref{ineq:basic1} and \eqref{ineq:basic2} with $j=j_k$ that
			\begin{equation}\label{ineq:uk}
				\E[\phi(\hat y_k)] \le \E[\hat u_k]
			\end{equation}
			and
			\begin{equation}\label{ineq:Gammak}
				\E[\hat \Gamma_k(z)] \le \phi(z) \quad \forall z\in \dom h.
			\end{equation}
			
			The following result describes an important
			recursive formula for the outer sequence
			$\{\hat y_k\}$ generated by SCPB.
			
			\begin{lemma}\label{lem:stoch-outer}
				In addition to
				Assumption \ref{assump:basic}, assume also that (B1) holds.
				Then, for every $k \ge 1$ and $z\in \dom h$, we have
				\[
				\frac{2 R \min\{\lam \bar M^2, \bar M D_h\}}{\lam k} + \frac{2 \lam \bar M^2}{\theta K} +\frac{1}{2\lam}[d_{k-1}(z)]^2 - \frac{1}{2\lam}[d_{k}(z)]^2 \ge \E[\phi(\hat y_k)] - \phi(z)
				\]
				where
				\begin{equation}\label{def:dk}
					d_k(z) : =\left( \E[\|\hat x_k - z\|^2  ] \right) ^{1/2}.
				\end{equation}
			\end{lemma}
			\begin{proof}
				First observe that \eqref{def:Gammak}, \eqref{not2}, and the definitions of $\Gamma_j^\lam$ and $t_j$ in \eqref{eq:xj} and \eqref{def:delta}, respectively, imply that \eqref{ineq:expectation} is equivalent to
				\begin{equation}\label{ineq:key}
					\E\left[\hat u_k - \hat \Gamma_k(\hat x_k)-\frac{1}{2\lam} \|\hat x_k-\hat x_{k-1}\|^2\right] \le \frac{2 R \min\{\lam \bar M^2, \bar M D_h\}}{\lam k} + \frac{2 \lam \bar M^2}{\theta K}.
				\end{equation}
				It follows from \eqref{def:zk} and the fact that the objective function of \eqref{def:zk} is $(1/\lam)$-strongly convex that for every $z\in \dom h$,
				\[
				\hat \Gamma_k(\hat x_k) + \frac1{2\lam}\|\hat x_k-\hat x_{k-1}\|^2 
				\le \hat \Gamma_k(z) + \frac1{2\lam}\|z-\hat x_{k-1}\|^2
				- \frac{1}{2 \lam}\|z-\hat x_k\|^2,
				\]
				and hence that
				\[
				\hat u_k - \hat \Gamma_k(\hat x_k) -  \frac1{2\lam}\|\hat x_k-\hat x_{k-1}\|^2 +\frac1{2\lam}\|\hat x_{k-1}-z\|^2
				\ge \hat u_k - \hat \Gamma_k(z) + \frac{1 }{2 \lam}\|\hat x_k - z\|^2. 
				\]
				Taking expectation of the above inequality and using \eqref{def:dk} and \eqref{ineq:key}, we conclude that
				\[
				\frac{2 R \min\{\lam \bar M^2, \bar M D_h\}}{\lam k} + \frac{2\lam \bar M^2}{\theta K} +\frac1{2\lam}[d_{k-1}(z)]^2
				\ge \E[\hat u_k] - \E[\hat \Gamma_k(z)] + \frac{1 }{2 \lam}[d_k(z)]^2
				\]
				which, in view of \eqref{ineq:uk} and \eqref{ineq:Gammak},
				immediately implies the
				conclusion of the lemma.
			\end{proof}
			
			\vgap

			We are now in a position to prove
			Theorem \ref{thm:main1}.
			
			\vgap

			\noindent
			{\bf Proof of Theorem \ref{thm:main1}:}
			a) This statement directly follows from \eqref{def:tau}, cycle rule (B1), the definition of $\ln_0^+$, and the facts that $|{\cal C}_k| = j_k - i_k + 1$ and $\ln \tau^{-1} \ge 1-\tau$.

			% ----------Renato-----
			
			% \[
			% \tau^{j_k-i_k} \le \frac{C}{\lam k}
			% \]
			% \[
			% (j_k-i_k) \log \tau \le 
			% \log \frac{C}{\lam k}
			% \]
			% \[
			% (j_k-i_k) \log \tau^{-1} \ge 
			% \log \frac{\lam k}{C}
			% \]
			% which is implied by
			% \[
			% (j_k-i_k) (1-\tau) \ge 
			% \log \frac{\lam k}{C}
			% \]
			% \[
			% | {\cal C}_k| =
			% j_k-i_k+1 \ge 1 + (1-\tau)^{-1}
			% \log \frac{\lam k}{C}
			% = 1+ (\theta K +1) 
			% \log \frac{\lam k}{C}
			% \]

			% -------------Vincent-----
			
			% We have that if 
			% $$
			% n_0
			% =(\theta K+1)
			% \log_0^+ \left(\frac{\lambda k}{C} \right)
			% $$
			% where $\log_0^+(\cdot) = \max\{0,\ln(\cdot)\}$.
			% then
			% $$
			% \tau^{\lceil n_0 \rceil} \leq 
			% \tau^{n_0} \leq \frac{C}{\lambda k} 
			% $$
			% which implies that
			% $j_k-i_k \leq \lceil n_0 \rceil$.

			% \[
			% n_0 (1-\tau) \ge 
			% \log \frac{\lam k}{C}
			% \]
			% Then
			% \[
			% n_0 \ge (1-\tau)^{-1}
			% \log \frac{\lam k}{C}
			% = (\theta K +1) 
			% \log \frac{\lam k}{C}
			% \]
			% We have
			% $$
			% \tau^{\lceil n_0 \rceil}
			% \leq \tau^{n_0}
			% \leq \frac{C}{\lambda k}
			% $$
			% and therefore
			% $$
			% j_k-i_k+1 \leq 1+
			% $$
			% \[
			% \]
			% \[
			% \frac{t-1}t \le \log t \le t - 1
			% \]
			% So if
			% \[
			% \lam \le C/K
			% \]
			% then $n_0=0$ and
			% $| {\cal C}_k| =1$.
			% Take $C=\alpha \sqrt{K}$. Then the above condition reduces to
			% \[
			% \lam \le \alpha/\sqrt{K}
			% \]

			% --------------
			
			b)      Using the definition of 
			$\hat y_K^a$ in \eqref{eq:output},
			Lemma \ref{lem:stoch-outer} with $z=x^*\in X^*$,
			and the facts that
			$\lceil K/2\rceil \ge K/2$ and
			\[
			\sum_{k=\lfloor K/2 \rfloor + 1}^K \frac1k \le \int_{\lfloor K/2 \rfloor}^K \frac1x dx =\ln \frac{K}{\lfloor K/2 \rfloor} \le \ln \frac{K}{K/4} =\ln 4 \le \frac32 \quad \forall K\ge 2,
			\]
			we then conclude that
			for every $K\ge 2$,
			\begin{align}
				&\E[\phi(\hat y_K^a)] - \phi_* 
				\leq \frac1{\lceil K/2 \rceil}\sum_{k=\lfloor K/2 \rfloor + 1}^K (\E[\phi(\hat y_k)]- \phi_*) \nn \\
				&\le \frac1{\lceil K/2 \rceil}\sum_{k=\lfloor K/2 \rfloor + 1}^K \left(\frac{2 R \min\{\lam \bar M^2, \bar M D_h\}}{\lam k}+ \frac{2\lam \bar M^2}{\theta K} +\frac1{2\lam}[d_{k-1}(x^*)]^2 -\frac{1}{2\lam}[d_k(x^*)]^2 \right) \nn \\
				& \le \frac{6 R \min\{\lam \bar M^2, \bar M D_h\}}{\lam K} + \frac{2\lam \bar M^2}{\theta K} + \frac{\left[d_{\lfloor K/2 \rfloor}(x^*)\right]^2}{\lam K} \label{ineq:key1}
			\end{align}
			where the first inequality is due to the convexity of $\phi$.
			% where the last inequality is due to the facts that
			%         $\lceil K/2\rceil \ge K/2$ and
			%         \[
			%         \sum_{k=\lfloor K/2 \rfloor + 1}^K \frac1k \le \int_{\lfloor K/2 \rfloor}^K \frac1x dx =\ln \frac{K}{\lfloor K/2 \rfloor} \le \ln \frac{K}{K/4} =\ln 4 \le \frac32 \quad \forall K\ge 2.
			%         \]
			It is also easy to see from Lemma \ref{lem:stoch-outer} that \eqref{ineq:key1} holds for $K=1$.
			Then \eqref{ineq:consequence} follows from \eqref{ineq:key1} and the fact that $d_{\lfloor K/2 \rfloor}(x^*)\le D_h$.
			{\tiny \qed}

			The above result strongly uses the fact that $\dom h$ is bounded in view of the last inequality of its proof. 
			
			We end this subsection by
			describing a complexity bound for a slightly modified SCPB1 variant which is derived without assuming that $\dom h$ is bounded.
			We start by describing the
			two changes one needs to make to SCPB1 in order to obtain the aforementioned 
			variant.
			First, instead of  the
			point $\hat y_K^a$ as in \eqref{eq:output},
			it outputs
			\begin{equation}\label{def:bar y}
				\bar y_K^a = \frac1K\sum_{k=1}^K \hat y_k.
			\end{equation}
			Second, instead of computing $j_k$ as in
			(B1), it sets $j_k$
			as the smallest integer greater than or equal to $ i_k$ such that
			$\lam K \tau^{j_k-i_k}\le R$.

			% and the cycle rule is a slightly modified version of (B1), namely:
			% \begin{itemize}
				%     \item[(B1')] for every $k \ge 1$, $j_k$ is the smallest integer $\ge i_k$ such that
				%     $\lam K \tau^{j_k-i_k}\le R$.
				% \end{itemize}
			
			\begin{theorem}\label{thm:unbounded}
				Assume that Assumption \ref{assump:basic} holds and let $d_0$ denote the
				distance of the initial
				point $x_0$ to the optimal set $X^*$, i.e.,
				\begin{equation}\label{def:d0}
					d_0 := \|x_0-x_0^*\|, \ \ \mbox{\rm where} \ \ \ 
					x_0^* := \argmin \{\|x_0-x^*\|: x^*\in X^*\}.
				\end{equation}
				Then, the aforementioned SCPB1 variant  satisfies the following statements:
				\begin{itemize}
					\item[a)] the number of iterations within each cycle is bounded by 
					\[
					\left\lceil (\theta K+1)
					\ln_0^+ \left(\frac{\lambda K}{R} \right)\right\rceil + 1;
					\]
					\item[b)] there holds
					\[
					\E[\phi(\bar y_K^a)] - \phi_* \le  \frac1K \left(\frac{d_0^2}{2\lam} + 2 R \bar M^2 +  \frac{2\lam \bar M^2}{\theta}\right).
					\]
				\end{itemize}
			\end{theorem}
			
			\begin{proof}
				(a) The proof of (a) is similar to that of Theorem \ref{thm:main1}(a).
				
				(b) First note that the new way of choosing $j_k$ and slightly different arguments than the ones used in the proofs of Proposition \ref{prop:rsk} and Lemma \ref{lem:stoch-outer} imply that for every $z\in \dom h$,
				\[
				\frac{2 R \bar M^2}{K} + \frac{2 \lam \bar M^2}{\theta K} +\frac{1}{2\lam}[d_{k-1}(z)]^2 - \frac{1}{2\lam}[d_{k}(z)]^2 \ge \E[\phi(\hat y_k)] - \phi(z).
				\]
				Using the fact that $\phi$ is convex and the definition of 
				$\bar y_K^a$ in \eqref{def:bar y}, and summing the above inequality from $k=1$ to $K$, we conclude that for every $z \in \dom h$,
				\begin{align*}
					\E[\phi(\bar y_K^a)] - \phi(z) 
					& \leq \frac1{K}\sum_{k=1}^K [ \E[\phi(\hat y_k)]- \phi(z)] \\
					&\le \frac1K \sum_{k=1}^K \left(\frac{2 R \bar M^2}{ K}+ \frac{2\lam \bar M^2}{\theta K} +\frac1{2\lam}[d_{k-1}(z)]^2 -\frac{1}{2\lam}[d_k(z)]^2 \right) \\
					& \le \frac{2 R \bar M^2}{ K} + \frac{2\lam \bar M^2}{\theta K} + \frac{[d_0(z)]^2}{2\lam K}. 
				\end{align*}
				The statement now follows from the above inequality
				with $z=x_0^*$ where $x_0^*$ is as in \eqref{def:d0}.
			\end{proof}

			\subsection{Proof of Theorem \ref{thm:main2}}\label{subsec:proof2}

			The following result,
			which is an analogue of Proposition \ref{prop:rsk} with cycle rule (B1) replaced by (B2), derives a bound
			on  $t_{j_k}$ in expectation.
			
			\begin{proposition}\label{prop:rsk-1}
				In addition to
				Assumption \ref{assump:basic} holds, assume also that cycle rule (B2) is used. For every $k \ge 1$, we have
				\begin{equation}\label{ineq:rNk}
					\mathbb{E}[t_{j_k}]\le \frac{R}{\lam k} + \frac{2\lam \bar M^2}{\theta K} + \frac{2\lam \bar M^2}{\theta^2 K^2}.
				\end{equation}
			\end{proposition}
			\begin{proof}
				% Fix $k \ge 1$.
				% We first prove \eqref{ineq:Enk}.
				%     By \eqref{def:nk}, we have
				%     \[
				%     n_k \le \frac{1}{1-\tau} \ln\left(\frac{\bar t_k \lam k}{D^2}\right) + 2
				%     \]
				%     Taking expectation of
				%     the above inequality,
				%     and using 
				%     the Jensen's inequality and the fact that $\ln(x)$ is a concave function, we then conclude that
				%         \[
				%         \E[n_k] = \frac{1}{1-\tau} \E \left[ \ln\left(\frac{\bar t_k \lam k}{D^2}\right) \right] +2 \le \frac{1}{1-\tau} \ln\left(\frac{\E[\bar t_k] \lam k}{D^2}\right) + 2
				%         \le \frac{1}{1-\tau} \ln\left(\frac{ (2M+M_h) \lam k}{D}\right) + 2,
				%         \]
				% where the last inequality is
				% due to Lemma \ref{lem:t1-1}.
				% To simplify notation, let
				% $\ell := N_{k-1}+1$.
				% Note that $t_\ell = \bar t_k$ where $\bar t_k$ is given by \eqref{def:ak2}. 
				Using \eqref{ineq:recursive} with $j=j_k$, 
				% \eqref{ineq:example}, recalling that $j_k \ge i_k+1$
				% for cycle rule (B2),
				% and  that $\tau\in (0,1)$ in view of \eqref{def:tau},
				we conclude that 
				% 	\begin{align}
					% 	    t_{j_k} - (1-\tau) \sum_{i=i_k+2}^{j_k} \tau^{j_k-i} b_i 
					% 	    & \le \tau^{j_k-i_k} t_{i_k} + (1-\tau) \tau^{j_k-i_k-1} |b_{i_k+1}|   \nn \\
					% 	    & \le \tau^{j_k-i_k} t_{i_k} +  (1-\tau) |b_{i_k+1}|. \label{ineq:tj-det}
					% 	\end{align}
				\begin{equation}\label{ineq:tjk2}
					t_{j_k} - (1-\tau) \sum_{i=i_k+2}^{j_k} \tau^{j_k-i} b_i 
					\overset{\eqref{ineq:recursive}}{\le} \tau^{j_k-i_k} t_{i_k} + (1-\tau) \tau^{j_k-i_k-1} b_{i_k+1}  
					\overset{(B2)}{\le} \frac{R}{\lam k} + (1-\tau) b_{i_k+1}
				\end{equation}
				where the second inequality is due to cycle rule (B2), the observation that (B2) is equivalent to 
				$\lam k \tau^{j_k-i_k} t_{i_k} \le R$, and $\tau\in (0,1)$ in view of \eqref{def:tau}.
				% 	Letting $j=j_k$ in the above inequality, and using \eqref{ineq:nk}, \eqref{def:ak2} and the fact that $i_k=j_{k-1}+1$ (see \eqref{def:Ck}),
				% 	we have
				% 	\[
				% 	    t_{j_k} -(1-\tau) |b_{i_k+1}| - (1-\tau)\sum_{i=i_k+2}^{j_k} \tau^{j_k-i} b_i \le
				% 	\tau^{j_k-i_k} t_{i_k} 
				% 	\le
				% 	\frac{C}{\lam k}.
				% 	\]
				% 	Let
				% 	\[
				% 	\bar b_j = \E[ b_j | \xi_{[\ell]}] \quad \forall j \in {\cal C}_k,
				% 	\]
				% 	and note that Lemma \ref{lem:tj}(a) implies that
				% 	\begin{equation}\label{ineq:barb}
					% 	    \E[ \bar b_j ] \le \frac{2M^2}{\eta} \quad j \in {\cal C}_k.
					% 	\end{equation}
				Noting that $j_k$ becomes
				deterministic once $\xi_{[i_k]}$ is given, taking expectation of
				the above inequality conditioned on $\xi_{[i_k]}$, rearranging the terms,
				and using the first inequality in \eqref{ineq:bj}, we have
				\begin{align*}
					\mathbb{E}\left[t_{j_k}|\xi_{[i_k]}\right]
					- \frac{R}{\lam k} - (1-\tau)\E[b_{i_k+1} |\xi_{[i_k]}]  &\le (1-\tau)\sum_{i=i_k+2}^{j_k} \tau^{j_k-i} \E[ b_i |\xi_{[i_k]}] \\ 
					&\overset{\eqref{ineq:bj}}{\le} (1-\tau) \left(\sum_{i=i_k+2}^{j_k} \tau^{j_k-i} \right) \frac{2\lam \bar M^2}{\theta K}
					\le \frac{2\lam \bar M^2}{\theta K}.
				\end{align*}
				Taking expectation of the above inequality
				with respect to $\xi_{[i_k]}$, rearranging the terms, and using the second inequality \eqref{ineq:bj} and
				the fact that $1-\tau\le 1/(\theta K)$ by \eqref{def:tau}, we conclude that
				\begin{align*}
					\E[t_{j_k}] & \le \frac{R}{\lam k} + (1-\tau)\E[b_{i_k+1} ] + \frac{2\lam \bar M^2}{\theta K} \\
					& \le \frac{R}{\lam k} + \frac1{\theta K} \frac{2\lam \bar M^2}{\theta K} + \frac{2\lam \bar M^2}{\theta K},
					% 	&\le \frac{C}{\lam k} + \frac{2 \sigma}{\theta K} + \frac{2\lam M^2}{\theta^2 K^2} + \frac{2\lam M^2}{\theta K}
				\end{align*}
				and hence that \eqref{ineq:rNk} holds.
				% in view of the assumption that $\theta K \ge 1$.
			\end{proof}
			
			It is worth noting that rule (B2) plays an important role in showing that the first term on the right-hand side of the \eqref{ineq:tjk2}
			is ${\cal O}(1/k)$.
			
			% In view of the definitions of $\Gamma_j^\lam$ and $t_j$ in \eqref{def:xj} and \eqref{def:delta}, respectively, we observe that \eqref{ineq:rNk} is equivalent to
			% 	\begin{equation}\label{ineq:key-1}
				% 	   \E\left[\hat u_k - \hat \Gamma_k(\hat x_k)-\frac{1}{2\lam} \|\hat x_k-\hat x_{k-1}\|^2\right] \le \frac{C}{\lam k} + \frac{2 \sigma + 2\lam M^2}{\theta K} + \frac{2\lam M^2}{\theta^2 K^2}.
				% 	\end{equation}
			
			The following result is an analogue of Lemma \ref{lem:stoch-outer}
			with (B1) replaced by (B2).
			
			\begin{lemma}\label{lem:stoch-outer-1}
				In addition to
				Assumption \ref{assump:basic}, assume also that cycle rule (B2) is used.
				%   and $\theta K \ge 1$.
				Then, for every $z\in \dom h$ and $k\ge 1$, we have
				\[
				\frac{R}{\lam k} + \frac{2\lam \bar M^2}{\theta K}+\frac{2\lam \bar M^2}{\theta^2 K^2} +\frac{1}{2\lam}[d_{k-1}(z)]^2 - \frac{1}{2\lam}[d_{k}(z)]^2 \ge \E[\phi(\hat y_k)] - \phi(z)
				\]
				where $d_k(z)$ is as in \eqref{def:dk}.
			\end{lemma}
			\begin{proof}
				First observe that 
				the definitions of $\Gamma_j^\lam$ and $t_j$ in \eqref{eq:xj} and \eqref{def:delta}, respectively,
				imply that
				\eqref{ineq:rNk} is equivalent to
				\begin{equation}\label{ineq:key-1}
					\E\left[\hat u_k - \hat \Gamma_k(\hat x_k)-\frac{1}{2\lam} \|\hat x_k-\hat x_{k-1}\|^2\right] \le \frac{R}{\lam k} + \frac{2\lam \bar M^2}{\theta K} + \frac{2\lam \bar M^2}{\theta^2 K^2}.
				\end{equation}
				The remaining part of the proof is now
				similar to that of Lemma \ref{lem:stoch-outer} except that
				\eqref{ineq:key-1} is used in place of \eqref{ineq:key}.
			\end{proof}
			
			\vgap 
			
			We are now ready to prove
			Theorem \ref{thm:main2}.
			
			\vgap

			\noindent
			{\bf Proof of Theorem \ref{thm:main2}:}
			a) Using \eqref{def:tau}, \eqref{ineq:example}, the definition of $\ln_0^+$, and the facts that $|{\cal C}_k| = j_k - i_k + 1$ and $\ln \tau^{-1} \ge 1-\tau$, we have
			\[
			|{\cal C}_k| \le \frac{1}{1-\tau} \ln_0^+\left(\frac{t_{i_k} \lam k}{R}\right) + 1 = (\theta K +1) \ln_0^+\left(\frac{t_{i_k} \lam k}{R}\right)  + 1.
			\]
			Taking expectation of
			the above inequality,
			and using 
			the Jensen's inequality and the fact that $\ln x$ is a concave function, we then conclude that
			\begin{align*}
				\E[|{\cal C}_k|] &\le (\theta K +1) \E \left[ \ln_0^+\left(\frac{t_{i_k} \lam k}{R}\right) \right] +1 \le (\theta K +1) \ln_0^+\left(\frac{\E[t_{i_k}] \lam k}{R}\right) + 1\\
				&\le (\theta K +1) \ln_0^+\left(\frac{ 2\bar M^2 \lam^2 k}{R}\right) + 1,
			\end{align*}
			where the last inequality is
			due to Lemma \ref{lem:t1-2}.
			
			b) This statement follows from the same argument as in the proof of Theorem \ref{thm:main1}(b) except that Lemma \ref{lem:stoch-outer-1} is used in place of Lemma \ref{lem:stoch-outer}.
			{\tiny \qed}

			\section{Numerical experiments}\label{sec:num}
			
			In this section, we report the results of numerical experiments
			where we compare the performance of RSA
			and our two variants of SCPB
			on three stochastic programming problems, namely: a stochastic
			utility problem given in Section 4.2 of \cite{nemjudlannem09}
			and the two two-stage nonlinear stochastic programs
			considered in the numerical experiments of \cite{guiguesinexactsmd}.
			These three problems are of form
			\eqref{eq:ProbIntro}-\eqref{pbint2} with  $h$ 
			the indicator function of a convex
			compact set
			$X$ with diameter $D_X$. Therefore, the problems can be written as
			\begin{equation} \label{testpb}
				\min \{f(x):=\mathbb{E}[F(x,\xi)]: x \in X\}.
			\end{equation}
			The implementations are
			coded in MATLAB, using Mosek optimization library \cite{mosek}
			to generate stochastic oracles $F(x,\xi)$ 
			and $s(x,\xi)$, and run on a laptop with
			Intel i7, 1.80 GHz processor.
			% However, when needed (for all methods), 
			For solving subproblem \eqref{def:xj}, we do not use Mosek but implement algorithms for projection onto $X$. In particular, we follow \cite{wang2013projection} to implement an exact algorithm for projection onto the unit simplex.
			% the (efficient) algorithm from
			%  to project onto the simplex, when it is needed (for all methods).
			
			\par {\textbf{Parameters for Robust Stochastic Approximation.}}
			Robust Stochastic Approximation, denoted by 
			E-SA (Euclidean Stochastic Approximation) in what follows, is described in Section 2.2
			of \cite{nemjudlannem09} (as explained in \cite{nemjudlannem09},
			in terms of Section 2.3 of \cite{nemjudlannem09},
			this is mirror descent robust SA with Euclidean setup).
			In the notation of \cite{nemjudlannem09}, for E-SA run for $N$ iterations, we output
			$\displaystyle \tilde x_1^N=\frac{1}{N} \sum_{i=1}^N x_i$
			(this is $\tilde x_i^N$
			given by \eqref{def:tx} with $i=1$
			and corresponds to the usual output
			of RSA) where $x_i$
			is computed at iteration $i$ taking the constant
			steps given in (2.23) of \cite{nemjudlannem09}
			by 
			\[
			\gamma_t=\frac{\theta D_X}{M \sqrt{N}}
			\]
			where $D_X$ is the diameter of the feasible set $X$
			in \eqref{testpb}.\footnote{Parameter $M$ is denoted
				by $M_*$ in \cite{nemjudlannem09}.}
			As in \cite{nemjudlannem09}, we take $\theta=0.1$
			which was calibrated in \cite{nemjudlannem09}
			using an instance of the stochastic 
			utility problem. For each problem, the value of $M$ is estimated as
			in \cite{nemjudlannem09} 
			taking the maximum of $\|s(\cdot,\cdot)\|$
			over 10,000 calls to the stochastic oracle at randomly
			generated feasible solutions.\\
			
			\begin{remark} 
				In  \cite{nemjudlannem09}, E-SA generates
				approximately
				$\log_2(N)$ candidate solutions 
				$\tilde x_i^N=\frac{1}{N-i+1}\sum_{k=i}^N x_k$
				with 
				$N-i+1 = \min[2^k,N]$, $k = 0, 1,\ldots,\log_2(N)$
				and an additional sample was used
				to estimate the objective at these candidate solutions in order to choose the best of
				these candidates.
				In \cite{nemjudlannem09}, the computational effort required by this postprocessing is not reflected in the
				experiments. However, we believe that
				for a fair comparison of 
				E-SA using this set of candidate solutions and SCPB, this computational
				effort should be taken into account
				and  without this additional computational bulk,
				SCPB is already faster than E-SA in our experiments.
			\end{remark}
			\vspace*{0.2cm}
			\par {\textbf{Parameters for SCPB1.}}
			SCPB1 uses
			parameters
			$\theta$, $\tau$, 
			$R$, and $\lambda$  given 
			by
			$$
			\theta = \frac{C}{K},
			\;\tau=\frac{\theta K}{\theta K +1},\;
			R=\frac{D_X}{M},\;
			\lambda=\beta_1 \frac{\sqrt{C} D_X}{M \sqrt{K}}
			$$
			where constant
			$C=9$ and
			constant $\beta_1$
			was calibrated with
			the stochastic utility problem,
			see below.
			We take $\beta_1=10$ in
			all our experiments.
			Constant $M$ was estimated
			as for RSA 
			taking the maximum of $\|s(\cdot,\cdot)\|$
			over 10,000 calls to the stochastic oracle at randomly
			generated feasible solutions.\\
			
			\par {\textbf{Parameters for SCPB2.}} 
			SCPB2 uses
			parameters
			$\theta$, $\tau$, 
			$R$, and $\lambda$ given by
			$$
			\theta = \frac{C}{K},
			\;\tau=\frac{\theta K}{\theta K +1},\;
			R=D_X^2,\;
			\lambda=\beta_2 \frac{\sqrt{C} D_X}{M \sqrt{K}}
			$$
			where constant
			$C=9$ and
			constant $\beta_2$
			was calibrated with
			the stochastic utility problem,
			see below.
			We take $\beta_2=10$ in
			all our experiments.
			Constant $M$ was again  estimated
			as for RSA 
			taking the maximum of $\|s(\cdot,\cdot)\|$
			over 10,000 calls to the stochastic oracle at randomly
			generated feasible solutions.\\
			
			\par {\textbf{Notation in the tables.}} In what
			follows, we denote by
			\begin{itemize}
				\item $n$ the design dimension of an instance;
				\item $N$ the sample size used to run the methods; this is also the number of iterations of E-SA;
				\item $K$ the number of SCPB outer iterations;
				\item {\tt{Obj}} the empirical mean 
				\begin{equation}\label{defempmean}
					{\hat F}_T(x):=\frac{1}{T}\sum_{i=1}^T F(x,\xi_i)
				\end{equation}
				of $F$ at $x$ based on a sample 
				$\xi_1,\ldots,\xi_T$
				of $\xi$ of size $T$, which provides  an estimation of $f(x)$. The 
				empirical means are computed
				with $x$ being
				the final iterate output
				by the algorithm and $T=10^4$;
				\item {\tt{CPU}} the CPU time in seconds.
			\end{itemize}
			
			\subsection{A stochastic utility problem}\label{subsec:utility}
			
			Our first set of experiments was carried out
			with the stochastic utility problem given by
			\[
			\min_{x \in X} \mathbb{E}\left[ 
			\phi\left(\displaystyle \sum_{i=1}^n \left(\frac{i}{n}+\xi_i \right)x_i \right)
			\right]
			\]
			where 
			\begin{equation}\label{simplexeprob}
				X=\left\{x \in \mathbb{R}^n :  \sum_{i=1}^n x_i = 1, \, x \ge 0 \right\},
			\end{equation}
			$\xi_i \sim \mathcal{N}(0,1)$ are independent and
			$\phi(t) =\max({v_1 + s_1 t,\ldots,v_m+s_m t)}$
			is piecewise convex
			with 10
			breakpoints, all
			located on $[0,1]$\footnote{Although the same problem
				class and a similar procedure to build 
				$\phi$ was used in the experiments
				of Section 4.2 in \cite{nemjudlannem09},
				we could not find in this reference
				the precise choices of $v_k, s_k$
				and the
				optimal values of our instances
				differ from the optimal values
				of the instances in \cite{nemjudlannem09}.
				Also, contrary to
				\cite{nemjudlannem09},
				we use the same function
				$\phi$ for all instances.
				The instances differ for the
				problem dimension $n$.
			}.
			%We consider three instances:
			%$L_1$, $L_2$, and $L_3$
			%which have respectively
			%problem dimension
			%$2000$, $5000$, and $10^5$.\\
			
			{\textbf{Calibration of $\beta_1$ and $\beta_2$.}}
			We run SCPB1 and SCPB2
			with 7 values of 
			$\beta_1$ and $\beta_2$
			on four instances
			of the stochastic utility
			problem for $K=1000$
			outer iterations (i.e., cycles)
			and $n=500$, $n=1000$, $n=2000$, and
			$n=5000$. For this experiment, the 
			values of $\beta_1$, $\beta_2$,
			the corresponding values
			of stepsize $\lambda$, and the
			optimal values
			computed by SCPB1 and 
			SCPB2 are reported
			in Table \ref{table_choicebeta}.
			We found out  that $\beta_1=10$
			slightly
			outperforms other choices of $\beta_1$ for SCPB1. 
			Surprisingly, SCPB2 was not affected by changes in $\beta_2$
			and all tested values allowed us to obtain with similar CPU times
			a good approximate optimal value.
			This value $\beta_1=10$ and the same value $\beta_2=10$
			will be chosen for all runs
			of SCPB and all the problem
			instances (the stepsizes in 
			\cite{nemjudlannem09}
			were calibrated similarly, on the basis of an instance of the stochastic utility problem).
			\if{
				\begin{table}[H]
					\centering
					\begin{tabular}{|c|c|c|c|c|c|c|c|c|}
						\hline
						$\beta_1, \beta_2$  & 0.01 &0.1&1&10&50&150&1000\\
						\hline
						$\lambda, n=500$& $1.70\small{\times}10^{-5}$&$1.70\small{\times}10^{-4}$&0.0017 &0.017&0.09&0.26 &1.7\\
						\hline
						{\tt{Obj}}$_1$, $n=500$&14.2795&
						10.6439&10.1819&10.1811&10.1811&10.1838&10.1937\\
						\hline
						{\tt{Obj}}$_2$, $n=500$&10.1937&10.1937&10.1937&10.1937&10.1937&10.1937&10.1937\\
						\hline
						\hline
						$\lambda, n=10^3$& $1.17\small{\times}10^{-5}$&$1.17\small{\times}10^{-4}$&0.0012 &0.012&0.0585&0.18 &1.17\\
						\hline
						{\tt{Obj}}$_1$, $n=10^3$&14.6307&
						11.1325&10.0510&10.0504&10.0509&10.0523&10.0710\\
						\hline
						{\tt{Obj}}$_2$, $n=10^3$&10.0710&10.0710&10.0710&10.0710&10.0710&10.0710&10.0710\\
						\hline
						\hline
						$\lambda, n=2\small{\times}10^3$& $8.36\small{\times}10^{-6}$&$8.36\small{\times}10^{-5}$&
						$8.36\small{\times}10^{-4}$ &0.0084&0.0418&0.1255&0.8364\\
						\hline
						{\tt{Obj}}$_1$, $n=2\small{\times}10^3$&13.7451&
						11.0836&10.0365&10.0363&10.0364&10.0375&10.0613\\
						\hline
						{\tt{Obj}}$_2$, $n=2\small{\times}10^3$&10.0613&10.0613&10.0613&10.0613&10.0613&10.0613&10.0613\\
						\hline
						\hline
						$\lambda,n=5\small{\times}10^3$& $7.93\small{\times}10^{-6}$&$7.93\small{\times}10^{-5}$&
						$7.93\small{\times}10^{-4}$ &0.0079&0.0397&0.119&0.793\\
						\hline
						{\tt{Obj}}$_1$, $n=5\small{\times}10^3$&14.0830&11.3370
						&10.0228&10.0228&10.0231&10.0237&10.0540\\
						\hline
						{\tt{Obj}}$_2$, $n=5\small{\times}10^3$&10.0540 &10.0540&10.0540 &10.0540&10.0540&10.0540&10.0540\\
						\hline
						\hline
					\end{tabular}
					\caption{Selecting parameters $\beta_1$ and $\beta_2$
						of SCPB1 and SCPB2. Framework: SCPB, $K=1000$ outer iterations, four instances of the stochastic utility problem with $n=500$, $1000$, $2000$, and $5000$. {\tt{Obj}}$_1$
						(resp., {\tt{Obj}}$_2$)
						is the approximate
						optimal value with
						SCPB1 (resp., SCPB2).}
					\label{table_choicebeta}
				\end{table}
			}\fi

			\begin{table}[H]
				\centering
				\begin{tabular}{|c|c|c|c|c|c|c|c|c|}
					\hline
					$\beta_1, \beta_2$  & 0.01 &0.1&1&10&50&150&1000\\
					\hline
					$\lambda (n=500)$& $1.70\small{\times}10^{-5}$&$1.70\small{\times}10^{-4}$&0.0017 &0.017&0.09&0.26 &1.7\\
					\hline
					{\tt{Obj}}$_1$ &14.2795&
					10.6439&10.1819&10.1811&10.1811&10.1838&10.1937\\
					\hline
					{\tt{Obj}}$_2$ &10.1937&10.1937&10.1937&10.1937&10.1937&10.1937&10.1937\\
					\hline
					\hline
					$\lambda (n=10^3)$& $1.17\small{\times}10^{-5}$&$1.17\small{\times}10^{-4}$&0.0012 &0.012&0.0585&0.18 &1.17\\
					\hline
					{\tt{Obj}}$_1$ &14.6307&
					11.1325&10.0510&10.0504&10.0509&10.0523&10.0710\\
					\hline
					{\tt{Obj}}$_2$ &10.0710&10.0710&10.0710&10.0710&10.0710&10.0710&10.0710\\
					\hline
					\hline
					$\lambda (n=2\small{\times}10^3)$& $8.36\small{\times}10^{-6}$&$8.36\small{\times}10^{-5}$&
					$8.36\small{\times}10^{-4}$ &0.0084&0.0418&0.1255&0.8364\\
					\hline
					{\tt{Obj}}$_1$ &13.7451&
					11.0836&10.0365&10.0363&10.0364&10.0375&10.0613\\
					\hline
					{\tt{Obj}}$_2$ &10.0613&10.0613&10.0613&10.0613&10.0613&10.0613&10.0613\\
					\hline
					\hline
					$\lambda (n=5\small{\times}10^3)$& $7.93\small{\times}10^{-6}$&$7.93\small{\times}10^{-5}$&
					$7.93\small{\times}10^{-4}$ &0.0079&0.0397&0.119&0.793\\
					\hline
					{\tt{Obj}}$_1$ &14.0830&11.3370
					&10.0228&10.0228&10.0231&10.0237&10.0540\\
					\hline
					{\tt{Obj}}$_2$ &10.0540 &10.0540&10.0540 &10.0540&10.0540&10.0540&10.0540\\
					\hline
				\end{tabular}
				\caption{Selecting parameters $\beta_1$ and $\beta_2$
					of SCPB1 and SCPB2. Framework: SCPB, $K=1000$ outer iterations, four instances of the stochastic utility problem with $n=500$, $1000$, $2000$, and $5000$. {\tt{Obj}}$_1$
					(resp., {\tt{Obj}}$_2$)
					is the approximate
					optimal value with
					SCPB1 (resp., SCPB2).}
				\label{table_choicebeta}
			\end{table}

			We now run E-SA, SCPB1, and SCPB2 on three
			instances $L_1$, $L_2$, and $L_3$ of the stochastic utility 
			problem
			with  $n=2000$, $5000$, and $10^5$ respectively.
			For SCPB1 and SCPB2, we used $K=1000$ outer iterations.
			The results are reported in Table \ref{tabfirstpb:cpuobj}.
			Several comments are now in order for the results reported in this table.
			\begin{itemize}
				\item  For SCPB, approximate solutions can only be computed at the end of every cycle. Namely, at the end of $L$-th cycle
				at iteration $j_L$ we can compute
				the approximate solution
				$$ \frac1{\lceil L/2 \rceil}\sum_{\ell=\lfloor L/2 \rfloor + 1}^L \hat y_{\ell}
				=
				\frac1{\lceil L/2 \rceil}\sum_{\ell=\lfloor L/2 \rfloor + 1}^L y_{j_\ell}.
				$$
				For a given value of $N$ in Table 
				\ref{tabfirstpb:cpuobj}, the approximate
				objective value {\tt{Obj}} we report
				for E-SA is the empirical
				mean of $F(x,\xi)$
				at the approximate solution 
				$\frac{1}{N} \sum_{i=1}^N x_i$
				(where $x_i$'s are computed along iterations of E-SA)
				while for SCPB the approximate
				value {\tt{Obj}} is the empirical
				mean of $F(x,\xi)$
				at the approximate solution 
				$\frac1{\lceil L(N)/2 \rceil}\sum_{\ell=\lfloor L(N)/2 \rfloor + 1}^{L(N)} \hat y_{\ell}$
				where 
				$$
				L(N)=\min\{k:j_k \geq N\}
				$$
				(since a cycle may not end at iteration $N$).
				% Therefore, the total number of iterations
				% (outer plus inner) for a given value
				% of $N$ 
				\item Each iteration of
				E-SA and SCPB takes a similar
				amount of time (in both cases we evaluate
				an inexact prox-operator at some
				point) and therefore for a given
				sample size $N$ the CPU time 
				for E-SA and SCPB
				is similar.
				\item For all instances,
				SCPB computes a good approximate optimal value much quicker than
				E-SA and the decrease in the
				objective function value is
				much slower with E-SA.
				We also refer to
				Table \ref{tabpb:cpuobjr}
				which reports for $L_1$ and $L_2$
				the distance between
				SCPB approximate optimal value 
				and
				E-SA approximate value
				as a percentage of
				SCPB decrease in the objective
				for several sample sizes. 
				%This percentage is above 90\%
				%for almost all instances
				%and sample sizes.
				This table confirms the slower convergence of E-SA in these instances.
			\end{itemize}
			
			\if{
				\begin{table}[H]
					\centering
					\begin{tabular}{|c|c|cc|cc|cc|cc|cc|}
						\hline
						\multicolumn{2}{|c|}{-} & 
						\multicolumn{2}{|c|}{$L_1: n=500$}&
						\multicolumn{2}{|c|}{$L_2: n=1000$}&
						\multicolumn{2}{|c|}{$L_3: n=2000$}&
						\multicolumn{2}{|c|}{$L_4: n=5000$}&
						\multicolumn{2}{|c|}{{\color{blue}$L_5: n=10^4$}}\\
						\hline
						ALG. & N & {\tt{Obj}} &{\tt{CPU}} &{\tt{Obj}} &{\tt{CPU}} &{\tt{Obj}} &{\tt{CPU}} &{\tt{Obj}} &{\tt{CPU}}&
						{\color{blue}{\tt{Obj}}} &{\color{blue}{\tt{CPU}}}\\
						\hline
						\multirow[t]{10}{*}{E-SA} & 10&    14.6155 & 0.0012 &14.6113  & 0.109 &14.6449  & 0.001 &14.6892& 0.05 &{\color{blue}15.4}&{\color{blue}0.009}\\
						%5  &14.6169   &0.001 & 14.6126 & 0.107 & 14.6465 & 0.0007 & 14.6904 &0.04 \\
						%&10&    14.6155 & 0.0012 &14.6113  & 0.109 &14.6449  & 0.001 &14.6892& 0.05\\
						&50&    14.6074 & 0.0035 & 14.6039 & 0.11 & 14.6322 & 0.006 &14.6813 &0.07&{\color{blue}15.3}&{\color{blue}0.04}\\
						&100&    14.5982 & 0.0065 & 14.5950 & 0.12 & 14.6169 & 0.01  &14.6725 &0.1&{\color{blue}15.2}&{\color{blue}0.08}\\
						&200&    14.5814 & 0.0125 & 14.5786 &  0.13& 14.5880 & 0.03 & 14.6574&0.2&{\color{blue}15.2}&{\color{blue}0.17}\\
						&1000&    14.4682 & 0.0597 & 14.4651 & 0.2 & 14.3992 &0.1  & 14.5604&0.5&{\color{blue}15.0}&{\color{blue}0.8}\\
						&$10^4$&12.43  & 0.57  & 12.34 & 1.23 &12.9656  & 1.28 & 12.7410&3.7&{\color{blue}13.2}&{\color{blue}8.4}\\
						\hline
						\multirow[t]{10}{*}{SCPB1} & 10&12.3639 & 0.0019 & 12.8431& 0.002& 13.9539 & 0.003 & 13.7763 &0.008& {\color{blue}14.3}& {\color{blue}0.12}\\
						%5 &12.8535 & 0.001 & 13.3252 & 0.0012 & 13.9843 &0.001  & 13.7750 &0.004\\
						%&10&12.3639 & 0.0019 & 12.8431& 0.002& 13.9539 & 0.003 & 13.7763 &0.008\\
						&50&11.4342&   0.0048 & 12.1323& 0.005 &13.6527 &0.01 & 13.4672 &0.02&{\color{blue}13.7}&{\color{blue}0.16}\\
						&100&10.9414& 0.0082 &  11.4249 & 0.01 &13.5986 & 0.02& 13.5346 &0.05&{\color{blue}12.8}&{\color{blue}0.22}\\
						&200&10.4741& 0.0145 &  10.7075 & 0.02 &13.5349 & 0.03& 13.4686 &0.08&{\color{blue}12.0}&{\color{blue}0.31}\\
						&1000&10.0824& 0.0653 & 10.0574 & 0.09 & 13.0370 & 0.2 & 12.8376 &0.5&{\color{blue}11.4}&{\color{blue}1.10}\\
						\hline
						\multirow[t]{10}{*}{SCPB2} & 10&12.5790&0.0015  & 11.2559 & 0.003 & 13.5968 & 0.002 & 13.7777 &0.01&{\color{blue}12.3}&{\color{blue}0.16}\\
						%5&13.1016&0.001  & 12.7594 & 0.002 & 13.7884 &0.001  &  13.8973&0.009\\
						%&10&12.5790&0.0015  & 11.2559 & 0.003 & 13.5968 & 0.002 & 13.7777 &0.01\\
						&50&10.1287& 0.0045 & 10.1430 & 0.01 & 12.9421 & 0.008 & 12.6959 &0.05&{\color{blue}11.7}&{\color{blue}0.7}\\
						&100&10.1180&0.0075  & 10.0799 &0.02  & 12.1317 & 0.02 & 11.7614 &0.09&{\color{blue}11.6}&{\color{blue}1.5}\\
						&200&10.0993& 0.0141 & 10.0656 & 0.04 & 11.3640 & 0.03 & 11.3698 &0.2&{\color{blue}11.5}&{\color{blue}3.4}\\
						&1000&10.0779& 0.0718 & 10.0569 & 0.14 & 11.5681 & 0.1 & 11.5572 &0.9&{\color{blue}11.4}&{\color{blue}25.4}\\
						\hline
					\end{tabular}
					\caption{E-SA versus two variants of SCPB on the stochastic utility problem run with
						$K=1000$ outer iterations.}
					\label{tabfirstpb:cpuobj}
				\end{table}
			}\fi

			\if{
				\begin{table}[H]
					\centering
					\begin{tabular}{|c|c|cc|cc|cc|cc|cc|}
						\hline
						\multicolumn{2}{|c|}{-} & 
						\multicolumn{2}{|c|}{$L_1: n=500$}&
						\multicolumn{2}{|c|}{$L_2: n=1000$}&
						\multicolumn{2}{|c|}{$L_3: n=2000$}&
						\multicolumn{2}{|c|}{$L_4: n=5000$}&
						\multicolumn{2}{|c|}{{\color{blue}$L_5: n=10^4$}}\\
						\hline
						ALG. & N & {\tt{Obj}} &{\tt{CPU}} &{\tt{Obj}} &{\tt{CPU}} &{\tt{Obj}} &{\tt{CPU}} &{\tt{Obj}} &{\tt{CPU}}&
						{\color{blue}{\tt{Obj}}} &{\color{blue}{\tt{CPU}}}\\
						\hline
						\multirow[t]{10}{*}{E-SA} & 10&    14.6155 & 0.0012 &14.6113  & 0.109 &14.6449  & 0.001 &14.6892& 0.05 &{\color{blue}15.4}&{\color{blue}0.009}\\
						%5  &14.6169   &0.001 & 14.6126 & 0.107 & 14.6465 & 0.0007 & 14.6904 &0.04 \\
						%&10&    14.6155 & 0.0012 &14.6113  & 0.109 &14.6449  & 0.001 &14.6892& 0.05\\
						&50&    14.6074 & 0.0035 & 14.6039 & 0.11 & 14.6322 & 0.006 &14.6813 &0.07&{\color{blue}15.3}&{\color{blue}0.04}\\
						&100&    14.5982 & 0.0065 & 14.5950 & 0.12 & 14.6169 & 0.01  &14.6725 &0.1&{\color{blue}15.2}&{\color{blue}0.08}\\
						&200&    14.5814 & 0.0125 & 14.5786 &  0.13& 14.5880 & 0.03 & 14.6574&0.2&{\color{blue}15.2}&{\color{blue}0.17}\\
						&1000&    14.4682 & 0.0597 & 14.4651 & 0.2 & 14.3992 &0.1  & 14.5604&0.5&{\color{blue}15.0}&{\color{blue}0.8}\\
						&$10^4$&12.43  & 0.57  & 12.34 & 1.23 &12.9656  & 1.28 & 12.7410&3.7&{\color{blue}13.2}&{\color{blue}8.4}\\
						\hline
						\multirow[t]{10}{*}{SCPB1} & 10&12.3639 & 0.0019 & 12.8431& 0.002& 13.9539 & 0.003 & 13.7763 &0.008& {\color{blue}14.3}& {\color{blue}0.12}\\
						%5 &12.8535 & 0.001 & 13.3252 & 0.0012 & 13.9843 &0.001  & 13.7750 &0.004\\
						%&10&12.3639 & 0.0019 & 12.8431& 0.002& 13.9539 & 0.003 & 13.7763 &0.008\\
						&50&11.4342&   0.0048 & 12.1323& 0.005 &13.6527 &0.01 & 13.4672 &0.02&{\color{blue}13.7}&{\color{blue}0.16}\\
						&100&10.9414& 0.0082 &  11.4249 & 0.01 &13.5986 & 0.02& 13.5346 &0.05&{\color{blue}12.8}&{\color{blue}0.22}\\
						&200&10.4741& 0.0145 &  10.7075 & 0.02 &13.5349 & 0.03& 13.4686 &0.08&{\color{blue}12.0}&{\color{blue}0.31}\\
						&1000&10.0824& 0.0653 & 10.0574 & 0.09 & 13.0370 & 0.2 & 12.8376 &0.5&{\color{blue}11.4}&{\color{blue}1.10}\\
						\hline
						\multirow[t]{10}{*}{SCPB2} & 10&12.5790&0.0015  & 11.2559 & 0.003 & 13.5968 & 0.002 & 13.7777 &0.01&{\color{blue}12.3}&{\color{blue}0.16}\\
						%5&13.1016&0.001  & 12.7594 & 0.002 & 13.7884 &0.001  &  13.8973&0.009\\
						%&10&12.5790&0.0015  & 11.2559 & 0.003 & 13.5968 & 0.002 & 13.7777 &0.01\\
						&50&10.1287& 0.0045 & 10.1430 & 0.01 & 12.9421 & 0.008 & 12.6959 &0.05&{\color{blue}11.7}&{\color{blue}0.7}\\
						&100&10.1180&0.0075  & 10.0799 &0.02  & 12.1317 & 0.02 & 11.7614 &0.09&{\color{blue}11.6}&{\color{blue}1.5}\\
						&200&10.0993& 0.0141 & 10.0656 & 0.04 & 11.3640 & 0.03 & 11.3698 &0.2&{\color{blue}11.5}&{\color{blue}3.4}\\
						&1000&10.0779& 0.0718 & 10.0569 & 0.14 & 11.5681 & 0.1 & 11.5572 &0.9&{\color{blue}11.4}&{\color{blue}25.4}\\
						\hline
					\end{tabular}
					\caption{E-SA versus two variants of SCPB on the stochastic utility problem run with
						$K=1000$ outer iterations.}
					\label{tabfirstpb:cpuobj}
				\end{table}
			}\fi

			\begin{table}[H]
				\centering
				\begin{tabular}{|c|c|cc|cc|cc|}
					\hline
					\multicolumn{2}{|c|}{-} &
					\multicolumn{2}{|c|}{$L_1: n=2000$}&
					\multicolumn{2}{|c|}{$L_2: n=5000$}&
					\multicolumn{2}{|c|}{ $L_3: n=10^5$}\\
					\hline
					ALG. & N  &{\tt{Obj}} &{\tt{CPU}} &{\tt{Obj}} &{\tt{CPU}}&
					{\tt{Obj}} & {\tt{CPU}}\\
					\hline
					\multirow[t]{10}{*}{E-SA} & 10& 14.6449  & 0.001 &14.6892& 0.05 &15.4& 0.05\\
					%5  &14.6169   &0.001 & 14.6126 & 0.107 & 14.6465 & 0.0007 & 14.6904 &0.04 \\
					%&10&    14.6155 & 0.0012 &14.6113  & 0.109 &14.6449  & 0.001 &14.6892& 0.05\\
					&50&   14.6322 & 0.006 &14.6813 &0.07&14.7& 0.35\\
					&100&   14.6169 & 0.01  &14.6725 &0.1&14.6&0.74\\
					&200&   14.5880 & 0.03 & 14.6574&0.2&14.6& 1.44\\
					&1000&  14.3992 &0.1  & 14.5604&0.5&14.3& 17.2\\
					&$10^4$&12.9656  & 1.28 & 12.7410&3.7&14.2& 80.3\\
					&$10^5$& -  & - & -&-&13.2& 860.1\\
					\hline
					\multirow[t]{10}{*}{SCPB1} & 10& 13.9539 & 0.003 & 13.7763 &0.008& 14.7& 0.08\\
					%5 &12.8535 & 0.001 & 13.3252 & 0.0012 & 13.9843 &0.001  & 13.7750 &0.004\\
					%&10&12.3639 & 0.0019 & 12.8431& 0.002& 13.9539 & 0.003 & 13.7763 &0.008\\
					&50&13.6527 &0.01 & 13.4672 &0.02&14.4& 0.39\\
					&100&13.5986 & 0.02& 13.5346 &0.05&14.2&0.9\\
					&200&13.5349 & 0.03& 13.4686 &0.08&14.3& 1.6\\
					&1000& 13.0370 & 0.2 & 12.8376 &1.6&14.2& 12.5\\
					&$10^4$& - & - & - & - &12.7& 72.3\\
					\hline
					\multirow[t]{10}{*}{SCPB2} &  10 & 13.5968 & 0.002 & 13.7777 &0.01&14.2& 0.06\\
					%5&13.1016&0.001  & 12.7594 & 0.002 & 13.7884 &0.001  &  13.8973&0.009\\
					%&10&12.5790&0.0015  & 11.2559 & 0.003 & 13.5968 & 0.002 & 13.7777 &0.01\\
					&50&12.9421 & 0.008 & 12.6959 &0.05&13.2& 0.7\\
					&100& 12.1317 & 0.02 & 11.7614 &0.09&12.2& 1.5\\
					&200& 11.3640 & 0.03 & 11.3698 &0.2&11.6& 3.4\\
					&1000& 11.5681 & 0.1 & 11.5572 &0.9&11.2& 25.4\\
					\hline
				\end{tabular}
				\caption{E-SA versus two variants of SCPB on the stochastic utility problem run with
					$K=1000$ outer iterations.}
				\label{tabfirstpb:cpuobj}
			\end{table}
			
			\subsection{A first two-stage stochastic program}\label{sec:sto1}
			
			Our second test problem is the
			nonlinear two-stage stochastic program 
			\begin{equation}\label{smdmodel11}
				\left\{ 
				\begin{array}{l}
					\min \;c^T x_1 + \mathbb{E}[\mathfrak{Q}(x_1,\xi)]\\
					x_1 \in \mathbb{R}^n : x_1 \geq 0, \sum_{i=1}^n x_1(i) = 1
				\end{array}
				\right.
			\end{equation}
			where the second stage recourse function is given by
			\begin{equation}\label{smdmodel12}
				\mathfrak{Q}(x_1 ,\xi)=\left\{ 
				\begin{array}{l}
					\displaystyle  \min_{x_2 \in \mathbb{R}^n} \;\frac{1}{2}\left( \begin{array}{c}x_1\\x_2\end{array} \right)^T \Big( \xi \xi^T + \gamma_0 I_{2 n} \Big)\left( \begin{array}{c}x_1\\x_2\end{array} \right) + \xi^T \left( \begin{array}{c}x_1\\x_2\end{array} \right)  \\
					s.t. \ \ \ x_2 \geq 0, \displaystyle \sum_{i=1}^n x_2(i) = 1.
				\end{array}
				\right.
			\end{equation}
			
			Problem \eqref{smdmodel11}-\eqref{smdmodel12} is 
			of form 
			\eqref{eq:ProbIntro}-\eqref{pbint2}
			where $F(x,\xi)=c^T x + \mathfrak{Q}(x,\xi)$
			with $\mathfrak{Q}$ given by \eqref{smdmodel12} 
			and where $h$ is the indicator function of
			set $X$ 
			where $X$  given by \eqref{simplexeprob}
			is the unit simplex.
			For problem \eqref{smdmodel11} we refer to 
			Lemma 2.1 in \cite{guiguessiopt2016}
			for the computation of stochastic subgradients
			$s(x,\xi)$. We take $\gamma_0=2$
			and
			consider a Gaussian random vector in $\mathbb{R}^{2n}$ for $\xi$.
			We consider two instances of  problem
			\eqref{smdmodel11}
			with $n=50$ and $n=100$.
			For each instance, the components of $\xi$ are independent with means and standard deviations 
			randomly generated in respectively intervals 
			$[5,25]$ and $[5,15]$.
			The components of $c$ are generated randomly in interval $[1,3]$.
			
			We run E-SA, SCPB1, and SCPB2 on our two
			instances $A_1$ and $A_2$
			with $n=50$ and  $n=100$, respectively.
			For SCPB1 and SCPB2, we used $K=1000$ outer iterations.
			The results are reported in Table \ref{tabsecondpb:cpuobj1}.
			The conclusions are similar to the experiments on
			the stochastic utility problem: SCPB computes a good approximate optimal value much quicker than
			E-SA and the decrease in the
			objective function value is
			much slower with E-SA.
			We again refer to
			Table \ref{tabpb:cpuobjr}
			which reports
			the distance between
			SCPB approximate optimal value 
			and
			E-SA approximate value
			as a percentage of
			SCPB decrease in the objective
			for several sample sizes. 
			This percentage is again above 90\%
			for almost all instances
			and sample sizes.
			
			\begin{table}[H]
				\centering
				\begin{tabular}{|c|c|cc|cc|}
					\hline
					\multicolumn{2}{|c|}{-} & 
					\multicolumn{2}{|c|}{$A_1: n=50$}&
					\multicolumn{2}{|c|}{$A_2: n=100$}\\
					\hline
					ALG. & $N$ & {\tt{Obj}} &{\tt{CPU}} &{\tt{Obj}} &{\tt{CPU}}\\
					\hline
					\multirow[t]{6}{*}{E-SA} & 10&24.3477  & 0.13&7.5134&0.5\\
					%5 & 24.3678 &0.07&7.5160&0.2\\
					%&10&24.3477  & 0.13&7.5134&0.5\\
					&50& 24.2378 &0.6 &7.5018&2.5\\
					&100& 24.0816  &1.2 &7.4868&5.0\\
					&200& 23.7947  & 3.0&7.4566&10.1\\
					&500& 22.9185  & 8.8&7.3790&25.9\\
					&1000& 21.5328  & 24.6&7.2587&  55.5\\ 
					&$2\small{\times}10^4$& 8.5482  & 377&5.1339&1282\\
					&$10^5$& 5.7358  & 1555.6&3.9193&6147\\
					\hline
					\multirow[t]{6}{*}{SCPB1} & 10& 11.5047 &0.2&3.0063&1.3\\
					%5 &12.4432  &0.1&3.2823&0.6\\
					%&10& 11.5047 &0.2&3.0063&1.3\\
					&50& 9.2959 & 0.6&2.7269&3.2\\
					&100&7.2031&1.5 &2.4914&6.9\\
					&200& 6.4626  &2.9 &2.2899&13.0\\
					&500&  5.3700 & 7.5&2.0635&39.2\\
					&1000& 5.0582  & 15.1&1.9609&70.4\\
					\hline
					\multirow[t]{6}{*}{SCPB2} & 10& 8.6325 &0.15 &3.3113&0.6\\
					%5 &9.3584  &0.10&5.2361&0.4\\
					%&10& 8.6325 &0.15 &3.3113&0.6\\
					&50& 7.8378 & 0.7&2.2478&3.2\\
					&100& 7.8602  & 1.5&2.1929&6.4\\
					&200& 6.5839  & 3.0&2.2913&13.4\\
					&500& 6.0361  & 7.4&1.9974&33.7\\
					&1000& 6.1989  & 14.9&1.8058&65.1\\
					\hline
				\end{tabular}
				\caption{E-SA versus two variants of SCPB on the 
					two-stage stochastic program \eqref{smdmodel11}-\eqref{smdmodel12}}
				\label{tabsecondpb:cpuobj1}
			\end{table}
			Table \ref{tabmlarge}
			reports
			the impact of overestimating 
			$M$ (taking $M$ 10 times the Monte Carlo estimation $\overline M_e$) and 
			underestimating 
			$M$
			(taking $M$ 10 times smaller than the Monte Carlo estimation $\overline M_e$).
			In this experiment, SCPB is essentially not affected by a
			bad estimation of $M$ while E-SA converge much slower when $M$
			is overestimated. 
			Additionally, Table \ref{tabsslarge} 
			reports the computational results for
			all methods applied to
			a variant of 
			two-stage stochastic program \eqref{smdmodel11}-\eqref{smdmodel12}
			of size $n=50$
			where the feasible set of the first stage problem
			is replaced by the larger simplex set 
			$\{x_1 \in \mathbb{R}^n : x_1  \geq 0, \sum_{i=1}^n x_1(i)=100\}$.
			These results show 
			that  the SCPB variants are more efficient
			that E-SA
			on this specific instance. Also SCPB is not much affected
			by an overestimation of the diameter $D_X$.
			
			\begin{table}[H]
				\centering
				\begin{tabular}{|c|c|cc|cc|cc|}
					\cline{1-2}
					\multicolumn{2}{|c|}{-} & 
					\multicolumn{2}{|c|}{$M= \overline M_e$}&
					\multicolumn{2}{|c|}{$M=10\overline M_e$}&
					\multicolumn{2}{|c|}{$M=0.1 \overline M_e$}\\
					\hline
					ALG. & N & {\tt{Obj}} &{\tt{CPU}} &{\tt{Obj}} &{\tt{CPU}}&{\tt{Obj}} &{\tt{CPU}}\\
					\hline
					E-SA & 2000&  9.8&29.7&22.1&32.5&10.5&36.2\\
					\hline
					SCPB1 & 2000& 5.1 &36.3&5.2&33.8&5.1&42.2\\
					\hline
					SCPB2 & 2000&   5.5&31.2&4.6&33.5&5.6&37.2\\
					\hline
				\end{tabular}
				\caption{E-SA versus two variants of SCPB on the 
					two-stage stochastic program \eqref{smdmodel11}-\eqref{smdmodel12}
					with overestimated and underestimated values of $M$ on an instance
					with $n=50$.}
				\label{tabmlarge}
			\end{table}
			
			\begin{table}[H]
				\centering
				\begin{tabular}{|c|c|cc|cc|}
					\cline{1-2}
					\multicolumn{2}{|c|}{-} & 
					\multicolumn{2}{|c|}{$D= \overline D$}&
					\multicolumn{2}{|c|}{$D=5 \overline D$}\\
					\hline
					ALG. & N & {\tt{Obj}} &{\tt{CPU}} &{\tt{Obj}} &{\tt{CPU}}\\
					\hline
					\multirow[t]{2}{*}{E-SA} & 2000& $1.0338\small{\times}10^6$ &33.6&$1.0365\small{\times}10^6$  & 33.8\\
					& 10000& $8.894\small{\times}10^5$ &155.5 &$1.0055\small{\times}10^6$&160.8\\
					\hline
					SCPB1 & 2000& $2.894\small{\times}10^5$ &35.4&$3.029\small{\times}10^5$&35.6\\
					\hline
					SCPB2 & 2000& $2.889\small{\times}10^5$&34.2&$3.003\small{\times}10^5$&30.8\\
					\hline
				\end{tabular}
				\caption{E-SA versus SCPB1 and SCPB2 on a variant of the
					two-stage stochastic program \eqref{smdmodel11}-\eqref{smdmodel12}
					of size $n=50$
					where the simplex feasible set of the first stage problem
					is replaced by the larger feasible set 
					$\{x_1 \in \mathbb{R}^n : x_1  \geq 0, \sum_{i=1}^n x_1(i)=100\}$.
					Exact value $D=D_X=\overline D$ of the diameter used 
					to solve the first instance and overestimated value
					$D=5 \overline D=5D_X$ used to solve the second instance.
				}
				\label{tabsslarge}
			\end{table}

			Finally, we report the length of SCPB cycle along iterations
			in the left plot of Figure \ref{fig1clength}.
			A few comments are now in order on the length of the cycles
			with SCPB1 and SCPB2:
			\begin{itemize}
				\item  We observe that the length of the cycles is much larger with
				SCPB1. 
				\item For SCPB1, sequence $\{j_k\}$ (and therefore
				the length of the cycles) can be computed independently
				of sequence $\{x_k\}$, before running SCPB, once
				constant $R$ is known. It is worth
				mentioning that we have an analytic expression for $j_k$ as a function
				of $\lambda$, $R$, $\tau$, and $k$, namely
				$j_k-i_k=0$ if $R \geq \lambda k$ and
				$$
				j_k-i_k=\left \lceil{\frac{\log\left(\frac{R}{\lambda k}\right)}{\log\left(\tau\right)}}\right \rceil
				$$
				otherwise. Therefore,
				the cycle length with SCPB1
				is a piecewise constant 
				nondecreasing function of outer iteration $k$ and the cardinality of the set of consecutive iterations with constant cycle length increases along the cycles.
				\item For SCPB2, the length of the cycles is in general small 
				with small variability, with 
				an average cycle length of 2.2
				for the instance with $n=50$ 
				and an 
				average cycle
				length of 2.1 for the instance with $n=100$.
			\end{itemize}
			
			\begin{figure}[H]
				\centering
				\begin{tabular}{cc}
					\includegraphics[scale=0.6]{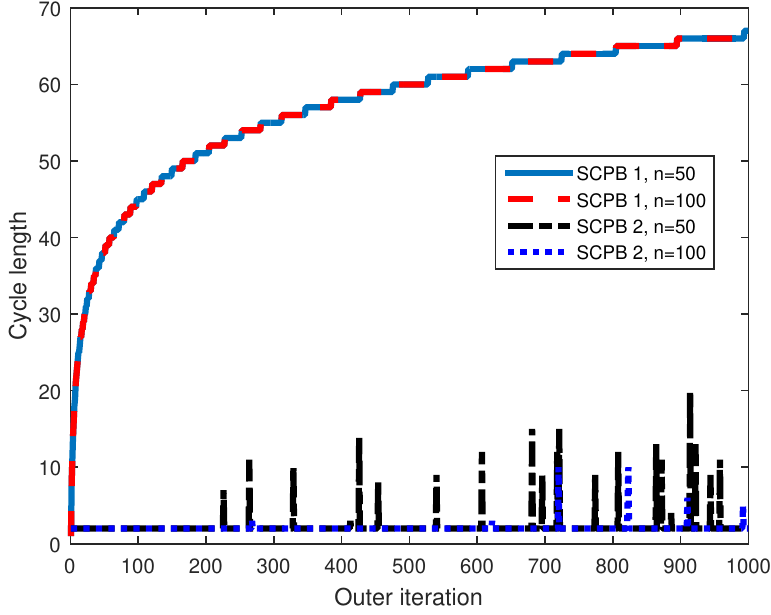}&
					\includegraphics[scale=0.6]{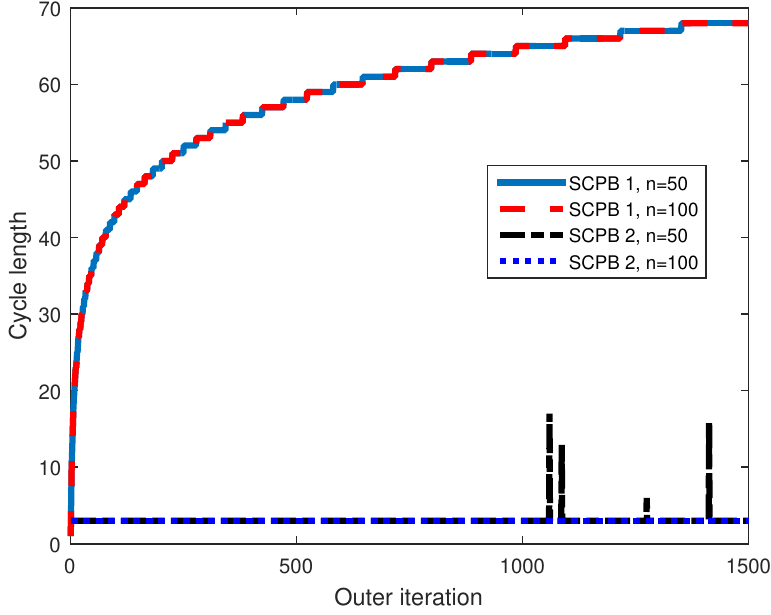}
				\end{tabular}
				\caption{Cycle length for SCPB1 and SCPB2 applied to  two-stage stochastic program 
					\eqref{smdmodel11}-\eqref{smdmodel12} (left figure) and two-stage stochastic program 
					\eqref{smdmodel21}-\eqref{smdmodel22} (right figure).}
				\label{fig1clength}
			\end{figure}
			
			\subsection{A second two-stage stochastic program}\label{subsec:sto2}
			
			Our third test problem is the nonlinear two-stage stochastic program 
			\begin{equation}\label{smdmodel21}
				\left\{
				\begin{array}{l}
					\min \; c^T x_1 + \mathbb{E}[\mathfrak{Q}(x_1,\xi)]\\
					x_1 \in \mathbb{R}^n : \|x_1-x_0\|_2 \leq 100 
				\end{array}
				\right. 
			\end{equation}
			where cost-to-go function $\mathfrak{Q}(x_1,\xi)$ has nonlinear objective and constraint coupling functions and is given by
			\begin{equation}\label{smdmodel22}
				\mathfrak{Q}(x_1,\xi)=\left\{ 
				\begin{array}{l}
					\displaystyle  \min_{x_2 \in \mathbb{R}^n} \;\frac{1}{2}\left( \begin{array}{c}x_1\\x_2\end{array} \right)^T \Big( \xi \xi^T + \gamma_0 I_{2 n} \Big) \left( \begin{array}{c}x_1\\x_2\end{array} \right) + \xi^T \left( \begin{array}{c}x_1\\x_2\end{array} \right)  \\
					s.t. \ \ \ \|x_2-y_0\|_2^2 + \|x_1-x_0\|_2^2 - R^2 \leq 0.
				\end{array}
				\right.
			\end{equation}
			Problem \eqref{smdmodel21}-\eqref{smdmodel22} is of form 
			\eqref{eq:ProbIntro}-\eqref{pbint2}
			where $F(x,\xi)=c^T x + \mathfrak{Q}(x,\xi)$
			with $\mathfrak{Q}$ given by \eqref{smdmodel22} 
			and where $h$ is the indicator function of
			set 
			$$
			X=\{x \in \mathbb{R}^n : \|x-x_0\|_2 \leq 100\}.
			$$
			For problem \eqref{smdmodel21}, we again refer to 
			Lemma 2.1 in \cite{guiguessiopt2016}
			for the computation of stochastic subgradients
			$s(x,\xi)$.
			We take $\gamma_0=2$
			and 
			consider for $\xi$ a Gaussian random vector in $\mathbb{R}^{2n}$
			with the components of $\xi$  independent with means and standard deviations 
			randomly generated in respectively intervals 
			$[-5,5]$ and $[0,10]$.
			The components of $c$ are generated randomly in interval $[-1,1]$
			and we take $R=200$, $x_0(i)=10$
			and $y_0(i)=1$ for
			$i=1,\ldots,n$.
			
			We run E-SA, SCPB1, and SCPB2 on two
			instances $B_1$ and $B_2$
			with $n=50$ and  $n=100$, respectively.
			For SCPB1 and SCPB2, we used $K=1500$ outer iterations.
			The results are reported in Tables \ref{tabthirdpb:cpuobj}
			and \ref{tabpb:cpuobjr}.
			The conclusions are similar to the experiments on
			the stochastic utility problem: it still takes much longer for E-SA to compute a solution with
			given accuracy. We also report in the right plot of Figure \ref{fig1clength}
			the evolution of the length of the cycles along outer iterations.
			The behavior of the length of these cycles is similar to what was observed 
			for the previous problem \eqref{smdmodel11}-\eqref{smdmodel12}.
			For SCPB2 the length of the cycles is still small on all iterations 
			and almost constant.
			%1.0953e+14  2.6423e+10 1.4516e+09 7.0698e+08  2.6360e+08 1.0984e+07 178
			%x 1000 6.1759    2.7736    0.0165   -0.0079   -0.0080   -0.0080
			
			\begin{table}[H]
				\centering
				\begin{tabular}{|c|c|cc|cc|}
					\hline
					\multicolumn{2}{|c|}{-} & 
					\multicolumn{2}{|c|}{$B_1: n=50$}&
					\multicolumn{2}{|c|}{$B_2: n=100$}\\
					\hline
					ALG. & N & {\tt{Obj}} &{\tt{CPU}} &{\tt{Obj}} &{\tt{CPU}}\\
					\hline
					\multirow[t]{6}{*}{E-SA} & 10& 15182 &0.2&18571&0.6\\
					%5 &15191 &0.1& 1.0953${\small{\times}}10^{14}$ &0.3\\
					%&10& 15182 &0.2&2.6423${\small{\times}}10^{10}$&0.6\\
					&50&  15108& 0.9 &18497&4.0\\
					&100& 15017 &0.9& 18405  &7.4\\
					&200&  14836& 1.8 &18222&13.9\\
					&1000& 13481& 3.4 &16830&63.8\\
					&$10^5$& 99.8  &16.2 &177.5& 6275\\
					\hline
					\multirow[t]{6}{*}{SCPB1} & 10&  2981.5&0.2  &4914 &0.8\\
					%5 & 3883.6 &0.1&6192.4&0.4\\
					%&10&  2981.5&0.2  & 4913.2&0.9\\
					&50&  761.2& 0.8  &1574&2.8\\
					&100& 288.1 &1.7&679&5.7\\
					&200& 64.8 &2.9&191&10.2\\
					&1000& -4.38 &14.3&-7.87&55.5\\
					\hline
					\multirow[t]{6}{*}{SCPB2} & 10&  1400.9& 0.13  &2766&0.5\\
					%5 & 3668.7 &0.06&6175.9&0.2\\
					%&10&  1400.9& 0.13  &2773.6&0.5\\
					&50& 0.45& 0.5 &17.5&2.1\\
					&100& -4.38 & 1.0  &-7.88&4.1\\
					&200& -4.38 & 1.9  &-7.95&8.1\\
					&1000 & -4.38 &9.0&-7.95&42.2\\
					\hline
				\end{tabular}
				\caption{E-SA versus two variants of SCPB on the 
					two-stage stochastic program \eqref{smdmodel21}, \eqref{smdmodel22}}
				\label{tabthirdpb:cpuobj}
			\end{table}
			
			%R=200;
			%Means=-5+10*rand(2*n,1);
			%Sds=10*rand(2*n,1);
			%cobj =-1+2*rand(n,1);
			%R0=100 in first stage problem
			%n=100 Pb2 
			%optrs
			%1.0953e+14  2.6423e+10 1.4516e+09 7.0698e+08  2.6360e+08 1.0984e+07 178.1
			%opts1
			%opts1
			%1000 x 6.1924    4.9132    1.5738    0.6785    0.1905   -0.0079
			%opts2
			%opts2
			%x1000 6.1759    2.7736    0.0165   -0.0079   -0.0080   -0.0080
			%times rs s1 s2
			%timess1
			%0.3556
			%    0.8671
			%    3.0636
			%    6.0338
			%   11.1968
			%   62.2367
			%   
			%timess2

			\subsection{Summarizing performance indicators}
			
			The computational results reported in Subsections \ref{subsec:utility}-\ref{subsec:sto2} show that SCPB computes
			a good approximate solution quicker than E-SA.
			To properly quantify the speed-up over a fixed number of iterations $N$, we compute
			the quantity
			% we report in Table \ref{tabpb:cpuobjr}
			% the distance between
			% SCPB approximate optimal value 
			% and
			% E-SA approximate value
			% as a percentage of
			% SCPB decrease in the objective
			% for several sample sizes $N$.
			% This percentage is given by
			\begin{equation}\label{eq:percentage}
				100
				\frac{{\tt{Obj}}\mbox{(E-SA)}-{\tt{Obj}}(\mbox{SCPBi})}{
					{\hat F}_T(x_0) - {\tt{Obj}}\mbox{(SCPBi})}
			\end{equation}
			associated with SCPBi,
			where ${\hat F}_T(x_0)$
			is the empirical
			mean (see \eqref{defempmean}) of 
			$F$ with $T=10^4$ and $x$ equal to the initial point $x_0$, and
			${\tt{Obj}}\mbox{(E-SA)}$,
			${\tt{Obj}}\mbox{(SCPB1)}$, and ${\tt{Obj}}(\mbox{SCPB2})$
			are
			the empirical
			means of 
			$F$ with $T=10^4$ and $x$ equal to the final iterates output by  E-SA, SCPB1, and SCPB2, respectively.
			We see that after $N=1000$ iterations this percentage is 
			above $90\%$ for most instances,
			which clearly shows
			that
			both variants of SCPB
			are faster than E-SA.
			
			\begin{table}[H]
				\centering
				\begin{tabular}{|l|c|c|c|c|c|c|}
					\hline
					Sample size $N$ &10 &50& 100 &200 &1000\\
					\hline
					%$L_1$, SCPB1 & 95.2&96.3 &96.5& 96.5 & 94.4\\
					%\hline
					%$L_1$, SCPB2&  94.7&97.3  &97.1 &96.8&94.4\\
					%\hline
					%$L_2$, SCPB1 & 95.5&96.5& 97.0 &97.1 &95.1\\
					%\hline
					%$L_2$, SCPB2 & 97.6 &98.0&97.9  &97.5 &95.1\\
					%\hline
					$L_1$, SCPB1 & 95.0 &95.2&94.1  &91.9 &82.8\\
					\hline
					$L_1$, SCPB2 & 96.6 &97.2  & 97.5 &97.2 &90.9\\
					\hline
					$L_2$, SCPB1 &92.1 &93.3&92.3  &91.5 &77.7\\
					\hline
					$L_2$, SCPB2 & 92.1&95.8  &96.8 &96.7&93.5\\
					\hline
					$A_1$, SCPB1 &99.5 &98.8& 98.1 &96.6 &85.1\\
					\hline
					$A_1$, SCPB2 & 99.6&99.0&98.0  &96.5 &84.2\\
					\hline
					$A_2$, SCPB1 &99.8 &99.6&99.3  &98.8 &95.0\\
					\hline
					$A_2$, SCPB2 &99.8 &99.6&99.3  &98.8 &95.4\\
					\hline
					$B_1$, SCPB1 &99.8 &99.4&98.8  & 97.6&88.7\\
					\hline
					$B_1$, SCPB2 &99.9 &99.4& 98.8 &97.6 &88.7\\
					\hline
					$B_2$, SCPB1 &99.9 &99.4&99.0  & 98.0&90.5\\
					\hline
					$B_2$, SCPB2 &99.9 &99.5& 99.0 &98.0 &90.5\\
					\hline
				\end{tabular}
				\caption{Percentages \eqref{eq:percentage} for SCPB1 and SCPB2}
				% Distance between
				% SCPB approximate optimal value 
				% and
				% E-SA approximate value
				% as a percentage of
				% SCPB decrease in the objective
				% for several sample sizes $N$.}
			\label{tabpb:cpuobjr}
		\end{table}

		\section{Concluding remarks}\label{sec:remarks}
		
		% This section summarizes the analysis developed in this paper and compares
		% it with other related papers.

		This paper proposes two single-cut stochastic
		composite proximal bundle variants, called SCPB,
		for solving SCCO problem \eqref{eq:ProbIntro}-\eqref{pbint2} where at each iteration
		a problem of form 
		\eqref{eq:x-pre}  is solved.
		The two SCPB variants, which differ in the way their cycle lengths are determined,
		are analyzed in
		Sections \ref{sec:SCPB1} and \ref{sec:SCPB2}, respectively.
		% 	The convergence rate bounds
		%     for the two corresponding SCPB variants are given
		%     in Theorems \ref{thm:main1} and 
		%     \ref{thm:main2} under standard assumptions.
		% , in particular uniformly bounded subgradients in expectation.
		More specifically,
		it is shown that both variants of SCPB with properly chosen parameters have optimal iteration complexity (up to a logarithmic term) for finding an $\varepsilon$-solution of \eqref{eq:ProbIntro} for a large range of prox stepsizes.
		Practical variants of SCPB 
		which keep their cycle lengths bounded  are also proposed
		and numerical experiments
		demonstrating their excellent performance against the RSA method of \cite{nemjudlannem09} on the
		%  (relatively small number of)
		instances considered in this paper are also reported.
		% is competitive
		% with SMD and can even outperform SMD
		% on some problem instances. 
		
		{\bf Comparison with other methods:}
		First, we have shown in Subsection~\ref{subsec:relation} that RSA is a special case of SCPB1  which performs only one iteration per cycle.
		Second, it is worth noting that SCPB has a slight similarity with the stochastic dual averaging (SDA) method discussed in %Section 6 of 
		\cite{nesterov2009primal,JMLR:v11:xiao10a} since both methods explore the idea of aggregating cuts into a single one.
		However, there are essential differences between the two methods, namely:
		1) while SCPB updates the prox-centers whenever a serious iteration occurs,
		SDA uses a fixed prox-center, and hence only performs null iterations; 
		% 2) SDA is developed for the case where the composite function $h$ is the indicator function of a closed
		% (not necessarily bounded) convex set while SCPB assumes $h$ is Lipschitz continuous on its domain;
		and
		2) SDA uses variable stepsizes which have to grow sufficiently large, while SCPB uses constant prox stepsizes.
		In summary, from the viewpoint of this paper, SDA is closest to the special case of
		SCPB 
		with a single cycle and
		a sufficiently large prox stepsize;  the difference between  the latter two methods is that
		SDA allows
		the prox stepsizes
		within its single cycle
		to gradually become sufficiently
		large.
		
		% \red{remove} Second, RSA (Robust Stochastic Approximation)
		%  ??????
		%  is a special case
		%  of 
		%  SCPB when $C$ and $\lam$ are chosen as $C=\alpha \sqrt{K}$
		%  and $\lam = \alpha /\sqrt{K}$,
		%  and  cycle rule (B1)
		%  is used to determine
		%  the size of a cycle.
		%  Indeed, it follows from the above choices of $C$ and $\lam$ that
		%  \[
		%  \frac{C}{\lam k} \ge \frac{C}{\lam K} = 1
		%  \]
		%  and hence that $j_k=i_k$ satisfies (B1).
		%  Hence, every cycle only performs 
		%  one iteration, i.e., its only
		%  serious iteration. Hence, every iteration of this SCPB variant is a serious one and $K$ is its total number of iterations.
		%  It is worth noting that
		%  the convergence rate bound \eqref{ineq:consequence} reduces to
		%  \begin{equation}\label{eq:rate}
			%      \E[\phi(\hat y_K^a)] - \phi_* \le \E[\hat u_K^a] - \phi_* \le \frac1{\sqrt{K}} \left(\frac{ d_0^2}{\alpha} + 6 \alpha M^2 +  \frac{2\alpha M^2}{\theta K}\right)
			%      \le \frac1{\sqrt{K}} \left(\frac{ d_0^2}{\alpha} + (6+2\theta^{-1}) \alpha M^2 \right)
			%  \end{equation}
		%  which
		% %  , disregarding the (negligible) last term
		% %  on its right hand side,
		% %  becomes 
		%  is similar (up to constants) to the one obtained
		%  for RSA with $\lam=\alpha/\sqrt{K}$.
		%  Note also that if an upper bound $D$ on $d_0$ is known and $\alpha$ is chosen as
		%  $\alpha=D/M$ then \eqref{eq:rate}
		%  becomes
		%  ${\cal O}(MD/\sqrt{K})$.
		
		In summary,
		while RSA (resp., SDA) performs only serious (resp., null) iterations,
		SCPB performs a balanced mix of serious and null iterations.
		Hence,  it is reasonable to conclude that SCPB lies between RSA and SDA.

		% 4) \cite{nesterov2009primal} considers the special
		% case where $h$ is the indicator function
		% of a convex set.

		\par {\textbf{Extensions.}}	We finally discuss some possible extensions of our analysis in this paper. A first question
		is how to extend SCPB  and the corresponding
		complexity analysis if
		instead of Assumption \ref{assump:basic} 
		(A3) we use the assumption that
		for every $u,v \in \dom h$, we have
		\[
		\|f'(u)-f'(v)\| \le 2M + L \|u-v\|,
		\]	
		which is called a uniform $(M,L)$-condition in \cite{liang2021unified}.
		A second natural question is how
		to extend SCPB and its complexity
		analysis when either the prox stepsize $\lam$ and parameter $\theta$ are allowed to change with the iteration count $k$.
		Recalling that the prox stepsize is the only ingredient of the second variant of SCPB that depends
		on an estimate $M$ of $\bar M$,
		a third natural question is
		whether it is possible to develop a SCPB variant which adaptively chooses a (variable)
		prox stepsize without the need
		of knowing $M$.
		% adaptive weights $\tau_k$
		% (instead of a fixed weight $\tau$)
		% and stepsizes $\lambda_k$ (instead
		% of a fixed stepsize $\lambda$)
		% are used for respectively the weights
		% defining the single cut and the prox
		% stepsizes. 
		A fourth question is whether it is possible to establish global convergence rate guarantees for proximal bundle methods based on
		two-cut or multiple-cut bundle models instead of the
		single-cut bundle models 
		considered in this paper.
		Finally, SCPB
		is able to
		solve two-stage convex stochastic programs with continuous distributions, under the assumption that the second-stage subproblems can be exactly solved (e.g., see Subsections~\ref{sec:sto1} and \ref{subsec:sto2}). It would be  interesting to extend it to the setting
		of {\it multistage} stochastic convex problems with continuous
		distributions.

		\bibliographystyle{plain}
		\bibliography{Proxacc_ref}

				\end{document}